\documentclass[12pt,reqno]{amsart}
\usepackage[utf8]{inputenc}
\usepackage[T1]{fontenc}
\usepackage[english]{babel}
\usepackage{mathrsfs, csquotes}
\usepackage{amsfonts,amssymb,amsmath,mathtools,commath,braket,amssymb}
\usepackage{color} \definecolor{bleu_sombre}{rgb}{0,0,0.6}  \definecolor{rouge_sombre}{rgb}{0.8,0,0}\definecolor{vert_sombre}{rgb}{0,0.6,0}
\usepackage[plainpages=false,colorlinks,linkcolor=bleu_sombre,
citecolor=rouge_sombre,urlcolor=vert_sombre,breaklinks]{hyperref}
\usepackage{enumerate}

\usepackage{graphicx} 
\usepackage[font=small,labelfont=bf]{caption} 
\usepackage[subrefformat=parens]{subcaption}

\theoremstyle{plain}
\newtheorem{theorem}{Theorem}[section]
\newtheorem{lemma}[theorem]{Lemma}
\newtheorem{corollary}[theorem]{Corollary}
\newtheorem{proposition}[theorem]{Proposition}
\newtheorem{hyp}[theorem]{Assumption}

\theoremstyle{definition}
\newtheorem{remark}[theorem]{Remark}
\newtheorem{notation}{Notation}

\newtheorem{definition}[theorem]{Definition}
\newtheorem{assumption}[theorem]{Assumption}

\newcommand{\Bk}{\color{black}}

\newcommand{\ck}{C_k(B,\Omega)}
\newcommand{\x}{\mathbf{x}}

\renewcommand{\leq}{\leqslant}	\renewcommand{\geq}{\geqslant}
\renewcommand{\Re}{\mathrm{Re}}

\def\CC{{\mathbb C}}
\def\RR{{\mathbb R}}
\def\R{{\mathbb R}}
\def\NN{{\mathbb N}}
\def\ZZ{{\mathbb Z}}

\def\({\left(}
\def\){\right)}
\def\<{\left\langle}
\def\>{\right\rangle}
\def\le{\leqslant}
\def\ge{\geqslant}

\def\eps{\varepsilon}

\DeclareMathOperator{\RE}{Re}

\DeclareMathOperator{\p}{\mathbf{p}}

\numberwithin{equation}{section}

\newcommand{\dd}{\mathrm{d}}

\newcommand{\be}{\begin{equation}}
\newcommand{\ee}{\end{equation}}
\newcommand{\bea}{\begin{eqnarray}}
\newcommand{\eea}{\end{eqnarray}}
\newcommand{\bee}{\begin{eqnarray*}}
\newcommand{\eee}{\end{eqnarray*}}

\def\eps{\varepsilon}

\def\pa{\partial}

\def\Nb{{N_\mathcal{B}}}

\def\disth{\mathsf{dist}_{\mathcal{H}}}

\def\distb{\mathsf{dist}_{\mathcal{B}}}
\def\DMh{\mathscr{D}_{h,\mathbf{A}}}
\def\dMh{d_{h,\mathbf{A}}}
\def\dMhX{d_{h,\mathbf{A}}^\times}
\def\HtwohA{\mathscr{H}^2_{h,\mathbf{A}}(\Omega)}
\def\HhA{\mathfrak{H}_{h,\mathbf{A} }}
\def\HhmA{\mathfrak{H}_{h,-\mathbf{A} }}
\def\dMmh{d_{h,-\mathbf{A}}}
\def\dMhmX{d_{h,-\mathbf{A}}^\times}
\def\HtwohmA{\mathscr{H}^2_{h,-\mathbf{A}}(\Omega)}
\def\dMmh{d_{h,-\mathbf{A}}}
\begin{document}

\title[Dirac bag model in strong magnetic fields]{The {D}irac bag model in strong magnetic fields}
\author[J.-M. Barbaroux]{J.-M. Barbaroux}
\email{barbarou@univ-tln.fr}
\address[J.-M. Barbaroux]{Aix Marseille Univ, Universit\'e de Toulon, CNRS, CPT, Marseille, France}
\author[L. Le Treust]{L. Le Treust}
\email{loic.le-treust@univ-amu.fr}
\address[L. Le Treust]{Aix Marseille Univ, CNRS, Centrale Marseille, I2M, Marseille, France}
\author[N. Raymond]{N. Raymond}
\email{nicolas.raymond@univ-angers.fr }
\address[N. Raymond]{Laboratoire Angevin de Recherche en Mathématiques, LAREMA, UMR 6093, UNIV Angers, SFR Math-STIC, 2 boulevard Lavoisier 49045 Angers Cedex 01, France}
\author[E. Stockmeyer]{E. Stockmeyer}
\email{stock@fis.puc.cl}
\address[E. Stockmeyer]{Instituto de F\'isica, Pontificia Universidad Cat\'olica de Chile, Vicu\~na Mackenna 4860, Santiago 7820436, Chile.}

\maketitle
\begin{abstract}
In this work we study  Dirac operators  on two-dimensional domains coupled to a magnetic field perpendicular to the plane. We focus on the infinite-mass boundary condition (also called MIT bag condition).  In the case of bounded   domains,  we establish the asymptotic behavior of the low-lying (positive and negative) energies  in the limit of strong magnetic field. Moreover, for a constant magnetic field $B$, we study the problem on the half-plane and find  that the Dirac operator has continuous spectrum except for a gap of size  $a_0\sqrt{B}$, where $a_0\in (0,\sqrt{2})$ is a universal constant. Remarkably, this constant characterizes certain energies of the system in a bounded domain as well.  We discuss how these findings, together with our previous work, give a fairly complete description  of the eigenvalue asymptotics of magnetic  two-dimensional Dirac operators under general boundary conditions.

\end{abstract}

\tableofcontents

\section{Introduction}
Consider an open, smooth and simply connected domain $\Omega\subset \RR^2$ and a magnetic field $\mathbf{B}=B \mathbf{\hat{z}}$, smooth and pointing in direction $\mathbf{\hat{z}}$ orthogonal to the plane.
In this work we consider a  Dirac  operator restricted to  $\Omega$ and  coupled to the magnetic field $\mathbf{B}$ through a magnetic vector  potential $\mathbf{A} = (A_1,A_2)^T$ satisfying $\nabla\times\mathbf{A}=\mathbf{B}$. The magnetic Dirac  operator acts on a dense subspace
of  $L^2(\Omega,\CC^2)$ as,   
\begin{equation}\label{eq.Dirac1}
\sigma\cdot (-ih\nabla-  \mathbf{A})=\begin{pmatrix}0&-ih (\partial_1-i\partial_2)-A_1+iA_2\\
-ih(\partial_1+i\partial_2)-A_1-iA_2&0
\end{pmatrix}\,, 
\end{equation}
where $h>0$ is the semiclassical parameter.
We write  
$\sigma\cdot \mathbf{x} = \sigma_1 x_1 + \sigma_2 x_2$ for $\mathbf{x} = (x_1,x_2)$ with the usual Pauli matrices
$$
\sigma_1 = \left(
\begin{array}{cc}
0&1\\1&0
\end{array}
\right),
\quad
\sigma_2 = \left(
\begin{array}{cc}
0&-i\\i&0
\end{array}
\right),
\quad
\sigma_3 = \left(
\begin{array}{cc}
1&0\\0&-1
\end{array}
\right)\, .
$$
%

%
If we assume that the spinors satisfy a boundary relation of the type $\varphi={\mathcal{T}}\varphi$ on $\partial \Omega$ with a unitary and self-adjoint boundary matrix ${\mathcal{T}}\colon \partial \Omega\to \CC^{2\times 2}$, then simple integration by parts shows that the local current density $\langle\varphi,\sigma\cdot\mathbf{n}\varphi\rangle_{\CC^2}$ vanishes at each point of the  boundary  if and only if 
\begin{align}
	{\mathcal{T}}\,\sigma\cdot\mathbf{n}+\sigma\cdot\mathbf{n}\, {\mathcal{T}}=0\quad \mbox{on}\quad \partial\Omega\,,
\end{align}
where $\mathbf{n}$ is the  normal vector  pointing outward to the boundary and $\langle\cdot,\cdot\rangle_{\CC^2}$ is the standard scalar product on $\CC^2$ (antilinear w.r.t. the left argument). 
In particular,  for these cases, the Dirac operator is formally symmetric and satisfies the bag condition, i.e.,  that no current flows through $\partial\Omega$  \cite{Berry1987}. In the physics literature these types of models have been earlier considered  to describe neutrino billards  \cite{Berry1987} and (in the three dimensional setting) quark confinement \cite{chodos1974baryon}. More recently,  they have regained  attention with the advent of graphene and other Dirac materials, see e.g.,  \cite{akhmerov2008boundary,castro2009electronic,shtanko2018robustness, MR2947949}. 

Using the properties of the Pauli matrices and those of ${\mathcal{T}}$ it is easy to see that the most general form  of ${\mathcal{T}}$ acts as a multiplication on $L^2(\partial \Omega)$ with
\begin{align}\label{beta}
	{\mathcal{T}}\equiv {\mathcal{T}}_\eta=({\sigma\cdot \mathbf{t}}) \cos\eta +\sigma_3 \sin\eta \,, 
\end{align} 
for certain sufficiently smooth $\eta:\partial \Omega\to \RR$ and $\mathbf{t}$ being  the unit tangent vector pointing clockwise (we have that $\mathbf{n}\times \mathbf{t}=\mathbf{\hat{z}}$). 
The most frequently used  boundary conditions  in the physics literature 
are the cases when $\cos\eta=0$ and $\sin\eta=0$ known as zigzag and infinite-mass boundary conditions, respectively.
 For recent mathematical literature on the subject in the two and three dimensional settings see for instance \cite{MR3626307,arrizabalaga:hal-01540149, le2018self, pizzichillo2019self, behrndt2019self} about self-adjointness, \cite{VS2019, ALTMR19, 2019resolvent} for the derivations as an infinite mass limit, and  \cite{MR3625007, MR3936982, ABLO-B20} for eigenvalue estimates.

In this work we consider  Dirac operators $D_\eta$ acting as in \eqref{eq.Dirac1} on spinors $\varphi$  satisfying $\varphi=\mathcal{T}_\eta \varphi$, with  $\eta\in [0,2\pi) $. 
We give the precise definition of the self-adjoint realization below.
Assuming that the magnetic field satisfies   $\inf_{x\in \overline{\Omega}} B(x)=b_0>0$ (besides certain geometrical conditions, see Assumption~\ref{hyp.regphi}), we provide asymptotic estimates for  the corresponding low-lying eigenvalues in the semiclassical limit $h\to 0$.  

The behavior of the corresponding operators in   the  physically most relevant  cases mentioned above are quite different  from each other. Indeed,  on the one hand,   
the spectrum of a zigzag operator is symmetric  with respect to zero and zero is an eigenvalue of infinite multiplicity. On the other hand, the spectrum in the case of infinite-mass boundary conditions 
is far from being symmetric in the semiclassical limit $h\to 0$ and zero is never in its spectrum. \\
Our main results can be roughly summarized as follows: \\
\\
Let $\Omega\subset \RR^2$ be bounded. For  $k \in \{1,2,3,\dots\}$ we denote by $\lambda^+_k(h)\ge 0$ and $-\lambda^-_k(h)<0$ the non-negative and negative eigenvalues of $D_0$, the MIT bag operator with $\eta=0$. They are ordered such that  $\lambda^{\pm}_k(h)\le 
\lambda^{\pm}_{k+1}(h)$. Then, there is a constant 
$C_k^+>0$ such that, as $h\to 0$, 
\begin{align}\label{alf}
	\lambda^+_k(h)=C_k^+ h^{1-k} e^{-2\phi_0/h}(1+o(1))\,.
\end{align}
We provide explicit expressions for the constants $C_k^+$ (see \eqref{ck}) and $\phi_0>0$ in terms of the geometry  and the magnetic field $B$ (see Theorem~\ref{thm.main}). In particular, the positive eigenvalues  of $D_0$ accumulate exponentially fast to zero in the semiclassical limit. If we consider that the square of the Dirac operator acts as a magnetic Schrödinger operator (the Pauli operator), the behavior of the eigenvalues \eqref{alf} is a surprising fact. Indeed the Dirac energies in this case scale in the same way as the ones from the Dirichlet-Pauli operator (see \cite{BLTRS20a}, where, in addition, the one term asymptotics is still an open problem in the general case). Moreover, the corresponding constants $C_k^+$, for Dirichlet-Pauli, have similar expressions -- based on the same functional spaces of holomorphic functions (but equipped with different boundary norms).

This behavior is in contrast to the one of the negative eigenvalues. Indeed, for the first negative eigenvalue
we show that there is a constant $C^->0$ such that  
\begin{align}\label{alf2}
	\lambda^-_1(h)=C^-  h^{\frac 12}(1+o(1))\,.
\end{align}
The constant 
$C^-$ obeys an effective minimization problem 
(see Theorem~\ref{thm.main'}) and it is related to a corresponding problem on the half-plane. Moreover, when the magnetic field is constant on the domain, we  describe the fine structure  of the first negative eigenvalues which are at a distance of order $h^{\frac32}$ to $\lambda^-_1(h)$. This is done  by means of an effective operator obtained by using microlocal techniques (see Theorem \ref{thm.main2}). In particular, we compute the exact fine structure in the case when $\Omega$ is the disk. 

Furthermore,  we study the half-plane problem  (when $\Omega=\RR\times \RR_+$). We find that the spectrum of $D_0$ is absolutely continuous and that it has a spectral gap of size $a_0\approx 1.312<\sqrt{2}$ (see Theorem \ref{thm.homogeneousOp}). In Section \ref{sec.5}, we will see that, in the case $B=1$, the constant $C^-$ from \eqref{alf2} equals $a_0$. This indicates an exceptional stability of $\lambda^-_1(h)$ to leading order under perturbations of the boundary when $h\to 0$. This contrasts with the case of the positive energies where the constants depend on the global magnetic geometry.

The proofs of  the above are based upon the asymptotic analysis of a min-max principle for the corresponding operator $D_0$. 
We show a min-max characterization, well adapted to our setting, whose proof is inspired by the pioneer works \cite{DES00} and \cite{GS99} (see Theorem \ref{mmax} and Remark \ref{chargec}).  A similar min-max characterization has recently been obtained in a non-magnetic framework for the grounstate energy, see \cite{ABLO-B20}. 
It is easy to see that our min-max characterization applies well to any boundary conditions with $\cos\eta\not=0$. This is  described in Appendix~\ref{app.A} where  we obtain the same type of asymptotic formulas \eqref{alf} and \eqref{alf2} with different constants.

Finally let us discuss  the   zigzag case, when $\cos\eta=0$. We obtain analogous results for the energies through a  simple application of   the asymptotic analysis  performed in \cite{BLTRS20a} and the relation between  zigzag  and Pauli-Dirichlet operators. This is explained in Section \ref{rk.zigzag} and the results can  be summarized as follows:  
For $k\in\{1,2,3,\dots\}$ we denote by $\alpha_k^{+}(h)$ and $\alpha_k^{-}(h)$ the $k$-th positive eigenvalue of $D_{\pi/2}$ and $D_{3\pi/2}$, respectively\footnote{Note that the $\pm$ are not related to the signs of the eigenvalues in this case, but rather to the spin of the system at the boundary.}. Then, we find  constants $0<c_k\le C_k<\infty$ that,  as $h\to 0$,
\begin{align*}
	c_k h^{\frac{1-k}{2}} e^{-\phi_0/h }(1+o(1)) \le  \alpha_k^{-}(h) \le C_k
	h^{\frac{1-k}{2}} e^{-\phi_0/h } 
	(1+o(1))\,,
\end{align*}
and 
\begin{align*}
	\alpha_k^{+}(h) \ge \sqrt{2b_0 h}\,,
\end{align*}
where $\phi_0>0$ is the same constant appearing in \eqref{alf}. 
\begin{remark}
	The semiclassical limit ($h\to 0$) corresponds to the strong magnetic field limit with fixed $h$. Indeed, a simple scaling argument shows that the energies of the problem for $h>0$ fixed and  a magnetic field $tB$, $t>0$, are given by $t\lambda_k^\pm(h/t)$  and $t\alpha_k^\pm(h/t)$, for infinite-mass and zigzag boundary conditions, respectively.
\end{remark}

It is well known that at low energies  the dynamics of charge carriers in graphene is effectively described by two copies of a massless Dirac  operator as described in \eqref{eq.Dirac1} (see e.g. \cite{castro2009electronic,doi:10.1080/00018732.2010.487978}). The two copies correspond to the so-called valley degrees of freedom. In the presence of a uniform 
magnetic field, the \enquote{unusual} Landau levels of the  Dirac operator,
$$
{\rm sgn} (n)\sqrt{2hB|n|}, \quad n\in \ZZ\,,
$$
have been experimentally observed in  extended graphene \cite{zhang2005experimental,deacon2007cyclotron}. The relation between specific cuts in the  graphene honeycomb structure and the boundary conditions of the corresponding Dirac operator have been studied, see e.g. \cite{akhmerov2008boundary}. The most important type of cuts corresponds to the so-called armchair and zigzag boundary conditions, which are called so due to the shape left in the honeycomb lattice.  
It has been shown that very  strong effective magnetic fields, of variable size, can be generated by mechanical strain in graphene \cite{guinea2010energy}. Dirac operators with uniform magnetic fields, with zigzag and infinite-mass boundary conditions, have being considered in the physics literature before. Based on analytic and partly numerical methods, the energies of the system have been investigated for  instance in   \cite{phys4,grujic2011electronic}  for disks,
and for rings and circular holes on the plane in \cite{thomsen2017analytical}.

Our results compare qualitatively well with the findings in  \cite{phys4,grujic2011electronic} for homogeneous magnetic fields. However, 
to the best of our knowledge, the spectral gap appearing in  the half-plane problem and  characterized by $a_0$, has not been reported before.
In particular, our results show that for strong homogeneous  magnetic fields with infinite-mass boundary conditions, there is 
a persistent gap between positive and negative energies, remarkably stable in the geometry,  of size $a_0\sqrt{hB}$, modulo an error that goes like the inverse square root of the magnetic field. This can be seen by comparing positive and negative energies from Theorems \ref{thm.main} and \ref{thm.main2}, in the proper large magnetic field scaling. Finally, let us mention that our analysis also reveals that the eigenfunctions associated with the energies around $-a_0\sqrt{hB}$ 
are concentrated at the boundary. This indicates 
the existence of edge states that might have further interesting properties.

\subsection{Basic definitions and assumptions}\label{sec.basdef}
We  study  the semiclassical problem given by 
the action of 
\begin{equation}\label{eq.Dirac}
\DMh=\sigma\cdot (\mathbf{p}-\mathbf{A})=\begin{pmatrix}0&\dMh\\
\dMhX&0
\end{pmatrix}\,, 
\end{equation}
where $\mathbf{p} = -ih\nabla$ for $h>0$,
\[\dMh=-2ih\partial_{z}-A_1+iA_2\,,\quad
\dMhX=-2ih\partial_{\overline{z}}-A_1-iA_2\,,\]
with $\partial_{\overline{z}} = \frac{\partial_1+i\partial_2}{2}$ and $\partial_{z} = \frac{\partial_1-i\partial_2}{2}$.
We focus on the boundary conditions described above for $\eta=0$, that is
\[
\mathcal{T}=\sigma\cdot\mathbf{t}=-i\sigma_3(\sigma\cdot\mathbf{n})\,,
\] 
where $\mathbf{n}$ is the outward pointing normal to the boundary $\partial \Omega$. 
The associated magnetic Dirac operator with infinite-mass boundary condition is $(\DMh, \mathsf{Dom}(\DMh))$ with
\[
\mathsf{Dom}(\DMh) = \left\{
\varphi\in  H^1(\Omega, \CC^2)\,,\quad \mathcal{T} \varphi=\varphi\text{ on }\partial \Omega
\right\}\,.
\]
\begin{remark}
	Note that 
	\[\sigma\cdot\mathbf{n}=\begin{pmatrix}0&\overline{\mathbf{n}}\\
	{\mathbf{n}}&0\end{pmatrix}\,,\]
	so that
	the boundary condition reads
	\[v=i\mathbf{n}u\,,\]
	where $\varphi = (u,v)^T$, and $\mathbf{n} = (n_1,n_2)^T$ denotes the normal vector in $\RR^2$ and also $\mathbf{n} = n_1+in_2\in \CC$.
\begin{notation} 
	We denote by $\langle \cdot,\cdot \rangle$ the standard $L^2$-scalar product (antilinear w.r.t. the left argument) on $\Omega$ and by $\|\cdot\|$ the associated norm. In the same way, we denote by $\langle\cdot,\cdot\rangle_{\partial \Omega}$ the $L^2$-scalar product on $L^2(\partial\Omega)$. 
\end{notation}
	
\end{remark}
The main purpose of our paper is to study the asymptotic behavior of the eigenvalues near $0$ in the semiclassical limit $h\to 0$.
%
%
\begin{hyp}\label{hyp.reg}
~
	\begin{enumerate}[\rm (i)]
		\item
		$\Omega$ is bounded, simply connected, $\partial \Omega$ is $\mathcal{C}^2$-regular,
		\item
		$B\in W^{1,\infty}(\overline{\Omega})$ . 
	\end{enumerate} 
\end{hyp}
Under Assumption \ref{hyp.reg}, the operator $\mathscr{D}_{h,0}$, without magnetic field, is self-adjoint on $L^2(\Omega)^2$ (see for instance \cite{MR3626307}). 
We work in the so-called {\it Coulomb gauge} that is given through the unique solution of 
the Poisson equation
\begin{equation}\label{gauge1}
\Delta\phi=B\,,\qquad \phi_{|\partial\Omega}=0\,,
\end{equation}
by choosing 
$\mathbf{A}= (-\pa_2 \phi, \pa_1 \phi)^T=\nabla\phi^\perp$.
Notice that by standard regularity theory the components of $\mathbf{A}$ are bounded. Hence   $\DMh$ is self-adjoint and it has compact resolvent since 
$\mathsf{Dom}(\DMh) \subset H^1$. In particular, the spectrum $\DMh$ of is discrete.
We denote by $(\lambda^+_k(h))_{k\geq 1}$ and $(-\lambda^-_k(h))_{k\geq 1}$ the positive and negative eigenvalues of $\DMh$ counted with multiplicities. In fact, $\DMh$ has no zero modes. This can be seen using the following lemma, which is a consequence of \cite{HP17} and \cite{BLTRS20a}.

\begin{lemma}\label{lem.lbPD}
	For all $h>0$, there exists $C(h)>0$ such that, for all $u\in H^1_0(\Omega)$, we have
	\[\|\dMhX u\|^2\geq C(h)\|u\|^2\,,\quad \|\dMh u\|^2\geq C(h)\|u\|^2  \,.\]	
\end{lemma}

\begin{proposition}\label{lem.zeromodes}
	The operator $\DMh$ has no zero modes.
\end{proposition}
\begin{proof}
	Consider $\varphi = (u,v)^T\in\mathrm{Dom}(\DMh{})$ such that $\DMh{} \varphi=0$. We have $\dMh v=\dMhX u=0$. Thus, integrating by parts, and using the boundary condition, we get
	\[0=\langle \dMh v,u\rangle=\langle v,\dMhX u\rangle+h\langle -i\overline{\mathbf{n}}\,v,u\rangle_{\partial\Omega}=h\|u\|^2_{\partial\Omega}\,.\]
	Therefore $u,v\in H^1_0(\Omega)$, and Lemma \ref{lem.lbPD} implies that $u=v=0$.	
\end{proof}

Since $\DMh$ has no zero mode, its spectrum is
\begin{equation}\label{eq.spectrum}
\mathrm{sp}(\DMh) = \{\dots\,,\ -\lambda^-_2(h)\,,\ -\lambda^-_1(h)\}\cup\{\lambda^+_1(h)\,, \ \lambda^+_2(h)\,, \dots\}\,.
\end{equation}
\begin{hyp}\label{hyp.posB}
	$B$ is positive. We define $b_0 = \inf_\Omega B>0$ and $b'_0=\min_{\partial\Omega} B$.
\end{hyp}
Under this assumption, $\phi$ is subharmonic so that \[\max_{x\in\overline{\Omega}}\phi=\max_{x\in\partial\Omega}\phi=0\,,\]
and the minimum of $\phi$ will be negative and attained in $\Omega$.
\begin{hyp}\label{hyp.regphi}
	~
	\begin{enumerate}[\rm (i)]
		\item The minimum $\phi_{\rm min}$ of $\phi$ is  attained at a unique point $x_{\rm min}$. 
		\item The Hessian matrix $\mathsf{Hess}_{\rm min}\phi$ of $\phi$ at $x_{\rm min}$ is positive definite i.e. $x_{\rm min}$ is a non-degenerate minimum. We also denote by $z_{\rm min}$, the minimum $x_{\rm min}$ seen as a complex number.
	\end{enumerate}
\end{hyp}

\subsection{Main results}
\subsubsection{A min-max characterization of the eigenvalues}
Our first result gives a non-linear min-max characterization for the positive eigenvalues of $\DMh$. It is expressed in terms of magnetic Hardy space.
\begin{definition}\label{def:hdspace}
\[
\HtwohA=\{u\in L^2(\Omega) : \dMhX u=0\,, u_{|\partial\Omega}\in L^2(\partial\Omega)\}\,.\]
We consider the following Hilbert space
\[\HhA=H^1(\Omega)+\HtwohA\,,\]
which is endowed with the Hermitian scalar product given by
\[\forall (u_1,u_2)\in\HhA\times\HhA\,,\quad \langle u_1,u_2\rangle_{\HhA}=\langle u_1,u_2\rangle+\langle \dMhX u_1,\dMhX u_2\rangle+\langle u_1,u_2\rangle_{\partial\Omega}\,.\]
\end{definition}
Some useful properties of these spaces are recalled in Section \ref{def:hdspace}.

\begin{theorem}\label{mmax}
	Under Assumption \ref{hyp.reg}. We have, for all $h>0$ and $k\ge 1$,	 
	$$
	\lambda_k^+(h)=\min_{\underset{\dim W=k}{W\subset \HhA}}\max_{u\in W\setminus\{0\}}\frac{h\|u\|^2_{\partial\Omega}+\sqrt{h^2\|u\|^4_{\partial\Omega}+4\|u\|^2\|\dMhX u\|^2}}{2\|u\|^2}\,.
	$$
\end{theorem}
\begin{remark}\label{chargec}
	Due to the symmetry of the problem we can also use  this min-max characterization for the negative eigenvalues of  $\DMh$ after changing the sign of the magnetic field. Indeed, consider  the charge conjugation operator 
	\[
	C\colon \varphi\in \CC^2\mapsto \sigma_1\overline{\varphi}\in \CC^2\,,
	\]
	where $\overline{\varphi}$ is the vector of $ \CC^2$ made of the complex conjugate of the coefficients of $\varphi$. We have 
	$
	C \mathsf{Dom}(\DMh) =  \mathsf{Dom}(\DMh)\,,
	$
	and
	$
	C\DMh C = -\mathscr{D}_{h,-\mathbf{A}}\,.
	$	
	In particular, we get that
	\[
	\mathrm{sp}(\DMh)= -\mathrm{sp}(\mathscr{D}_{h,-\mathbf{A}})\,.
	\]
\end{remark}
\subsubsection{About the positive eigenvalues}
In order to state our next result on the asymptotic estimates of the $\lambda_k^+(h)$ 
we introduce  some notation to explicitly define the constant $C_k^+$ from \eqref{alf}.
\begin{notation}\label{not.BH}
	Let us denote by $\mathscr{O}(\Omega)$ and $\mathscr{O}(\CC)$ the sets of holomorphic functions on $\Omega$ and $\CC$.
	We consider the following (anisotropic)  Segal-Bargmann space
	\[\mathscr{B}^2(\CC) = \{u\in\mathscr{O}(\mathbb{C}) :\Nb(u)<+\infty\}\,,\] 
	where
	\[
	\Nb(u)=\left(\int_{\mathbb{R}^2} \left|u \left(y_1+iy_2\right)\right|^2e^{-\mathsf{Hess}_{x_{\min}}\phi(y, y)} \dd y\right)^{1/2}\,.
	\]
	We also introduce the Hardy space
	\[\mathscr{H}^2(\Omega)=\{u\in\mathscr{O}(\Omega) :\|u\|_{\partial \Omega}<+\infty\}\,,\] 
	where
	\[
	\|u\|_{\partial \Omega}=\left(\int_{\partial \Omega} \left|u \left(y_1+iy_2\right)\right|^2\dd y\right)^{1/2}\,.
	\]
	We also define for $P\in \mathscr{H}^2(\Omega)$, $A\subset \mathscr{H}^2(\Omega)$,
	\[\begin{split}
	\disth(P,A) = \inf\left\{
	\|P-Q\|_{\partial \Omega}\,, \mbox{ for all }Q\in A
	\right\}\,,
	\end{split}\]
	and
	for $P\in \mathscr{B}^2(\CC)$, $A\subset \mathscr{B}^2(\CC)$,
	\[\begin{split}
	\distb(P,A) = \inf\left\{
	\Nb(P-Q)\,, \mbox{ for all }Q\in A
	\right\}\,.
	\end{split}\]
	The following constant is important in  our asymptotic analysis
	\begin{align}\label{ck}
		\ck=\left(\frac{\mathrm{dist}_{\mathcal{H}}\left((z-z_{\min})^{k-1},\mathscr{H}^2_k(\Omega)\right)}{\mathrm{dist}_{\mathcal{B}}\left(z^{k-1},\mathcal{P}_{k-2}\right)}\right)^2\,,
	\end{align}
	where  $\mathcal{P}_{k-2} = {\rm span}\left(1,\dots , z^{k-2}\right)\subset \mathscr{B}^2(\CC)$, $\mathcal{P}_{-1}=\{0\}$ and
	\begin{equation}\label{space.Hk}
	\mathscr{H}^2_k(\Omega) = \{u\in \mathscr{H}^2(\Omega),\ u^{(n)}(z_{\min}) = 0, \mbox{ for }n\in\{0,\dots,k-1\}\}\,.
	\end{equation}
\end{notation}
\begin{theorem}\label{thm.main}
	Under Assumptions \ref{hyp.reg}, \ref{hyp.posB} and \ref{hyp.regphi}, we have for all $k\geq 1$,
	\[\lambda^+_k(h)= \ck h^{1-k} e^{2\phi_{\min}/h}(1+o_{h\to 0}(1))\,.\]
\end{theorem}
\begin{remark}
	Let us assume that $\Omega$ is the disk of radius $R$ centered at $0$, and that $B$ is radial. In this case $z_{\min}=0$ and $\mathsf{Hess}_{x_{\min}}\phi = B(0){\rm Id}/2$. Moreover,  using Fourier series, we see  that $(z^n)_{n\geq0}$ is an orthogonal basis for $\Nb$ and $\|\cdot\|_{\partial \Omega}$. In particular, $\mathscr{H}^2_k(\Omega)$ is $\|\cdot\|_{\partial \Omega}$-orthogonal to $z^{k-1}$ so that
	\[
	\disth \left(z^{k-1},\mathscr{H}^2_k(\Omega)\right)^2 = 	\|z^{k-1}\|_{\partial \Omega}^2 =  R^{2k-2}|\partial \Omega|= 2\pi R^{2k-1}\,.
	\]
	In addition, $\mathcal{P}_{k-2}$ is $\Nb$-orthogonal to $z^{k-1}$ so that
	\[
	\distb \left(z^{k-1},\mathcal{P}_{k-2}\right)^2 = \Nb(z^{k-1})^2 = 
	2\pi\frac{2^{k-1}(k-1)!}{B(0)^k}
	\,,
	\]
	%
	Thus,
	we get that
	\[\ck
	=\frac{B(0)^k}{(k-1)!}    \Big( \frac{R^2}{2}\Big)^{k-1} R
	\,.\]
\end{remark}

\begin{remark}
Theorem \ref{thm.main} can evoke some kind of tunneling estimate. In the semiclassical study of electric Schrödinger operators with symmetric wells, it is well-known that the lowest eigenvalues differ from each other modulo terms in the form $e^{-S/h}$. The quantity $S$ reflects the interaction between the wells and is related to lengths of geodesics connecting the wells. Here, the eigenvalues are themselves exponentially small and the $S$ is replaced by $\phi_{\min}$. In our analysis we will even see that the corresponding eigenfunctions are essentially localized near $x_{\min}$ which is determined by the \emph{global} magnetic geometry. That is why, we could interpret Theorem \ref{thm.main} as measure of a tunneling effect between every points of $\overline{\Omega}$.

\end{remark}

\subsubsection{About the negative eigenvalues}
We now turn to the negative eigenvalues of $\DMh$.  Consider, for all $\alpha\geq 0$,
\begin{equation}\label{eq.nuc}
\nu(\alpha)= \inf_{\underset{u\neq 0}{u\in \mathfrak{H}^2_{-\mathbf{A}_0}(\mathbb{R}^2_+)}}\frac{\int_{\mathbb{R}^2_+} |(-i\partial_{x_1}-x_2+i(-i\partial_{x_2})) u|^2\dd x_1\dd x_2
	+\alpha\int_{\mathbb{R}}|u(x_1,0) |^2\dd x_1}{\|u\|^2}\,,
\end{equation}
with $\mathbf{A}_0=(-x_2,0)$. Notice that the quadratic form minimized in \eqref{eq.nuc} corresponds to the magnetic Schrödinger operator on a half-plane with a constant magnetic field (equaling $1$) and equipped by a Robin-like boundary condition. 

\begin{remark}\label{rem.defa0}
We can prove (see Proposition \ref{prop.nu}) that the equation $\nu(\alpha)=\alpha^2$ has a unique positive solution, denoted by $a_0$. 
In fact, we will see that $a_0\in(0,\sqrt{2})$ equals
	\[
	 \inf_{\underset{u\neq 0}{u\in \mathfrak{H}^2_{-\mathbf{A}_0}(\mathbb{R}^2_+)}}\!\!\!\frac{
		\int_{\mathbb{R}}|u(s,0) |^2\dd s + \sqrt{\left(\int_{\mathbb{R}}|u(s,0) |^2\dd s\right)^2 + 4\|u\|^2\int_{\mathbb{R}^2_+} |(-i\partial_{s}-\tau+i(-i\partial_{\tau})) u|^2\dd s\dd \tau}	
	}{2 \|u\|^2}\,.
	\]
We emphasize that the constant $a_0$ is universal and it is strictly below the first Landau level $\sqrt{2}$. Numerical calculations suggest that $a_0$ is approximately equal to $1.31236$. 
More details are given in Sections \ref{sec.13} and \ref{sec.4}.
\end{remark}
\begin{theorem}\label{thm.main'}
	Under Assumptions \ref{hyp.reg} and \ref{hyp.posB}, we have
	\[\lambda^-_1(h)=h^{\frac 12}\min(\sqrt{2b_0},a_0\sqrt{b'_0})+o_{h\to 0}(h^{\frac 12})\,,\]
	where $\lambda^-_1(h)$ is defined in Section \ref{sec.basdef}, $b_0=\min_{\overline{\Omega}} B(x)$ and $b'_0=\min_{\partial\Omega} B(x)$. In particular, when $B\equiv b_0$ is constant, we have
	\[\lambda^-_1(h)=a_0\sqrt{b_0 h}+o_{h\to 0}(h^{\frac 12})\,.\]
\end{theorem}

\begin{remark}
The asymptotic analysis leading to Theorems \ref{thm.main} and \ref{thm.main'} strongly differ from each other. Indeed, the eigenfunctions are localized near $x_{\min}$ for the positive energies, whereas, when $B$ is constant, they are localized near the boundary for the negative ones. Moreover, in this last case, for non-constant magnetic fields (see the discussion in Appendix \ref{app.C}), the eigenfunctions might be localized inside if $b_0/b'_0$ is small enough. Consequently, the underlying semiclassical problems do not share the same structure. 
\end{remark}

\subsubsection{Fine structure of the negative eigenvalues for a constant magnetic field}
Let us now focus on the case with constant magnetic field $B=1$, and improve Theorem \ref{thm.main'}. In order to establish our improvement, and to make the analysis more elegant, we make the following assumption (see Appendix \ref{app.holo} for more detail). This assumption will allow to define \enquote{holomorphic tubular coordinates}, which are particularly adapted to our operator.
\begin{hyp}
The boundary $\partial\Omega$ is an analytic curve.	
\end{hyp}
\begin{notation}\label{not.negeigconst}
Various properties of the eigenvalues of the operator $\mathscr{M}_{\RR_+,\alpha,\xi}^-$ associated with the following (analytic) family of quadratic forms
\[q_{\RR_+, \alpha,\xi}^-(u)=\int_{\mathbb{R}_+}\left(|u'|^2+|(\xi-\tau)u|^2\right)\dd \tau+(\alpha-\xi)|u(0)|^2+\|u\|^2\,,\quad \alpha>0\,,\xi\in\mathbb{R}\,,\]
play a fundamental role in the analysis of the negative eigenvalues. The Robin boundary condition reads
\[\partial_t u(0)=(\alpha-\xi)u(0)\,.\]
We denote by $(\nu_{\RR_+,j}^-(\alpha,\xi))_{j\geq 1}$ the non-decreasing sequence of its eigenvalues. For shortness, we let $\nu^-(\alpha,\xi)=\nu_{\RR_+,1}^-$, and we denote by $u_{\alpha,\xi}$ the corresponding normalized positive eigenfunction. \end{notation}

We can prove (see Section \ref{sec.4}) that $\nu^-(\alpha,\cdot)$ has a unique minimum at some $\xi_{\alpha}$, which is non-degenerate.

The operator $\mathscr{M}_{\RR_+,\alpha,\xi}^-$ appears after using the partial Fourier transform in relation with \eqref{eq.nuc}.

	Let us consider	the following differential operator
	\begin{equation}\label{eq.Qeff}
		{\mathfrak{Q}}^{\mathrm{eff}}_h=\left( D_s+\mathfrak{t}_h\right)^2\\
	-\frac{\kappa^2}{12}\,,
	\end{equation}
	where 
	\[
			\mathfrak{t}_h=\frac{|\Omega|}{h|\partial\Omega|}-\frac{a_0}{h^{\frac 12}}+\frac{\pi}{|\partial \Omega|}\,,
	\]
	$a_0$ is defined in Remark \ref{rem.defa0} and 
	\[\mathfrak{c}_0 := \frac{a_0u^2_{a_0,a_0}(0)}{2a_0-u^2_{a_0,a_0}(0)}>0\,.\] 
	\begin{remark}
	 We will see that the denominator of $\mathfrak{c}_0$ is indeed positive. Moreover, this constant is directly related to the second derivative of the first negative dispersion curve $\vartheta_1^-$ at $a_0$ of the Dirac operator on the half-plane with constant magnetic field (equal to $1$), see Section \ref{sec.13} and Theorem \ref{thm.dispertionCurve}.
	 \end{remark}
	We denote by $\lambda_n(\mathfrak{Q}^{\mathrm{eff}}_h)$ the $n$-th eigenvalue of $\mathfrak{Q}^{\mathrm{eff}}_h$.

\begin{remark}\label{rem.gaugeth}
		 By gauge invariance, the spectrum of $\mathfrak{Q}_h^{\mathrm{eff}}$ does not change whenever $\mathfrak{t}_h$ is replaced by $\mathfrak{t}_h+\frac{2k\pi}{|\partial\Omega|}$. In particular, this means that $\lambda_n(\mathfrak{Q}_h^{\mathrm{eff}})$ is a periodic function of $\mathfrak{t}_h$.  
		We can easily check that, for all $n\geq 1$, there exists $C>0$ such that, for all $h\in(0,h_0)$,
		\[|\lambda_n(\mathfrak{Q}^{\mathrm{eff}}_h)|\leq C\,.\]

\end{remark}
Here comes our last main result.
\begin{theorem}\label{thm.main2}
	We have
	\[\lambda^-_n(h)=h^{\frac 12} a_0+h^{\frac 32}\mathfrak{c}_0\lambda_n(\mathfrak{Q}^{\mathrm{eff}}_h)+o(h^{\frac 32})\,.\]
\end{theorem}
In the disk case, we can compute the eigenvalues $\lambda_n(\mathfrak{Q}^{\mathrm{eff}}_h)$ recursively.
\begin{proposition}
Let $(m_n(h))_{n\geq1}$ be a sequence of $\ZZ$ which satisfies
\[
m_n(h) \in \arg\min\{|m+\mathfrak{t}_h|\,,m\in \ZZ\setminus\{m_1(h),\dots,m_{n-1}(h)\} \} \setminus\{m_1(h),\dots,m_{n-1}(h)\}\,.
\]
If $\Omega$ is a disk of radius $R>0$, we have for all $n\geq 1$,
\[
	\lambda_n(\mathfrak{Q}^{\mathrm{eff}}_h) 
	= |m_n(h)+\mathfrak{t}_h|^2-\frac{1}{12 R^2}
	\,.
\]
\end{proposition}
\begin{remark}
Since $\mathfrak{t}_h$ depends continuously on $h$ and $\mathfrak{t}_h\to+\infty$ as $h\to0$,  there are infinitely many $h>0$ for which there exists $k\in \ZZ$ such that
\begin{equation}\label{eq.divisioneuclidienneintro}
	\mathfrak{t}_h = \frac{1}{2}+k\,.
\end{equation}
In these cases, the spectrum of $\mathfrak{Q}^{\mathrm{eff}}_h$ for the disk of radius $R>0$ is
\[
\mathrm{sp}(\mathfrak{Q}^{\mathrm{eff}}_h) = \left\{\left|\frac{1}{2}+m\right|^2-\frac{1}{12 R^2}\,,m\in\NN\right\}\,,
\]
 each eigenvalue has multiplicity $2$ and the sequence $(m_n(h))_{n\geq1}$ is not uniquely defined. If \eqref{eq.divisioneuclidienneintro} is not satisfied then, all the eigenvalues are simple.
\end{remark}

Actually, the microlocal strategy used to obtain Theorem \ref{thm.main2} also allows to get results for variable magnetic fields. Such results are described with some details in Appendix \ref{app.C}. Somehow, the case with variable magnetic field is simpler since the variations of the field have a stronger effect than the geometry.

Theorem \ref{thm.main2} should also be compared to \cite[Theorem 1.1]{FH06} which deals with the Neumann Laplacian with constant magnetic field. In their paper, Fournais and Helffer show the crucial influence of the curvature on the spectral asymptotics and on the spectral gap. This gap is directly related to the localization of the eigenfunctions near the points of maximal curvature. We stress that it is not the case with Theorem \ref{thm.main2} since the effective operator does not induce a particular localization on the boundary. This reflects that our problem is more degenerate from the semiclassical point of view. In order to deal with such a degeneracy, we use a microlocal dimensional reduction to the boundary (also known as the Grushin method). As far as we know, such a method, combined with a non-linear characterization of the eigenvalues, does not seem to have been used before  to study semiclassical Dirac operators.  The version of this method that we use in this paper is inspired by \cite{M07} and closely related to the Ph. D. thesis by Keraval \cite{Keraval}. It was also recently used to establish a formula describing a pure magnetic tunnel effect in \cite{BHR19}.

\subsection{Dirac operators with uniform magnetic field on $\R^2$ and $\R^2_+$}\label{sec.13}
When considering Theorem \ref{thm.main'}, we can wonder what the interpretation of the positive constant $a_0$ is. In fact, an important part of the semiclassical analysis of spectral problems relies on the study of some operators obtained (formally) after a semiclassical zoom around each point of $\overline{\Omega}$. In the present article, these are magnetic Dirac operators with uniform magnetic field.
Thus, let us consider homogenenous Dirac operators on $\RR^2$ and $\RR^2_+ = \RR\times \RR_+$ with the same formalism by choosing the gauge associated with the vector potential $\mathbf{A}_0=(-x_2,0)^T$. Here, $B = 1$.
\begin{definition}\label{def.opHomogeneous}
	The operators $\mathscr{D}_{\RR^2_{\phantom{+}}}$ and $\mathscr{D}_{\RR^2_+}$ act as $\sigma\cdot(-i\nabla-\mathbf{A}_0)$ on
	\[
	\mathsf{Dom}(\mathscr{D}_{\RR^2_{\phantom{+}}}) = \left\{
	\varphi\in  H^1(\RR^2_{\phantom{+}}, \CC^2)\,,\quad x_2\varphi\in  L^2(\RR^2_{\phantom{+}}, \CC^2)
	\right\}
	\]
	and
	\[
	\mathsf{Dom}(\mathscr{D}_{\RR^2_+}) 
	= \left\{
	\varphi\in  H^1(\RR^2_+, \CC^2)\,,\quad x_2\varphi\in  L^2(\RR^2_+, \CC^2)\,,\quad\sigma_1 \varphi=\varphi\text{ on }\partial \RR^2_+
	\right\}\,.
	\]
\end{definition}
The spectral properties of $\mathscr{D}_{\RR^2_{\phantom{+}}}$ can be found for instance in \cite[Theorem 7.2]{T92}. A novelty in this paper is the study of $\mathscr{D}_{\RR^2_+}$.  

\begin{theorem}\label{thm.homogeneousOp}
	The operators $\mathscr{D}_{\RR^2_{\phantom{+}}}$ and $\mathscr{D}_{\RR^2_+}$ are self-adjoint and satisfy
	\[
	\mathrm{sp}(\mathscr{D}_{\RR^2_{\phantom{+}}}) = \{\pm\sqrt{2k}\,, k\in \NN\}\,,
	\]
	and
	\[
	\mathrm{sp}(\mathscr{D}_{\RR^2_+}) = (-\infty, -a_0]\cup[0,+\infty)\,,
	\]
	where $a_0\in(0,\sqrt{2})$ is defined in Remark \ref{rem.defa0}. The spectrum of $\mathscr{D}_{\RR^2}$ is made of infinitely degenerate eigenvalues. The spectrum of $\mathscr{D}_{\RR^2_+}$ is purely absolutely continuous.
\end{theorem}
\begin{remark}
We obtain the spectra of the Dirac operators with uniform magnetic field ${B} = b>0$ by rescaling. The spectra of Theorem \ref{thm.homogeneousOp} have  to be multiplied $\sqrt{b}$.
\end{remark}

We will present with more details the study of the negative part of the spectrum of $\mathscr{D}_{\RR^2_+}$ since many of the associated results will be used in the proof of the asymptotics of the negative eigenvalues.

\Bk

\subsection{The zigzag case}\label{rk.zigzag}
In this paper, we consider the Dirac operator with infinite-mass boundary condition (and its variants in Appendix \ref{app.A}). The so-called zigzag boundary condition also appears commonly in the description of the electrical properties of pieces of graphene. It is worth noticing that the spectral properties of the related operators exhibit completely different asymptotic behaviors compared with the ones studied here. More precisely, the operators $(\mathscr{Z}_{h,\mathbf{A}}^{\rm{\pm}}, \mathsf{Dom}(\mathscr{Z}_{h,\mathbf{A}}^{\rm{\pm}}))$ acting as $\sigma\cdot (\mathbf{p}-\mathbf{A})$ on different domains
\[\begin{split}
\mathsf{Dom}(\mathscr{Z}_{h,\mathbf{A}}^{\rm{-}}) 
&=
H^1_0(\Omega, \CC)\times
\{u\in L^2(\Omega, \CC)\,, \partial_z u\in L^2(\Omega, \CC)
\}\,,\\
\mathsf{Dom}(\mathscr{Z}_{h,\mathbf{A}}^{\rm{+}}) 
&=
\{u\in L^2(\Omega, \CC)\,, \partial_{\overline{z}} u\in L^2(\Omega, \CC)\}
\times H^1_0(\Omega, \CC)\,,
\end{split}\]
are self-adjoint. This is easily seen since by construction the operators
 $\mathscr{Z}_{h,\mathbf{A}}^{\rm{\pm}}$ have the supersymmetric structure 
	\begin{align*}
	\mathscr{Z}_{h,\mathbf{A}}^{\rm{\pm}}=
		\begin{pmatrix}
		0&D_{\pm}\\
		D_{\pm}^*&0
		\end{pmatrix}\,,
	\end{align*}
	where $D_+$ and $D_-^*$ have Dirichlet boundary conditions and act as $D_+=\dMh$ and $D_-^*=\dMhX$.
Moreover,  $0$ is an eigenvalue of infinite multiplicity for both of them and their kernels can be determined explicitly
(see \cite[Chapter 5]{T92}, \cite{S95} and \cite[Proposition 4.4]{BLTRS20a}).

Next notice that since $\sigma_3	\mathscr{Z}_{h,\mathbf{A}}^{\rm{\pm}}=-	\mathscr{Z}_{h,\mathbf{A}}^{\rm{\pm}}\sigma_3$ holds, the spectra of both operators is symmetric with respect to zero. Moreover, by simply squaring the operators one sees that, due to the isospectrality of $D_{\pm}D_{\pm}^*$ and 
$D_{\pm}^*D_{\pm}$ away from zero, 
\begin{align*}
\left\{\lambda^2, \lambda\in\mathrm{sp}{\,(\mathscr{Z}_{h,\mathbf{A}}^{+})}\setminus\{0\}\right\} =\mathrm{sp}\{D_+^*D_+\},
\,\,\mbox{and}\,\,
\left\{\lambda^2, \lambda\in\mathrm{sp}{\,(\mathscr{Z}_{h,\mathbf{A}}^{-})}\setminus\{0\}\right\}=\mathrm{sp}\{D_-D_-^*\}\,.
\end{align*}
Thus, their discrete spectrum satisfy
\[
	\mathrm{sp}_{d}{\,(\mathscr{Z}_{h,\mathbf{A}}^{\pm})} 
	= \mathrm{sp}{\,(\mathscr{Z}_{h,\mathbf{A}}^{\pm})} \setminus\{0\}
	= \left\{
	 \sqrt{\alpha_k^{\pm}(h)}\,, k\in\NN^*
	\right\}\cup
	\left\{
	 -\sqrt{\alpha_k^{\pm}(h)}\,, k\in\NN^*
	\right\}\,,
\]
where $(\alpha_k^{+}(h))_{k\geq1}$ and $(\alpha_k^{-}(h))_{k\geq1}$ are the ordered sequences of the eigenvalues (counted with multiplicity) of the operators $ D_+^*D_+$ and $D_-D_-^*$ that act as  
\[
|\p-\mathbf{A}|^2+ hB\,,\quad\mbox{and } \quad |\p-\mathbf{A}|^2- hB\,,
\]
 on $H^1_0(\Omega, \CC)\cap H^2(\Omega, \CC)$. Therefore, we deduce from \cite[Theorem 1.3.]{BLTRS20a}, that there exists $(\widetilde{C}_k(B, \Omega))_{k\geq 1}$ and $\theta_0\in(0,1]$ such that for all $k\geq 1$
\[\begin{split}
\left(\theta_0\widetilde{C}_k(B, \Omega) h^{1-k} e^{2\phi_{\min}/h}\right)^{1/2}&(1+o_{h\to 0}(1))
\leq\sqrt{\alpha_k^{-}(h)}
\\
&\leq  \left(\widetilde{C}_k(B, \Omega) h^{1-k} e^{2\phi_{\min}/h}\right)^{1/2}(1+o_{h\to 0}(1))\,,
\end{split}\]
as $h\to0$. Finally, it is well known that 
\[\begin{split}
\sqrt{\alpha_k^{+}(h)}
\geq  \sqrt{2b_0 h}\,.
\end{split}\]

\subsection{Structure of the article}
The article is organized as follows.

Section \ref{sec.NLminmax} is devoted to establish a non-linear min-max characterization of the positive eigenvalues, see Proposition \ref{prop.minmax}. A crucial step is given in Proposition \ref{prop.bij} which establishes an isomorphism between an eigenspace associated with a positive eigenvalue and a kernel of a Schrödinger operator.

In Section \ref{sec.3}, we prove Theorem \ref{thm.main} by using the non-linear min-max characterization. First, we establish an upper bound, see Lemma \ref{lem.ub} and Proposition \ref{prop.ubnuk}. Then, we prove that the minimizers of our non-linear min-max are approximated by functions such that $\dMhX u=0$ (see Section \ref{sec.approx}). This allows us to establish the lower bound, see Corollary \ref{cor.315} and Proposition \ref{prop.lbnuk}.

In Section \ref{sec.4}, we prove Theorem \ref{thm.homogeneousOp} about the spectrum of homogeneous magnetic Dirac operators on $\R^2$ and $\R^2_+$. Various properties of the corresponding dispersion curves are also established. The characterization of $a_0\in(0,\sqrt{2})$ as the unique solution of $\nu(\alpha)=\alpha^2$ is explained in Section \ref{sec.C4}. Numerical illustrations are also provided, see Section \ref{sec.C3}. In Section \ref{sec.moreformula}, we continue investigating the properties of the dispersion curves in order to understand how they behave when perturbing the flatness of the boundary. Especially, our computations are key to derive the explicit expression of the effective operator, see \eqref{eq.Qeff}. 

Section \ref{sec.5} is devoted to the proof of Theorem \ref{thm.main'}. One of the main ingredients is Proposition \ref{prop.Lambdah} which establishes a one-term asymptotics of the ground-state energy of a Pauli-Robin operator. This proposition is proved by means of a semiclassical partition of the unity. Near the boundary, due to the lack of ellipticity of the Cauchy-Riemann operators, we are led to introduce \emph{conformal} tubular coordinates thanks to the Riemann mapping theorem. This is the price to pay to be able to approximate the magnetic field by the constant magnetic field, and to control the remainders.

In Section \ref{sec.6}, we consider the case with constant magnetic field $B=1$ on $\Omega$ and we start the proof of Theorem \ref{thm.main2}. The first step is to show that the first eigenfunctions of our Pauli-Robin operator are localized near the boundary at the scale $h^{\frac12}$, see Proposition \ref{red.bound}. This allows to reduce the analysis to a tubular neighborhood of the boundary, see Corollary \ref{cor.redtub}. In this neighborhood, we use the holomorphic tubular coordinates given in Appendix \ref{app.holo} (where it is explained how to construct such coordinates by imposing a \emph{parametrization by arc length} of $\partial\Omega$), and put the operator under a normal form by means of changes of gauge and of functions. By rescaling with respect to the normal variable, we get the operator $\mathscr{N}_{a,\hbar}$, see \eqref{eq.Nahbar}. 

In Section \ref{sec.7}, up to inserting cutoff functions, this operator is seen as a pseudo-differential operator with respect to the curvilinear coordinate, see \eqref{eq.Nahbar2} and Section \ref{sec.microcut}. Corollary \ref{cor.approxlambdan} tells us that it is enough to study our pseudo-differential operator with cutoff functions $\check{\mathscr{N}}_{a,\hbar}$. Then, a parametrix is constructed by means of the Grushin formalism, see Proposition \ref{prop.parametrix}. This parametrix is used to reveal an effective operator, see \eqref{eq.peff0}, whose connection with the spectrum of $\check{\mathscr{N}}_{a,\hbar}$ is made in Proposition \ref{prop.distsp}. The spectral analysis of the effective operator is done in Section \ref{sec.peff}, see especially Proposition \ref{prop.lambdaneff+-}. Finally, the relation between the spectrum of the Pauli-Robin operator and the one of the Dirac operator is explained in Section \ref{sec.end}. 

In Appendix \ref{app.A}, we discuss some straightforward extensions of our results related to variable boundary conditions.

In Appendix \ref{app.C}, we explain how to describe all the negative eigenvalues when the magnetic field is variable, under generic assumptions. The main steps are only sketched since the analysis does not crucially involve subprincipal terms as for the constant magnetic field case. 

In Appendix \ref{sec.prooflemmahardy}, for the convenience of the reader, we recall why the magnetic Hardy space is a Hilbert space.

\section{A non-linear min-max characterization}\label{sec.NLminmax}
The aim of this section is to establish Theorem \ref{mmax}. To do so, we first establish in Section \ref{sec.2} some fundamental properties of the natural minimization space $\HhA$. Then, we prove that the $\lambda$-eigenspace of $\DMh$ are isomorphic with the $0$-eigenspace of an auxiliary operator $\mathscr{L}_\lambda$ depending quadratically on $\lambda$, see Proposition \ref{prop.bij}. Section \ref{sec.muk} is devoted to describe the spectrum of $\mathscr{L}_\lambda$, and in particular when $0\in\mathrm{sp}(\mathscr{L}_\lambda)$.

Throughout this section, $h>0$ is fixed.

\subsection{Magnetic Hardy spaces}\label{sec.2}
%

%
%
Let us recall the following proposition.
\begin{proposition}[{ \cite[Proposition 2.1.]{BLTRS20a}}]\label{prop:zeromode}
	The free Dirac operator and the magnetic Dirac operator are related by the formula 
	\be\label{eq:facto}
	e^{\sigma_3\phi/h}\sigma\cdot \p e^{\sigma_3\phi/h} = \sigma\cdot(\p-\mathbf{A})\,,
	\ee
	as operators acting on $H^1(\Omega, \CC^2)$ functions.
\end{proposition}
\begin{remark}\label{rem.Hardy-disc}
	By using the change of function $u=e^{-\phi/h}w$ suggested by Proposition \ref{prop:zeromode}, we have
	\[
	\HtwohA=e^{-\phi/h}\mathscr{H}^2_0(\Omega)\,, \quad \HhA=e^{-\phi/h}\mathfrak{H}_{\mathbf{0}}\,,
	\]
	where 
	\[
	\mathfrak{H}_{\mathbf{0}}=H^1(\Omega)+\mathscr{H}^2_0(\Omega)\,,
	\]
	and
	\[
	\mathscr{H}^2_0(\Omega) = \{w\in L^2(\Omega) : \partial_{\overline{z}} w=0\,, w_{|\partial\Omega}\in L^2(\partial\Omega)\}\,.
	\]
	Note that, for all $(u_1,u_2)\in\HhA\times\HhA$,
	\[
	\langle u_1,u_2\rangle_{\HhA}=\langle w_1,w_2\rangle_{L^2(e^{-2\phi/h})}+\langle -2ih\partial_{\overline{z}} w_1, -2ih\partial_{\overline{z}} w_2\rangle_{L^2(e^{-2\phi/h})}+\langle w_1,w_2\rangle_{\partial\Omega}\,,\]
	where $w_j = e^{\phi/h}u_j$ for $j = 1,2$.
	Then, by using the Riemann biholomorphism $F : \mathbb{D}\to\Omega$, the classical Hardy space $\mathscr{H}^2_0(\Omega)=\mathscr{H}^2(\Omega)$ becomes the canonical Hardy space
	\[\mathscr{H}^2(\mathbb{D})=\left\{f\in\mathscr{O}(\mathbb{D}) : \left(\frac{f^{(n)}(0)}{n!}\right)_{n\geq 0}\in\ell^2(\mathbb{N}) \right\}\,.\]
	Note that, for $f\in\mathscr{H}^2(\mathbb{D})$,
	\begin{equation}\label{eq.L2D}
	\|f\|^2=2\pi\sum_{n\geq 1}(2n+2)^{-1}|u_n|^2\,,\quad \|f\|^2_{\partial\Omega}=2\pi\sum_{n\geq 0}|u_n|^2\,,\quad u_n=\frac{f^{(n)}(0)}{n!}\,.
	\end{equation}
\end{remark}

The following lemma is a classical result. For the reader's convenience, we recall the proof in Appendix \ref{sec.prooflemmahardy}. 
\begin{lemma}\label{lem.Hardy}
	The space $(\HtwohA ,\langle\cdot,\cdot\rangle_{\partial\Omega})$ is a Hilbert space. Moreover, $\HtwohA$ is compactly embedded in $L^2(\Omega)$.
\end{lemma}
The next lemma is related to elliptic estimates for magnetic Cauchy-Riemann operators.
\begin{lemma}[{\cite[Theorem 4.6.]{BLTRS20a}}]\label{lem.elliptic}
	There exists $c>0$ such that, for all $h>0$, and for all $u\in \{v\in L^2(\Omega)\,,  \dMhX v\in L^2(\Omega)\} $,
	\[
	\sqrt{2hb_0}\|\Pi^\perp_{h,\mathbf{A}}u\|\leq\|\dMhX u\|\,,\qquad ch^2(\|\Pi^\perp_{h,\mathbf{A}}u\|_{\partial\Omega}+\|\nabla\Pi^\perp_{h,\mathbf{A}}u\|)\leq \|\dMhX u\|\,,
	\]
	where $\Pi_{h,\mathbf{A}}$ is the (orthogonal) spectral projection on the kernel of the adjoint of the operator $d_{h,\mathbf{A}}$ with Dirichlet boundary conditions, \emph{i.e.} $(d_{h,\mathbf{A}}, H^1_0(\Omega))^{\star}$, and \[\mathrm{Id} = \Pi_{h,\mathbf{A}} + \Pi_{h,\mathbf{A}}^\perp\,.\]
\end{lemma}
Let us now prove some properties of the spaces $\HhA$.
\begin{proposition}\label{prop.H}
	The following holds.
	\begin{enumerate}[\rm (i)]
		\item\label{eq.HAi} $( \HhA,\langle\cdot,\cdot\rangle_{ \HhA})$ is a Hilbert space.
		\item\label{eq.density} $H^1(\Omega)$ is dense in $ \HhA$.
		\item\label{eq.density2} The embedding $ \HhA\hookrightarrow L^2(\Omega)$ is compact.
	\end{enumerate}
\end{proposition}
\begin{proof}
	
	Let us prove \eqref{eq.HAi}. We consider a Cauchy sequence $(u_n)$ for $\|\cdot\|_{ \HhA}$. 
	It is obviously a Cauchy sequence for $\|\cdot\|$ and $\|\cdot\|_{\partial\Omega}$.	
	We write $u_n=\Pi_{h,\mathbf{A}}u_n+\Pi_{h,\mathbf{A}}^\perp u_n$. From Lemma \ref{lem.elliptic}, we see that $(\Pi_{h,\mathbf{A}}^\perp u_n)$ is a Cauchy sequence in $H^1(\Omega)$, and thus converges to some $u^\perp\in H^1(\Omega)$. Moreover, by using again Lemma \ref{lem.elliptic}, $(\Pi_{h,\mathbf{A}} u_n)$ is a Cauchy sequence in $\HtwohA$. From Lemma \ref{lem.Hardy}, $(\Pi_{h,\mathbf{A}} u_n)$ converges to some $u\in\HtwohA$. It follows that $(u_n)$ converges to $u+u^\perp$ in $\HhA$.
	
	Item \eqref{eq.density} is a consequence of \cite[Lemma C.1]{BLTRS20a}.
	
	By using again the orthogonal decomposition induced by $\Pi_{h,\mathbf{A}}$, and the compactness of $H^1(\Omega)\hookrightarrow L^2(\Omega)$, and of $\HtwohA\hookrightarrow L^2(\Omega)$ (see Lemma \ref{lem.Hardy}), we get \eqref{eq.density2}.
	
\end{proof}

\subsection{Statement of the min-max characterization}
The proof of Theorem \ref{mmax} is a consequence of Propositions \ref{cor.muk+} and \ref{prop.minmax}, see below.

\begin{notation}\label{not.muk}
	For all $k\geq 1$ and all $h>0$, we define
	\[\mu_k(h)=\inf_{\underset{\dim W=k}{W\subset H^1(\Omega)}}\sup_{u\in W\setminus\{0\}}\rho_+(u)\,,	\]
	where
	\begin{equation}\label{eq.rho+}
	\rho_+(u)=\frac{h\|u\|^2_{\partial\Omega}+\sqrt{h^2\|u\|^4_{\partial\Omega}+4\|u\|^2\|\dMhX{} u\|^2}}{2\|u\|^2}\,.
	\end{equation}
	\end{notation}

\begin{proposition}\label{cor.muk+}
	We have, for all $k\geq 1$,
	\[\mu_k(h)=\inf_{\underset{\dim W=k}{W\subset \mathfrak{H}_{h,\mathbf{A}}}}\sup_{u\in W\setminus\{0\}}\rho_+(u)=\min_{\underset{\dim W=k}{W\subset \mathfrak{H}_{h,\mathbf{A}}}}\sup_{u\in W\setminus\{0\}}\rho_+(u)>0\,.\]
\end{proposition}

\begin{proof}
	We use Proposition \ref{prop.H} \eqref{eq.density} \& \eqref{eq.density2}, and observe that $\rho_+(u)>0$ for all $u\in \mathfrak{H}_{h,\mathbf{A}}\setminus\{0\}$.
\end{proof}

\begin{proposition}\label{prop.minmax}
For all $k\geq 1$, and $h>0$, we have
\[\lambda^+_k(h)=\mu_k(h)\,.\]
\end{proposition}
The following sections are devoted to the proof of Proposition \ref{prop.minmax}.

In the following, we drop the $h$-dependence in the notation.

\subsection{A characterization of the $\mu_k$}\label{sec.muk}
\begin{notation}\label{not.quadminmax}
	Let $\lambda\geq 0$. Consider the quadratic form defined by
	\[\forall u\in\mathfrak{H}_{h,\mathbf{A}}\,,\quad Q_\lambda(u)=\|\dMhX{} u\|^2+h\lambda\|u\|^2_{\partial\Omega}-\lambda^2\|u\|^2\,,\]	
	and, for all $k\geq 1$,
	\[\begin{split}
	\ell_k(\lambda)
	&=\inf_{\underset{\dim W=k}{W\subset H^1(\Omega)}}\sup_{u\in W\setminus\{0\}}\frac{Q_\lambda(u)}{\|u\|^2}
	\\&=
	\inf_{\underset{\dim W=k}{W\subset H^1(\Omega)}}\sup_{u\in W\setminus\{0\}}\frac{\|\dMhX{} u\|^2+h\lambda\|u\|^2_{\partial\Omega}}{\|u\|^2}-\lambda^2
	\,,
	\end{split}
	\]	
where we recognize the $k$-th Rayleigh quotient of a magnetic Schrödinger operator with a parameter dependent boundary condition.
\end{notation}

Note that, for all $u\in\mathfrak{H}_{h,\mathbf{A}}\setminus\{0\}$,	
\begin{equation}\label{eq.sndorder}
Q_\lambda(u)=-\|u\|^2(\lambda-\rho_-(u))(\lambda-\rho_+(u))\,,
\end{equation}
where $\rho_+(u)$ is defined in \eqref{eq.rho+} and $\rho_-(u)$ is the other zero of the polynomial above.

From Proposition \ref{prop.H}, we deduce the following.
\begin{lemma}\label{lem.minmaxlk}
	For $\lambda>0$, the (bounded below) quadratic form $Q_\lambda$ is closed. The associated (unbounded) self-adjoint operator $\mathscr{L}_\lambda$ has compact resolvent, and its eigenvalues are characterized by the usual min-max formulas
	\[\ell_k(\lambda)=\inf_{\underset{\dim W=k}{W\subset H^1(\Omega)}}\sup_{u\in W\setminus\{0\}}\frac{Q_\lambda(u)}{\|u\|^2}=\min_{\underset{\dim W=k}{W\subset\mathfrak{H}_{h,\mathbf{A}}}}\max_{u\in W\setminus\{0\}}\frac{Q_\lambda(u)}{\|u\|^2}\,.\]
	
\end{lemma}
We prove some properties of $\ell_k$ seen as a function of $\lambda$.
\begin{lemma}\label{lem.lk}
	For all $k\geq 1$, the function $\ell_k : (0,+\infty)\to\mathbb{R}$ satisfies the following: 
	\begin{enumerate}[\rm (i)]
		\item\label{eq.lki} $\ell_1$ is concave,
		\item\label{eq.lkii} for all $\mu\in(0,\mu_1)$, and all $k\geq 1$, $\ell_k(\mu)>0$,
		\item\label{eq.lkiii}   $\lim_{\lambda\to+\infty}\ell_k(\lambda)=-\infty$,
		\item\label{eq.lkiv} $\ell_k$ is continuous,
		\item\label{eq.lkv} the equation $\ell_k(\lambda)=0$	has exactly one positive solution, denoted by $E_k$.
		\item \label{eq.lkvi} for all $\lambda>0$, 
		\[
		|\ell_k(\lambda)|\geq \lambda|E_k-\lambda|\,.
		\]
	\end{enumerate}	
\end{lemma}
\begin{proof}
	Item \eqref{eq.lki} follows by observing that the infimum of a family of concave functions is itself concave.
	
	It is enough to check Item \eqref{eq.lkii} for $k=1$. Consider $\mu>0$. Thanks to Proposition \ref{cor.muk+}, there exists a normalized function $u\in\mathfrak{H}_{h,\mathbf{A}}$ such that $\ell_1(\mu)=Q_\mu(u)$. If $\ell_1(\mu)\leq 0$, then, by \eqref{eq.sndorder}, we have that $\mu\geq\rho_+(u)\geq \mu_1$.
	
	By taking any finite dimensional space $W\subset H^1_0(\Omega)$, we readily see that
	\[\ell_k(\lambda)\leq \sup_{u\in W,\,\|u\|=1}\|\dMhX{} u\|-\lambda^2\,.\]
	We get Item \eqref{eq.lkiii}.
	
	Since $\ell_1$ is concave, it is also continuous. Then, the family $(\mathscr{L}_\lambda)_{\lambda>0}$ is analytic of type (B) in the sense of Kato (i.e., $\mathrm{Dom}(Q_\lambda)$ is independent of $\lambda>0$). This implies that the $\ell_k$ are continuous functions. Actually, this can directly be seen from the following equality
	\begin{equation}\label{eq.Ql1l2}
	\lambda_1^{-1}Q_{\lambda_1}(u)-\lambda_2^{-1}Q_{\lambda_2}(u)=(\lambda_2-\lambda_1)\left(\|\dMhX{} u\|^2(\lambda_1\lambda_2)^{-1}+\|u\|^2\right)\,,
	\end{equation}
	for all $\lambda_1,\lambda_2>0$ and $u\in \mathfrak{H}_{h,\mathbf{A}}$.
	
	Let us now deal with Item \eqref{eq.lkv}.
	Firstly, let $0<\lambda_1<\lambda_2$ and $W\subset \mathfrak{H}_{h,\mathbf{A}}$ with $\dim W=k$. By \eqref{eq.Ql1l2}, for all $u\in W\setminus\{0\}$, we have
	\[
	\lambda_1^{-1}Q_{\lambda_1}(u)\geq(\lambda_2-\lambda_1)\|u\|^2+ \lambda_2^{-1}Q_{\lambda_2}(u)\,.
	\]
	 and taking the supremum over the vectors $u\in W\setminus\{0\}$,
	\[
	\lambda_1^{-1}\sup_{u\in W\setminus\{0\}}\frac{Q_{\lambda_1}(u)}{\|u\|^2}\geq(\lambda_2-\lambda_1)+ \lambda_2^{-1}\sup_{u\in W\setminus\{0\}}\frac{Q_{\lambda_2}(u)}{\|u\|^2}\,.
	\]	
	Hence, taking the infimum over the subsets $W\subset \mathfrak{H}_{h,\mathbf{A}}$ of dimension $k$, we get
	\begin{equation}\label{eq.bb}
	\lambda_1^{-1}\ell_k(\lambda_1)\geq (\lambda_2-\lambda_1) + \lambda_2^{-1}\ell_k(\lambda_2)\,.
	\end{equation}	
	 By Items \eqref{eq.lkii}, \eqref{eq.lkiii} and \eqref{eq.lkiv}, there is at least one $\lambda>0$ such that $\ell_k(\lambda)=0$. Assume by contradiction that $\ell_k$ has two zeros $0<\lambda_1<\lambda_2$. By \eqref{eq.bb}, we get the contradition $0\geq \lambda_2-\lambda_1>0$.
Therefore, $\ell_k$ has only one positive zero.

To deal with Item \eqref{eq.lkvi}, we take first $\lambda_1 = E_k<\lambda_2$ so that
\[
	-\ell_k(\lambda_2)\geq  \lambda_2(\lambda_2-E_k) \,,
\]
and $|\ell_k(\lambda_2)|\geq  \lambda_2|\lambda_2-E_k|$. Then, consider $0<\lambda_1 <\lambda_2= E_k$,
\[
	\ell_k(\lambda_1)\geq \lambda_1(E_k-\lambda_1) \,,
\]
and $|\ell_k(\lambda_1)|\geq  \lambda_1|\lambda_1-E_k|$. These two inequalities give Item \eqref{eq.lkvi}.

\end{proof}

\begin{proposition}\label{prop.Em}
For all $k\geq 1$, $\mu_k$ is the only positive zero of $\ell_k$, i.e.,
	\[E_k=\mu_k\,.\]
\end{proposition}

\begin{proof}
In virtue of Proposition \ref{cor.muk+}, we notice that $\mu_k>0$. Then, we proceed in two steps.

Firstly, consider a subspace $W_k\subset\mathfrak{H}_{h,\mathbf{A}}$ of dimension $k$ such that
\[\max_{u\in W_k\setminus\{0\}}\rho_+(u)=\mu_k\,.\]
For all $u\in W_k\setminus\{0\}$, $\rho_+(u)\leq \mu_k$. By the definition of $\ell_k(\mu_k)$ and \eqref{eq.sndorder}, we have
\[\ell_k(\mu_k)\leq \max_{u\in W_k\setminus\{0\}} Q_{\mu_k}(u)\leq 0\,.\]
Secondly, for all subspace $W\subset\mathfrak{H}_{h,\mathbf{A}}$ of dimension $k$, we have
\[\mu_k\leq \max_{u\in W\setminus\{0\}}\rho_+(u)\,.\]
There exists $u_k \in W\setminus\{0\}$ such that $\mu_k\leq \rho_+(u_k)$. Then, we have
\[ \max_{u\in W\setminus\{0\}} Q_{\mu_k}(u)\geq Q_{\mu_k}(u_k)\geq 0\,,\]
and taking the infimum over $W$, we find $\ell_k(\mu_k)\geq 0$.

We deduce that $\ell_k(\mu_k)=0$ and conclude by using Lemma \ref{lem.lk} \eqref{eq.lkv}.
\end{proof}

\subsection{Proof of Proposition \ref{prop.minmax}}

\subsubsection{An isomorphism}
The following proposition is crucial.
\begin{proposition}\label{prop.bij}
Let $\lambda>0$. Then, the map 
\[
\mathscr{J}_\lambda : 
\left\{\begin{array}{ccc}
\ker\mathscr{L}_\lambda&\longrightarrow&\ker(\DMh{}-\lambda)\\
u&\longmapsto&\begin{pmatrix}
u\\
\lambda^{-1}\dMhX{}u
\end{pmatrix}
\end{array}\right.
\]
is well-defined and it is an isomorphism.
\end{proposition}
\begin{proof}
First, we show that the range of $\mathscr{J}_\lambda$ is indeed contained $\ker(\DMh{}-\lambda)$. Let $u\in\ker(\mathscr{L}_\lambda)$. Notice that $u\in\ker(\mathscr{L}_\lambda)$ is equivalent to
	\begin{equation}\label{eq.critL}
	\forall w\in\mathfrak{H}_{h,\mathbf{A}}\,,\quad Q_\lambda(u,w)=\langle \dMhX{} u,\dMhX{}w\rangle+h\lambda\langle u,w\rangle_{\partial\Omega}-\lambda^2\langle u,w\rangle=0\,.
	\end{equation}
	We set 
	\[\varphi=\begin{pmatrix}
	u\\
	v
	\end{pmatrix}\,,\quad v=\frac{\dMhX{} u}{\lambda}\,.\]
For all $\psi=\begin{pmatrix} w_1\\ w_2
	\end{pmatrix}\in\mathrm{Dom}(\mathscr{D}_{h,\mathbf{A}})$, we have
	\[\begin{split}
	\langle\varphi,\mathscr{D}_{h,\mathbf{A}}\psi\rangle&=\langle u,\dMh{}w_2\rangle+\langle v,\dMhX{}w_1\rangle\\
	&=\langle \dMhX{}u,w_2\rangle+\langle u,-ih\overline{n}\,w_2\rangle_{\partial\Omega}+\left\langle \frac{\dMhX{}u}{\lambda},\dMhX{}w_1\right\rangle\\
	&=\langle \lambda v,w_2\rangle+h\langle u,w_1\rangle_{\partial\Omega}+\left\langle \frac{\dMhX{}u}{\lambda},\dMhX{}w_1\right\rangle\\
	&=\langle \lambda v,w_2\rangle+\lambda\langle u,w_1\rangle=\lambda\langle\varphi,\psi\rangle\,,
	\end{split}\]
	where 
	\begin{enumerate}[--]
		\item the second equality comes from an integration by parts using Proposition \ref{prop.H} \eqref{eq.density},
		\item the third uses the boundary condition $w_2=in w_1$,
		\item the fourth uses \eqref{eq.critL}.
	\end{enumerate}
	This shows, by the definition of the adjoint, that $\varphi	\in\mathrm{Dom}(\DMh{}^*)=\mathrm{Dom}\,\DMh{}$ and in particular that $\DMh{}\varphi=\lambda\varphi$. Therefore, the map is well-defined, and we observe that it is injective.
	
	Let us show that $\mathscr{J}_\lambda$ is surjective. Consider $\begin{pmatrix}
		u\\
		v
	\end{pmatrix}\in\ker(\DMh{}-\lambda)$. The eigenvalue equation reads
\[\dMhX{}u=\lambda v\,,\quad \dMh{}v=\lambda u\,,\quad \mbox{ and } v=inu \mbox{ on $\partial\Omega$}\,.\]
Let $w\in\mathfrak{H}_{h,\mathbf{A}}$. Using the eigenvalue equation, and again an integration by parts, we get
\[\begin{split}
Q_\lambda(u,w)&=\langle \dMhX{} u,\dMhX{}w\rangle+h\lambda\langle u,w\rangle_{\partial\Omega}-\lambda^2\langle u,w\rangle\\
&=\lambda\langle v,\dMhX{}w\rangle+h\lambda\langle -i\overline{n}v,w\rangle_{\partial\Omega}-\lambda^2\langle u,w\rangle\\
&=\lambda\langle \dMh{}v,w\rangle-\lambda^2\langle u,w\rangle=\lambda^2\langle u,w\rangle-\lambda^2\langle u,w\rangle=0\,.
\end{split}\]
Hence, $u\in\mathrm{Dom}(\mathscr{L}_\lambda^*) = \mathrm{Dom}(\mathscr{L}_\lambda)$
 and $u\in\ker\mathscr{L}_\lambda$.
\end{proof}
\begin{corollary}\label{cor.P0}
	We set $\Lambda=\{\lambda^+_j\,,j\geq 1\}$ and $M=\{\mu_k\,,k\geq 1\}$. We have $\Lambda=M$. In particular, $\mu_1=\lambda^+_1$.	
\end{corollary}
\begin{proof}
Let $\lambda\in\Lambda$. Proposition \ref{prop.bij} implies that $0\in\mathrm{sp}(\mathscr{L}_\lambda)$. Then, there exists $j\geq 1$ such that $\ell_j(\lambda)=0$ and thus (Proposition \ref{prop.Em}) $\lambda=E_j=\mu_j\in M$.

Let $\mu\in M$. Then, there exists $j\geq 1$ such that $\mu=E_j$, and hence $\ell_j(\mu)=0$. In particular, $0\in\mathrm{sp}(\mathscr{L}_\mu)$ and thus $\mu\in\mathrm{sp}(\DMh{})$ by the isomorphism.
\end{proof}
\begin{notation}
Let us denote by $(a_k)_{k\geq 1}$ the unique increasing sequence such that $\Lambda=M=\{a_k\,,k\geq 1\}$. $(\mu_k)_{k\geq 1}$ is just a priori a non decreasing sequence.
In addition, for all $k\geq 1$, we set $m_k=\dim\ker(\DMh{}-a_k)$.
\end{notation}
\subsubsection{Induction argument}
Now, we can prove Proposition \ref{prop.minmax} by induction.

For $n\geq 0$, the induction statement is
\[\mathscr{P}(n)\quad :\quad \forall j\in\{1,\ldots, m_1+\ldots+m_n+1\}\,,\qquad \mu_j=\lambda^+_j\,.\]
Thanks to Corollary \ref{cor.P0}, $\mathscr{P}(0)$ is satisfied.

Let $n\geq 0$. Assume that, for all $0\leq k\leq n$, $\mathscr{P}(k)$ holds.

Notice that
\begin{equation}\label{eq.+1}
\mu_{m_1+\ldots+m_n+1}=\lambda^+_{m_1+\ldots+m_n+1}=a_{n+1}\,.
\end{equation}
By definition, we have $a_{n+1}\in\mathrm{sp}(\DMh{})$. Moreover, by using the isomorphism, \[m_{n+1}=\dim\ker(\mathscr{L}_{a_{n+1}})\,.\]
By the min-max theorem, there exists $j_0\geq 0$ such that
\[\ell_{j_0+1}(a_{n+1})=\ldots=\ell_{j_0+m_{n+1}}(a_{n+1})=0\,.\]
By Lemma \ref{lem.lk} \eqref{eq.lkv}, we have
\[a_{n+1}=E_{j_0+1}=\ldots=E_{j_0+m_{n+1}}\,,\]
so that, using again Proposition \ref{prop.Em},
\[a_{n+1}=\mu_{j_0+1}=\ldots=\mu_{j_0+m_{n+1}}\,.\]
Let us now show that $j_0=m_1+\ldots+m_{n}$. By the induction hypothesis, we have \[\mu_{m_1+\ldots+m_n}=a_n<a_{n+1}\,.\]
Thus, $j_0\geq m_1+\ldots+m_n$.

Let us suppose, by contradiction, that $j_0\geq m_1+\ldots+m_n+1$. With \eqref{eq.+1}, we get
\[\mu_{m_1+\ldots+m_n+1}=\mu_{j_0+1}=\ldots=\mu_{j_0+m_{n+1}}=a_{n+1}\,.\]
In particular, we have the $m_{n+1}+1$ relations:
\[\ell_{m_1+\ldots+m_n+1}(a_{n+1})=\ell_{j_0+1}(a_{n+1})=\ldots=\ell_{j_0+m_{n+1}}(a_{n+1})=0\,.\]
By the min-max theorem, this shows that
\[\dim\ker\mathscr{L}_{a_{n+1}}\geq m_{n+1}+1>m_{n+1}=\dim\ker(\DMh{}-a_{n+1})\,.\]
This contradicts the isomorphism property. Therefore, $j_0=m_1+\ldots+m_n$. This argument also shows that the multiplicity of $\mu_{m_1+\ldots+m_{n+1}}$ equals $m_{n+1}$. With the induction hypothesis, we get
\[\forall j\in\{1,\ldots, m_1+\ldots+m_{n+1}\}\,,\qquad \mu_j=\lambda^+_j\,.\]
By definition, we have
\[\lambda^+_{m_1+\ldots+m_{n+1}+1}=\min\left(\Lambda\setminus\{a_1,\ldots,a_{n+1}\}\right)=\min\left(M\setminus\{a_1,\ldots,a_{n+1}\}\right)\,.\]
We observe that $\mu_{m_1+\ldots+m_{n+1}+1}>a_{n+1}$ since the muliplicity of $\mu_{m_1+\ldots+m_{n+1}}$ equals $m_{n+1}$.
This proves that
\[\lambda^+_{m_1+\ldots+m_{n+1}+1}=\min\left(M\setminus\{a_1,\ldots,a_{n+1}\}\right)=\mu_{m_1+\ldots+m_{n+1}+1}\,.\]

This concludes the induction argument.

\section{Semiclassical  analysis of the positive eigenvalues}\label{sec.3}
In this section we prove Theorem \ref{thm.main} by applying Theorem \ref{mmax}, and considering the asymptotic analysis of a simpler problem. If one wants to estimate $\lambda^+_k(h)$, it is natural to use the functions of the Hardy space $\mathscr{H}^2_{h,\mathbf{A}}(\Omega)$ introduced in Definition \ref{def:hdspace} as test functions. This cancels the $\dMhX{}$-term in $\rho_+$ and leads to define
\[\nu_k(h)=\inf_{\underset{\dim W=k}{W\subset\mathscr{H}^2_{h,\mathbf{A}}(\Omega)}}\sup_{u\in W\setminus\{0\}}\frac{h\|u\|^2_{\partial\Omega}}{\|u\|^2}\,.\]
Theorem \ref{thm.main} is a consequence of the following three results.
\begin{lemma}\label{lem.ub}
	For all $k\in\mathbb{N}\setminus\{0\}$ and all $h>0$, we have
	\[\lambda^+_k(h)\leq \nu_k(h)\,.\]	
\end{lemma}
\begin{proof}
	It follows from the definition of $\nu_k(h)$.
	\end{proof}

Actually, we can prove that $\nu_k(h)$ is also a good asymptotic lower bound for $\lambda^+_k(h)$, see Section \ref{sec.approx} where the following is proved.
\begin{proposition}\label{prop.lb2}
	For all $k\geq 1$, we have
	\[\nu_k(h)\leq \lambda^+_k(h)(1+\mathscr{O}(h^\infty))\,.\]	
\end{proposition}
In the next section, we study the asymptotic behavior of $\nu_k(h)$, which is summarized in the following proposition.
\begin{proposition}\label{prop.nuk}
	For all $k\geq 1$,
	\[\nu_k(h)= C_k(B,\Omega) h^{1-k} e^{2\phi_{\min}/h}(1+o(1))\,,\]
	where $C_k(B,\Omega)$ is defined in \eqref{ck}.
\end{proposition}

\begin{remark}\label{rem.expsmall}
Proposition \ref{prop.nuk} shows that each $\nu_k(h)$ goes to zero exponentially when $h$ goes to zero. 
The analysis in Section \ref{sec.approx} strongly relies on this fact.
\end{remark}

\subsection{About the proof of Proposition \ref{prop.nuk}}

Using the change of function $u=e^{-\phi/h}v$ suggested by Proposition \ref{prop:zeromode} and detailed in Remark \ref{rem.Hardy-disc}, we get
\[\nu_k(h)=\inf_{\underset{\dim W=k}{W\subset\mathscr{H}^2(\Omega)}}\sup_{v\in W\setminus\{0\}}\frac{h\|v\|^2_{\partial\Omega}}{\|e^{-\phi/h}v\|^2}\,.\]
In what follows we give upper and lower bounds for $\nu_k(h)$. The technics borrow ideas from our previous work \cite{BLTRS20a}.

\subsubsection{Upper bound}
Let us consider $k\geq 1$ fixed.
\begin{notation}\label{not.Pn}
	Let us denote by $(P_n)_{n\in\NN}$ the $\Nb$-orthogonal family  such that $P_n(Z) = Z^n + \sum_{j=0}^{n-1}b_{n,j}Z^j$ obtained after a Gram-Schmidt process on $(1,Z,\dots, Z^n,\dots)$. Since $P_n$ is $\Nb$-orthogonal to $\mathcal{P}_{n-1}$,  we have
	\begin{equation}\label{rem:testfuncupbound}
	\begin{split}
	\distb\left(Z^{n},\mathcal{P}_{n-1}\right)
	&=\distb\left(P_n,\mathcal{P}_{n-1}\right)  
	= \inf\{\Nb(P_n-Q)\,, Q\in\mathcal{P}_{n-1}\}
	\\&
	=\inf\{\sqrt{\Nb(P_n)^2+\Nb(Q)^2}\,, Q\in\mathcal{P}_{n-1}\}
	= \Nb(P_{n})\,,\mbox{ for }n\in\NN\,.
	\end{split}
	\end{equation}
	%
	%
	 By the Cauchy-Schwarz inequality and the Cauchy formula, the subspace $\mathscr{H}^{2}_k(\Omega)$ is closed in $\mathscr{H}^2(\Omega)$.
	Therefore, there exists a unique $Q_{n}\in \mathscr{H}^{2}_k(\Omega)$ such that
	\[\disth \left((z-z_{\rm min})^{n},\mathscr{H}^2_k(\Omega)\right) = \|(z-z_{\rm min})^n-Q_n(z)\|_{\pa\Omega}\,,\]
	for $n\in \{0,\dots,k-1\}$. We recall that $\Nb$, $\mathcal{P}_{n-1}$, and $\mathscr{H}^2_k(\Omega)$ are defined in Notation \ref{not.BH}.
\end{notation}
The following proposition gives the wanted upper bound.
\begin{proposition}\label{prop.ubnuk}
	\[\nu_k(h)\leq \left(\frac{\disth\left((z-z_{\min})^{k-1},\mathscr{H}^2_k(\Omega)\right)}{\distb\left(z^{k-1},\mathcal{P}_{k-2}\right)}\right)^2 h^{1-k} e^{2\phi_{\min}/h}(1+o(1))\,.\]
	\end{proposition}
Proposition \ref{prop.ubnuk} is a consequence of the following lemmas and relies on the introduction of an appropriate $k$-dimensional vector space $V_{h,k}$ of test functions. 
Let us define $V_{h,k}$  by
\begin{equation}\label{eq.spaceupbd}
V_{h,k} = {\rm span}(w_{0,h},\dots,w_{k-1,h})\subset\mathscr{H}^2(\Omega)\,,
\end{equation}
\[
w_{n,h}(z) = h^{-\frac{1}{2}}P_{n}\left(\frac{z-z_{\rm min}}{h^{1/2}}\right) - h^{-\frac{1+n}{2}}Q_{n}(z),\mbox{ for }n\in \{0,\dots,k-1\}\,.
\]
\begin{lemma}[{\cite[Lemma 3.5]{BLTRS20a}}]\label{lem:L2esti1}
	Let  $h\in(0,1]$, $v_h = \sum_{j=0}^{k-1}c_j w_{j,h}\in V_{h,k}$ with $c_0,\dots c_{k-1}\in \CC$, and $(w_{j,h})_{j\in\{0,\dots,k-1\}}$ defined in \eqref{eq.spaceupbd}. We have 
	\begin{equation}\label{eq:normuV2}
	\int_{\Omega}|v_h|^2 e^{-2(\phi(x)-\phi_{\rm min})/h}\dd x
	= (1+o(1))\sum_{j=0}^{k-1}|c_j|^2\Nb(P_j)^2\,,
	\end{equation}
	 where $o(1)$ does not depend on $c = (c_0,\dots,c_{k-1})$.
\end{lemma}
For the numerator, we have the following result.
\begin{lemma}
		Let  $h\in(0,1]$, $v_h = \sum_{j=0}^{k-1}c_j w_{j,h}\in V_{h,k}$ with $c_0,\dots c_{k-1}\in \CC$. We have
	\[\|v_{h}\|_{\pa\Omega}^2\leq |c_{k-1}|^2h^{-k}\left\|(z-z_{\min})^{k-1}-Q_{k-1}\right\|_{\pa\Omega}^2+o(h^{-k})\|c\|^2_{\ell^2}\,.\]
		Here, $o(1)$ does not depend on $c_0,\dots c_{k-1}$.
	\end{lemma}
\begin{proof}
	Let us estimate $\|v_h\|_{\pa\Omega}$. From the triangle inequality, we get
\[\|v_h\|_{\pa\Omega}\leq |c_{k-1}|\|w_{k-1,h}\|_{\pa\Omega}+\sum_{j=0}^{k-2}|c_{j}|\|w_{j,h}\|_{\pa\Omega}\,.\]
Then, from degree considerations and the triangle inequality, we get, for $1\leq j\leq k-2$,
\[\|w_{j,h}\|_{\pa\Omega}=\mathscr{O}\left(h^{\frac{1-k}{2}}\right)\,,\]
and
\[\|w_{k-1,h}\|_{\pa\Omega}=(1+o(1))h^{-\frac{k}{2}}\left\|(z-z_{\min})^{k-1}-Q_{k-1}\right\|_{\pa\Omega}\,.\]
The conclusion follows. 
\end{proof}
Proposition \ref{prop.ubnuk} follows from these last two lemmas and a straightforward study of a finite dimensional min-max problem on vectors $(c_0,\dots,c_{k-1})\in \CC^k$.

\subsubsection{Lower bound}
Let $k\geq 1$. Let us consider an orthonormal family $(v_{j,h})_{1\leq j\leq k}$ (for the scalar product of $L^2(e^{-2\phi/h}\dd x)$) associated with the eigenvalues $(\nu_{j}(h))_{1\leq j\leq k}$. We define 
\[\mathscr{E}_{k}(h)=\underset{1\leq j\leq k}{\mathrm{span}}\, v_{j,h}\,.\]
The next two lemmas gives a priori bounds on the functions in $\mathscr{E}_{k}(h)$.
\begin{lemma}\label{lem.vhL2}
	There exist $C, h_{0}>0$ such that for all $v_{h}\in\mathscr{E}_{k}(h)$ and $h\in(0,h_{0})$, we have,
	\[\|v_{h}\|^2\leq C h^{-k} e^{2\phi_{\min}/h}\int_{\Omega}e^{-2\phi/h}|v_{h}|^2\dd x\,.\]
\end{lemma}

\begin{proof}
From the continuous embedding $\mathscr{H}^2(\Omega)\hookrightarrow L^2(\Omega)$, and Proposition \ref{prop.ubnuk}, there exist $c, C, h_0>0$ such that, for all $h\in(0,h_0)$ and all $v\in\mathscr{E}_k(h)$,
\[ch\|v\|^2\leq h\|v\|^2_{\partial\Omega}\leq \nu_k(h)\int_{\Omega}e^{-2\phi/h}|v_{h}|^2\dd x\leq C h^{1-k} e^{2\phi_{\min}/h}\int_{\Omega}e^{-2\phi/h}|v_{h}|^2\dd x\,.\]
\end{proof}

\begin{lemma}\label{lem.normL2} 
	Let $\alpha\in(1/3,1/2)$.
	We have 
	\[\lim_{h\to 0}\,\sup_{v_h\in\mathscr{E}_{k}(h)\setminus\{0\}}\left|\frac{\int_{D(x_{\rm min}, \ h^{\alpha})}e^{-2\phi/h}|v_{h}(x)|^2\dd x}{\int_{\Omega}e^{-2\phi/h}|v_{h}(x)|^2\dd x}-1\right|=0\,,\]
	\end{lemma}

\begin{proof}
Assume that $\alpha\in\left(\frac{1}{3},\frac{1}{2}\right)$. We have for all $x\in D(x_{\min},h^\alpha)$,
\[\phi(x)=\phi_{\min}+\frac{1}{2}\mathsf{Hess}_{x_{\min}}\phi (x-x_{\min}, x-x_{\min})+\mathscr{O}(h^{3\alpha})\,.\]
By the maximum principle, 
\[
\min_{x\in D(x_{\rm min}, \ h^{\alpha})^c}\phi(x) 
= \min_{x\in \partial D(x_{\rm min}, \ h^{\alpha})^c}\phi(x)
\geq \phi_{\min} + h^{2\alpha}\min \mathrm{sp}(\mathsf{Hess}_{x_{\min}}) + \mathscr{O}(h^{3\alpha})\,.
\]
The result follows from Lemma \ref{lem.vhL2}.
\end{proof}
We are now in a good position to study the lower bound.
\begin{proposition}\label{prop.lbnuk}
	We have
		\[\nu_k(h)\geq \left(\frac{\disth\left((z-z_{\min})^{k-1},\mathscr{H}^2_k(\Omega)\right)}{\distb\left(z^{k-1},\mathcal{P}_{k-2}\right)}\right)^2 h^{1-k} e^{2\phi_{\min}/h}(1+o(1))\,.\]
	\end{proposition}
\begin{proof}
Let  $\alpha\in(1/3,1/2)$. With Lemma \ref{lem.normL2},
	\begin{equation}\label{eq.taylphase} he^{2\phi_{\min}/h}\norm{v_h}_{\partial\Omega}^2(1+o(1))
	\leq
	\nu_k(h)\norm{e^{-\frac{1}{2h}\mathsf{Hess}_{x_{\min}}\phi (x-x_{\min}, x-x_{\min})}v_h}_{L^2(D(x_{\min},h^{\alpha}))}^2\,.
	\end{equation}
	In the following, we split the proof into two parts. Firstly, we replace $v_h$ by its Taylor expansion at the order $k-1$ at $x_{\min}$ in the R. H. S. of \eqref{eq.taylphase}. Secondly, we do the same for the L.H.S. of the same equation.

	\begin{enumerate}[\rm i.]
		
		\item In view of the Cauchy formula, and the Cauchy-Schwarz inequality, there exist $C>0, h_{0}>0$ such that, for all $h\in(0,h_{0})$, for all $v\in\mathscr{H}^2(\Omega)$,  all $z_0\in D(x_{\min},h^{\alpha})$, and all $n\in \{0,\dots k\}$,
		\begin{equation}\label{eq.CCS2}
		|v^{(n)}(z_{0})|\leq C\|v\|_{\partial\Omega}\,.
		\end{equation}
		Let us define, for all $v\in \mathscr{H}^2(\Omega)$,
		\[N_{h}(v)=\norm{e^{-\frac{1}{2h}\mathsf{Hess}_{x_{\min}}\phi (x-x_{\min}, x-x_{\min})}v}_{L^2(D(x_{\min},h^{\alpha}))}\,.\]
		 By the Taylor formula, we can write
		\[v_{h}=\mathrm{Tayl}_{k-1}v_{h}+R_{k-1}(v_{h})\,,\]
		where
		\[\mathrm{Tayl}_{k-1}v_{h}=\sum_{n=0}^{k-1}\frac{v^{(n)}_{h}(z_{\min})}{n!}(z-z_{\min})^n\,,\]
		and, for all $z_{0}\in D(z_{\min}, h^\alpha)$,
		\[|R_{k-1}(v_{h})(z_{0})|\leq C|z-z_{\min}|^k\sup_{D(z_{\min}, h^\alpha)}|v_{h}^{(k)}|\,.\]
		With \eqref{eq.CCS2} and a rescaling, the Taylor remainder satisfies
		\[N_{h}(R_{k-1}(v_{h}))\leq Ch^{\frac{k}{2}}h^{\frac{1}{2}}\|v_{h}\|_{\partial\Omega}\,.\]
		Thus, by the triangle inequality,
		\[N_{h}(v_{h})\leq N_{h}(\mathrm{Tayl}_{k-1}v_{h})+Ch^{\frac{k}{2}}h^{\frac{1}{2}}\|v_{h}\|_{\partial\Omega}\,.\]
		Thus, with \eqref{eq.taylphase}, we get
		\[(1+o(1))e^{\phi_{\min}/h}\sqrt{h}\|v_{h}\|_{\partial\Omega}\leq \sqrt{\nu_{k}(h)}N_{h}(\mathrm{Tayl}_{k-1}v_{h})+C\sqrt{\nu_{k}(h)}h^\frac{1+k}{2}\|v_{h}\|_{\partial\Omega}\,,\]
		so that, thanks to Proposition \ref{prop.ubnuk},
		\begin{equation}\label{eq.ineqTaylor}
		(1+o(1))e^{\phi_{\min}/h}\sqrt{h}\|v_{h}\|_{\partial\Omega}\leq \sqrt{\nu_{k}(h)}N_{h}(\mathrm{Tayl}_{k-1}v_{h})\leq \sqrt{\nu_{k}(h)}\hat N_{h}(\mathrm{Tayl}_{k-1}v_{h})\,,
		\end{equation}
		with 
		\[\hat N_{h}(w)=\norm{e^{-\frac{1}{2h}\mathsf{Hess}_{x_{\min}}\phi (x-x_{\min}, x-x_{\min})}w}_{L^2(\RR^2)}\,.\]
		This inequality shows in particular that  $\mathrm{Tayl}_{k-1}$ is injective on $\mathscr{E}_{k}(h)$ and 
		\begin{equation}\label{eq.dim-k}
		\mathrm{dim}\mathrm{Tayl}_{k-1}\mathscr{E}_{k}(h)=k\,.
		\end{equation}
		
		\item
		Let us recall that
		\[\mathscr{H}^2_{k}(\Omega)=\{u\in\mathscr{H}^2(\Omega) : \forall n\in\{0,\ldots, k-1\}\,, u^{(n)}(x_{\min})=0\}\,.\]
		Since $(v_{h}-\mathrm{Tayl}_{k-1}v_{h})\in\mathscr{H}^2_{k}(\Omega)$, we have, by the triangle inequality,
		\[
		\begin{split}
		&\|v_{h}\|_{\partial\Omega}
		\geq
		\left\| \frac{v_{h}^{(k-1)}(z_{\min})}{(k-1)!}(z-z_{\min})^{k-1}+(v_{h}-\mathrm{Tayl}_{k-1}v_{h})\right\|_{\partial\Omega}
		-\left\|\mathrm{Tayl}_{k-2}v_{h} \right\|_{\partial\Omega}
		\\
		&\geq
		\frac{|v_{h}^{(k-1)}(z_{\min})|}{(k-1)!}\mathrm{dist}_{\mathcal H}((z-z_{\min})^{k-1},\mathscr{H}^2_{k}(\Omega))
		-\left\|\mathrm{Tayl}_{k-2}v_{h} \right\|_{\partial\Omega}\,,
		\end{split}
		\]
		where
		\[\begin{split}
		&\mathrm{dist}((z-z_{\min})^{k-1},\mathscr{H}^2_{k}(\Omega)) 
		\\&\qquad= \inf\left\{
		\left\| (z-z_{\min})^{k-1}-Q(z)\right\|_{\partial\Omega}\,, \mbox{ for all }Q\in \mathscr{H}^2_{k}(\Omega)
		\right\}\,.
		\end{split}\]
		Using again the triangle inequality,
		\[\|\mathrm{Tayl}_{k-2}v_{h}\|_{\partial\Omega}\leq C\sum_{n=0}^{k-2}|v_{h}^{(n)}(z_{\min})|\,.\]
		Moreover,
		\[\begin{split}
		\sum_{n=0}^{k-2}|v_{h}^{(n)}(z_{\min})|\leq h^{-\frac{k-2}{2}}\sum_{n=0}^{k-2}h^{\frac{n}{2}}|v_{h}^{(n)}(z_{\min})|&\leq h^{-\frac{k-2}{2}}\sum_{n=0}^{k-1}h^{\frac{n}{2}}|v_{h}^{(n)}(z_{\min})|\\
		&\leq Ch^{-\frac{k-2}{2}} h^{-\frac{1}{2}}\hat N_{h}(\mathrm{Tayl}_{k-1}v_{h})\,,
		\end{split}
		\]
		where we used the rescaling property
		\begin{equation}\label{eq.rescaling.Nh}
		\hat N_{h}\left(\sum_{n=0}^{k-1} c_{n}(z-z_{\min})^n\right)=h^{\frac{1}{2}}\hat N_{1}\left(\sum_{n=0}^{k-1} c_{n}h^{\frac{n}{2}}(z-z_{\min})^n\right)\,,
		\end{equation}
		and the equivalence of the norms in finite dimension:
		\[\exists C>0\,,\forall d\in\mathbb{C}^k\,,\quad C^{-1}\sum_{n=0}^{k-1}|d_{n}|\leq \hat N_{1}\left(\sum_{n=0}^{k-1} d_{n}(z-z_{\min})^n\right) \leq C\sum_{n=0}^{k-1}|d_{n}|\,.\]
		We find
		\begin{equation*}
		\|v_{h}\|_{\partial\Omega}\geq \frac{|v_{h}^{(k-1)}(z_{\min})|}{(k-1)!}\mathrm{dist}((z-z_{\min})^{k-1},\mathscr{H}^2_{k}(\Omega))
		-Ch^{-\frac{k-2}{2}} h^{-\frac{1}{2}}\hat N_{h}(\mathrm{Tayl}_{k-1}v_{h})\,,
		\end{equation*}
		and thus, by \eqref{eq.ineqTaylor},
		\begin{multline}\label{eq.ineqTaylor2}
		(1+o(1))e^{\phi_{\min}/h}\sqrt{h}\frac{|v_{h}^{(k-1)}(z_{\min})|}{(k-1)!}\mathrm{dist}((z-z_{\min})^{k-1},\mathscr{H}^2_{k}(\Omega))\\
		\leq \left(\sqrt{\nu_{k}(h)}+Ch^{\frac{2-k}{2}}e^{\phi_{\min}/h}\right)\hat N_{h}(\mathrm{Tayl}_{k-1}v_{h})\,.
		\end{multline} 
	\end{enumerate}
Let us now end the proof of the lower bound by using \eqref{eq.ineqTaylor2} and \eqref{eq.dim-k}.

Since we have \eqref{eq.dim-k}, we deduce that
\begin{multline}\label{eq.ineqTaylor4}
(1+o(1))e^{\phi_{\min}/h}\sqrt{h}\mathrm{dist}_{\mathcal{H}}((z-z_{\min})^{k-1},\mathscr{H}^2_{k}(\Omega))
\sup_{c\in\mathbb{C}^k} \frac{|c_{k-1}|}{\hat N_{h}(\sum_{n=0}^{k-1}c_{n}(z-z_{\min})^n)}\\
\leq \sqrt{\nu_{k}(h)}+Ch^{\frac{2-k}{2}}e^{\phi_{\min}/h}\,.
\end{multline}
By \eqref{eq.rescaling.Nh}, we infer
\[h^{\frac{1}{2}}\sup_{c\in\mathbb{C}^k} \frac{|c_{k-1}|}{\hat N_{h}(\sum_{n=0}^{k-1}c_{n}(z-z_{\min})^n)}=\sup_{c\in\mathbb{C}^k} \frac{h^{\frac{1-k}{2}}|c_{k-1}|}{\hat N_{1}(\sum_{n=0}^{k-1}c_{n}(z-z_{\min})^n)}\,.\]
Since $\hat N_{1}$ is related to the Segal-Bargmann norm $N_{\mathcal{B}}$ via a translation, and recalling Notation \ref{not.Pn}, we get 
\[\sup_{c\in\mathbb{C}^k} \frac{|c_{k-1}|}{\hat N_{1}(\sum_{n=0}^{k-1}c_{n}(z-z_{\min})^n)}=\sup_{c\in\mathbb{C}^k} \frac{|c_{k-1}|}{N_{\mathcal{B}}(\sum_{n=0}^{k-1}c_{n}z^n)}=\frac{1}{N_{\mathcal{B}}(P_{k-1})}\,.\]
Thus,
\begin{equation}\label{eq.ineqTaylor5}
(1+o(1))h^{\frac{1-k}{2}}e^{\phi_{\min}/h}\frac{\mathrm{dist}_{\mathcal{H}}((z-z_{\min})^{k-1},\mathscr{H}^2_{k}(\Omega))}{N_{\mathcal{B}}(P_{k-1})}\leq \sqrt{\nu_{k}(h)}\,.
\end{equation}

The conclusion follows.
\end{proof}

\subsection{Approximation results}\label{sec.approx}
Let us roughly explain the strategy to establish Proposition \ref{prop.lb2}.  Recall Theorem \ref{mmax} which gives a nonlinear min-max formulation for $\lambda_k^+(h)$. Let us remark that the functions in the range of the orthogonal projector $\Pi_{h,\mathbf{A}}$ defined in Lemma \ref{lem.elliptic},
\[
{\mathrm{ran}}\, \Pi_{h,\mathbf{A}} = \{u\in L^2(\Omega)\,,\dMhX u=0\} = e^{-\phi/h }\mathscr{O}(\Omega)\cap L^2(\Omega)\,,
\] 
do not have in general an $L^2(\partial \Omega)$-trace. Nevertheless, since 
\[
{\rm ran}\, \Pi_{h,\mathbf{A}}^\perp \subset H^1(\Omega)\,, 
\] 
 the projection $\Pi_{h,\mathbf{A}}u = u - \Pi_{h,\mathbf{A}}^\perp u$ has an $L^2(\partial \Omega)$-trace whenever $u\in \mathfrak{H}_{h,\mathbf{A}}$ ($u$ itself has an $L^2(\partial \Omega)$-trace) :

 \[
{\rm rg} \,{\Pi_{h,\mathbf{A}}}\!\!\upharpoonleft_{\mathfrak{H}_{h,\mathbf{A}}} = \HtwohA\,.
 \]
%
%
Consider a minimizing subspace $W_h\subset \mathfrak{H}_{h,\mathbf{A}}=\HtwohA+H^1(\Omega)$ (of dimension $k$). Then, we can prove that $W_h$ is quasi invariant under the orthogonal projector $\Pi_{h,\mathbf{A}}$,  see Lemma \ref{lem.approximation}.
So, we would like to write $\rho_+(u)\simeq \rho_+(\Pi_{h,\mathbf{A}}u)$ for all $u\in W_h$. 
In the following, we will use approximate subspaces to highlight the stability of the projection procedure. For that purpose, we will use a number $M_k(h)\geq \lambda_k^+(h)$ such that
\[M_k(h)=\lambda_k^+(h)(1+o(1))\,.\]

\begin{remark}\label{rem.Mkh}
	By Remark \ref{rem.expsmall}, 
	$M_k(h)$ goes itself exponentially fast to zero.
\end{remark}
\begin{notation}
For notational simplicity, we write $M\equiv M_k(h)$.	
\end{notation}	

There exists $W_h\subset \mathfrak{H}_{h,\mathbf{A}}$ with $\dim W_h=k$ such that
\begin{equation}\label{eq.Wh}
\lambda_k^+(h)\leq \sup_{W_h\setminus\{0\}}\rho_+(u)\leq M\,.
\end{equation}
The following lemma is straightforward.

\begin{lemma}\label{lem.Qh}
	For all $u\in \mathfrak{H}_{h,\mathbf{A}}$, 
	we have
	\[2h\|u\|^2_{\partial\Omega}\leq \mathscr{Q}_h(u)\,,\quad 2\|u\|\|\dMhX{} u\|\leq \mathscr{Q}_h(u)\,,\]		
	where
	\[\mathscr{Q}_h(u)=h\|u\|^2_{\partial\Omega}+\sqrt{h^2\|u\|^4_{\partial\Omega}+4\|u\|^2\|\dMhX{} u\|^2}\,.\]
\end{lemma}
Thanks to Lemma \ref{lem.Qh} and \eqref{eq.Wh}, we get the following.
\begin{lemma}\label{lem.ubndint}
For all $u\in W_h$,
\begin{equation}\label{eq.ubnd}
h\|u\|^2_{\partial\Omega}\leq M\|u\|^2\,,
\end{equation}
and thus
\begin{equation}\label{eq.uint}
\|\dMhX{} u\|^2\leq M^2\|u\|^2\,.
\end{equation}
\end{lemma}

\begin{lemma}\label{lem.approximation}
	For all $u\in W_h$, we have
	\begin{equation}\label{eq.a}
	\|\Pi_{h,\mathbf{A}}^\perp u\|\leq \frac{M}{\sqrt{2hb_0}}\|u\|\,,
	\end{equation}
	\begin{equation}\label{eq.b}
	\|\Pi_{h,\mathbf{A}}^\perp u\|_{\partial\Omega}\leq \frac{M}{ch^2} \|u\|\,.
	\end{equation}
	Moreover, for $h$ small enough, $\Pi_{h,\mathbf{A}}\!\!\upharpoonleft_{W_h}$ is injective.
	
	\end{lemma}

\begin{proof}
	Combining \eqref{eq.uint} and Lemma \ref{lem.elliptic}, we readily get \eqref{eq.a} and \eqref{eq.b}. The injectivity follows from \eqref{eq.a} and Remark \ref{rem.Mkh}.
	\end{proof}

\begin{proposition}\label{prop.uPiu}
For all $u\in W_h$, we have
\[\lambda_k^+(h)(1+\mathscr{O}(h^\infty))\|\Pi_{h,\mathbf{A}}u\|^2\geq h\|\Pi_{h,\mathbf{A}}u\|^2_{\partial\Omega}\,.\]
\end{proposition}

\begin{proof}
Let us consider \eqref{eq.ubnd}. We have
\[M^{\frac{1}{2}}\|u\|\geq  \sqrt{h}\|u\|_{\partial\Omega}=\sqrt{h}\|\Pi_{h,\mathbf{A}}u+\Pi_{h,\mathbf{A}}^\perp u\|_{\partial\Omega}\geq \sqrt{h}(\|\Pi_{h,\mathbf{A}}u\|_{\partial\Omega}-\|\Pi_{h,\mathbf{A}}^\perp u\|_{\partial\Omega})\,.\]
By \eqref{eq.b}, we get
\[M^{\frac{1}{2}}\left(1+h^{-\frac 32}\frac{1}{c}M^{\frac{1}{2}}\right)\|u\|\geq \sqrt{h}\|\Pi_{h,\mathbf{A}}u\|_{\partial\Omega}\,.\]
From \eqref{eq.a}, and the triangle inequality, we have
\[\left(1-\frac{M}{\sqrt{2hb_0}}\right)\|u\|\leq\|\Pi_{h,\mathbf{A}}u\|\,.\]
By Remark \ref{rem.Mkh}, we see that, for $h$ small enough, $1-\frac{M}{\sqrt{2hb_0}}>0$. Hence,
\[M^{\frac{1}{2}}\left(1+h^{-\frac 32}\frac{1}{c}M^{\frac{1}{2}}\right)\left(1-\frac{M}{\sqrt{2hb_0}}\right)^{-1}\|\Pi_{h,\mathbf{A}}u\|\geq \sqrt{h}\|\Pi_{h,\mathbf{A}}u\|_{\partial\Omega}\,.\]
Squaring this, and using Remark \ref{rem.Mkh}, we obtain the desired estimate.
	\end{proof}
\begin{corollary}\label{cor.315}
For all $k\geq 1$, we have
\[\nu_k(h)\leq \lambda_k^+(h)(1+\mathscr{O}(h^\infty))\,.\]
	
	\end{corollary}

\begin{proof}
Since $\Pi_{h,\mathbf{A}}\!\!\upharpoonleft_{W_h}$ is injective, we have $\dim\Pi_{h,\mathbf{A}}(W_h)=k$. Moreover, $ \Pi_{h,\mathbf{A}}(W_h)\subset \HtwohA$. 
%
The conclusion follows from Proposition \ref{prop.uPiu} and the definition of $\nu_k(h)$.
	\end{proof}

\section{Homogeneous Dirac operators}\label{sec.4}
This section is devoted to the proof of Theorem \ref{thm.homogeneousOp}. After using the fibration induced by the partial Fourier transform with respect to the tangential variable, it is a consequence of Theorem \ref{thm.dispertionCurve}. More formulas related to the fibered operators are established in Section \ref{sec.moreformula} in our way of proving Theorem \ref{thm.main2}.

\subsection{The fibered operators}

Using the Fourier transform in the $x_1$-direction, we are lead to introduce the following family of fiber operators that are one-dimensional Dirac operators.
\begin{definition}\label{def.fiberOp}
	Let $\xi\in \RR$ be the Fourier variable in the $x_1$-direction. The operators $\mathscr{D}_{\xi, \RR}$ and $\mathscr{D}_{\xi, \RR_+}$ act as 
	\[
	-i\sigma_2\pa_t  + \sigma_1(\xi+t)\,,
	\]
	on
	\[
	\mathsf{Dom}(\mathscr{D}_{\xi,\RR}) = B^1(\RR, \CC^2) :=
	\left\{
	\varphi\in  H^1(\RR, \CC^2)\,,\quad t\varphi\in  L^2(\RR, \CC^2)
	\right\}
	\]
	and
	\[
	\mathsf{Dom}(\mathscr{D}_{\xi, \RR_+}) 
	= \left\{
	\varphi\in  B^1(\RR^2_+, \CC^2)\,,\quad\sigma_1 \varphi(0)=\varphi(0)
	\right\}\,,
	\]
	where
	\[
	B^1(\RR_+, \CC^2) :=
	\left\{
	\varphi\in  H^1(\RR_+, \CC^2)\,,\quad t\varphi\in  L^2(\RR_+, \CC^2)
	\right\}\,.
	\]
\end{definition}
Note that
\[
\mathscr{D}_{\RR^2} = \int_\RR^\oplus\mathscr{D}_{\xi,\RR}\, \dd\xi
\,\text{ and }\,
\mathscr{D}_{\RR^2_+} = \int_\RR^\oplus\mathscr{D}_{\xi,\RR_+}\dd\xi\,.
\]
Let us state the main properties of these operators.
\begin{proposition}\label{prop.fiberSA}
	Let $\xi\in \RR$. The operators $\mathscr{D}_{\xi, \RR}$ and $\mathscr{D}_{\xi, \RR_+}$ are self-adjoint with compact resolvent. 
\end{proposition}
\begin{notation}\label{not.disperCurve}
	$(\vartheta_{k, \RR}^+(\xi))_{k\geq 1}$ and $(\vartheta_{k, \RR_+}^+(\xi))_{k\geq 1}$ are the increasing sequences of the non-negative eigenvalues of $\mathscr{D}_{\xi, \RR}$ and $\mathscr{D}_{\xi, \RR_+}$. $(-\vartheta_{k, \RR}^-(\xi))_{k\geq 1}$ and $(-\vartheta_{k, \RR_+}^-(\xi))_{k\geq 1}$ are the decreasing sequences of the negative eigenvalues of $\mathscr{D}_{\xi, \RR}$ and $\mathscr{D}_{\xi, \RR_+}$. The eigenvalues are counted according to their multiplicity.
\end{notation}
Let us present the properties of the dispersion curves.
\begin{theorem}\label{thm.dispertionCurve}
	Let $k\geq 1$. We have
	\begin{enumerate}[\rm (i)]
		\item \label{pt.harmonicOscillator}$\vartheta_{k, \RR}^+(\xi) = \sqrt{2(k-1)}$ and $\vartheta_{k, \RR}^-(\xi) = \sqrt{2k}$, for all $\xi\in \RR$,
		\item \label{pt.harmonicOscillator2}$\xi\mapsto \vartheta_{k, \RR_+}^+(\xi)$ is a regular increasing function with no critical point such that
		\[
		\lim_{\xi\to-\infty}\vartheta_{k, \RR_+}^+(\xi) = \sqrt{2(k-1)}\, 
		\mbox{ and }
		\lim_{\xi\to+\infty}\vartheta_{k, \RR_+}^+(\xi) = +\infty\,,
		\]
		\item  \label{pt.harmonicOscillator3} $\xi\mapsto \vartheta_{k, \RR_+}^-(\xi)$ is a regular function with a single critical point $\xi_k$, that decreases on $(-\infty, \xi_k)$ and increases on $(\xi_k,+\infty)$ and such that
		\[
		\lim_{\xi\to-\infty}\vartheta_{k, \RR_+}^-(\xi) = +\infty\, 
		\mbox{ and }
		\lim_{\xi\to+\infty}\vartheta_{k, \RR_+}^-(\xi) = \sqrt{2k}\,.
		\]
		Moreover, $\xi_1=a_0$ where $0<a_0 := \min \vartheta_{1, \RR_+}^- = \vartheta_{1, \RR_+}^-(\xi_1)<\sqrt{2}$,  and
		\begin{equation}\label{eq.derivseconde}
		\pa_\xi^2\vartheta_{1}^-(\xi_1) 
		= \frac{2 a_0u^2_{a_0,a_0}(0)}{
			2a_0-u^2_{a_0,a_0}(0)
		}>0\,,\end{equation}	
		where $u_{\alpha,\xi}$ is defined in Notation \ref{not.negeigconst}.
	\end{enumerate}
\end{theorem}
%

Theorem \ref{thm.homogeneousOp} follows from Proposition \ref{prop.fiberSA}, Theorem \ref{thm.dispertionCurve}, Theorem XIII.86 and Theorem XIII.85 of \cite{reed-4}.
\subsection{Proof of Proposition \ref{prop.fiberSA}}
Let $\xi\in \RR$. 
Let us prove the proposition for $\mathscr{D}_{\xi, \RR_+}$. The proof for $\mathscr{D}_{\xi, \RR}$ follows the same line.
Since $\sigma_1 \xi$ is a bounded and symmetric matrix, $\mathscr{D}_{\xi, \RR_+}$ is self-adjoint if and only if $\mathscr{D}_{0, \RR_+}$  is self-adjoint.
Let $\mathscr{D}$ be the operator acting as $-i\sigma_2\pa_t+\sigma_1 t$ on 
\[
\mathsf{Dom}(\mathscr{D}) = \left\{
\varphi\in  H^1_0(\RR_+, \CC^2)\,,\quad t\varphi\in  L^2(\RR_+, \CC^2)
\right\}\,.
\]
By the anti-commutation relations of the Pauli matrices, the operator $\mathscr{D}$ is symmetric and for all $\varphi \in \mathsf{Dom}(\mathscr{D})$,
\[
\begin{split}
\|\mathscr{D}\varphi\|^2_{\RR_+}
&= \|\pa_t\varphi\|^2_{\RR_+} + \|t\varphi\|^2_{\RR_+} + 2\RE\langle -i\sigma_2\pa_t\varphi, \sigma_1t\varphi \rangle
\\
&=
\|\pa_t\varphi\|^2_{\RR_+} + \|t\varphi\|^2_{\RR_+} + \int_0^{+\infty} t\pa_t\langle\varphi,\sigma_3\varphi\rangle_{\CC^2}\dd t
\\
&=
\|\pa_t\varphi\|^2_{\RR_+} + \|t\varphi\|^2_{\RR_+} 
- \langle\varphi,\sigma_3\varphi\rangle
\,.
\end{split}
\]
Hence, $\mathscr{D}$ is closed. The adjoint $\mathscr{D}^*$ acts as $-i\sigma_2\pa_t+\sigma_1 t$ on 
\[
\mathsf{Dom}(\mathscr{D}^*) = \left\{
\varphi\in  L^2(\RR_+, \CC^2)\,,\quad(-i\sigma_2\pa_t+\sigma_1 t)\varphi\in  L^2(\RR_+, \CC^2)
\right\}\,.
\]
Studying the deficiency indices, we consider a solution $\varphi\in L^2(\RR_+,\CC^2)$ of
\[
\mathscr{D}^*\varphi = i\varphi\,.
\]
We get
\[
(\mathscr{D}^*)^2\varphi = \left(-\pa_t^2 +t^2-\sigma_3\right)\varphi=-\varphi\,.
\]
By \cite[Eq. 12.2.2, 12.8.2, 12.8.3]{NIST:DLMF}, we deduce that the only $L^2$-solutions of $\mathscr{D}^*\varphi = i\varphi$ are of the form
\[
t\mapsto a\begin{pmatrix}U(0,\sqrt{2}t)\\\frac{i}{\sqrt{2}}U(1,\sqrt{2}t) \end{pmatrix} :=a\varphi_+ \,,
\] 
where $a\in \CC$ and $U(0,\cdot)$, $U(1,\cdot)$ are parabolic cylinder functions. Similarly, the only $L^2$-solutions of $\mathscr{D}^*\varphi = -i\varphi$ are of the form
\[
t\mapsto b\begin{pmatrix}U(0,\sqrt{2}t)\\\frac{-i}{\sqrt{2}}U(1,\sqrt{2}t) \end{pmatrix}:=b\varphi_-\,,
\] 
where $b\in \CC$.
By  \cite[Theorem X.2]{reed-2}, there is a one-to-one correspondance between the set of self-adjoint extensions of $\mathscr{D}$ and the circle $\{e^{i\theta},\quad \theta\in \RR \}$. The corresponding operators $\mathscr{D}^\theta$ act as $-i\sigma_2\pa_t+\sigma_1 t$ on
\[
\mathsf{Dom}(\mathscr{D}^\theta) = \left\{\varphi+a(\varphi_++e^{i\theta}\varphi_-)\,,\quad
\varphi\in  \mathsf{Dom}(\mathscr{D})\,,a\in\CC
\right\}\,.
\]
We have  $\mathscr{D}_{0, \RR_+}=\mathscr{D}^{\theta_0}$ where $\theta_0\in(-\pi,\pi)$ is the unique solution on $(-\pi,\pi]$ of 
\[\frac{\sqrt{2}U(0,0)}{U(1,0)}=\tan\left(\frac\theta 2\right)\,.\]
Indeed, $\theta_0$ is the unique $\theta\in(-\pi,\pi]$ such that the infinite-mass boundary condition 
\[\sigma_1(\varphi_++e^{i\theta}\varphi_-)(0)=\varphi_+(0)+e^{i\theta}\varphi_-(0)\,,\]
holds.
This means that
\[U(1,0)\frac{i}{\sqrt{2}}\left(1-e^{i\theta}\right)=U(0,0)(1+e^{i\theta})\,,\]
so that $\theta\neq\pi$ (since $U(1,0)\neq 0$) and
\[\frac{\sqrt{2}U(0,0)}{U(1,0)}=i\frac{1-e^{i\theta}}{1+e^{i\theta}}=i\frac{-2i\sin\left(\frac\theta 2\right)}{2\cos\left(\frac\theta 2\right)}=\tan\left(\frac\theta 2\right)\,.\]
This ends the proof of Proposition \ref{prop.fiberSA}.
\subsection{Min-max characterization of the eigenvalues of the fibered operators}
The arguments for the min-max principle of Section \ref{sec.NLminmax} apply to the homogeneous magnetic Dirac operators. This leads to a family of Schrödinger operators with parameter dependent boundary conditions (see Notation \ref{not.quadminmax}).
\begin{notation}\label{not.quadformfiberOp}
	Let $\alpha>0, \xi\in \RR$. We introduce
	\[\begin{split}
	q_{\RR,\xi}^\pm(u)
	=\int_{\mathbb{R}} |(\xi\pm t+\partial_{t}) u|^2\dd t 
	=\int_{\mathbb{R}} \left(|\partial_{t} u|^2 + |(\xi\pm t)u|^2\mp |u|^2\right)\dd t \,,
	\end{split}\]
	for all $u\in B^1(\RR,\CC)$ and
	\[\begin{split}
	q_{\RR_+, \alpha,\xi}^\pm(v)
	&=\int_{\mathbb{R}_+} |(\xi\pm t+\partial_{t}) v|^2\dd t+\alpha|v(0)|^2 
	\\&=\int_{\mathbb{R}_+} \left(|\partial_{t} v|^2 + |(\xi\pm t)v|^2\mp |v|^2\right)\dd t+(\alpha-\xi)|v(0)|^2  \,,
	\end{split}\]
	for all $v\in B^1(\RR_+,\CC)$.
\end{notation}
These quadratic forms are closed and non-negative. The associated operators 
\[\mathscr{M}_{\RR,\xi}^\pm=-\partial^2_t+(t\pm\xi)^2\mp1\]
and
\[\mathscr{M}_{\RR_+,\alpha,\xi}^\pm=-\partial^2_t+(t\pm\xi)^2\mp1\]
with the boundary condition $\varphi'(0)=(\alpha-\xi)\varphi(0)$ are self-adjoint with compact resolvent.
%
Note also that the family $(\mathscr{M}_{\RR_+,\alpha,\xi}^\pm)_{\alpha>0,\xi\in\mathbb{R}}$ is of type $(B)$ in the sense of Kato:
\begin{enumerate}[\rm (i)]
	\item $\mathrm{Dom}(q_{\RR_+, \alpha,\xi}^\pm)=B^1(\mathbb{R}_+)$ does not depend on $\xi$ or $\alpha$,
	\item for all $u\in B^1(\mathbb{R}_+)$, $(\alpha,\xi)\mapsto q_{\RR_+, \alpha,\xi}^\pm(u)$ is analytic.
\end{enumerate}

\begin{remark}
	For $\alpha>0$, the operator $\mathscr{M}_{\RR_+,\alpha,\alpha}^\pm\pm 1$ coincides with the famous de Gennes operator (see \cite{FH11}).
\end{remark}
\begin{notation}\label{not.dispersionCurves}
	Let $\alpha>0, \xi\in \RR$, $k\in \NN\setminus\{0\}$. We introduce
	\[
	\begin{split}
	\nu_{\RR,k}^\pm(\xi) 
	= \inf_{\tiny
		\begin{array}{c}
		V\subset B^1(\RR)
		\\
		\dim V = k
		\end{array}
	}
	\sup_{u\in V\setminus\{0\}}
	\frac{q_{\RR,\xi}^\pm(u)}{\|u\|_{\RR}^2}\,,
	\end{split}
	\]
	and
	\[
	\begin{split}
	\nu_{\RR_+,k}^\pm(\alpha,\xi) 
	= \inf_{\tiny
		\begin{array}{c}
		V\subset B^1(\RR_+)
		\\
		\dim V = k
		\end{array}
	}
	\sup_{u\in V\setminus\{0\}}
	\frac{q_{\RR_+,\alpha,\xi}^\pm(u)}{\|u\|_{\RR_+}^2}\,.
	\end{split}
	\]
\end{notation}
	The arguments of Section \ref{sec.NLminmax} can be easily adapted to this setting and imply that
	\[
	\vartheta_{k, \RR}^\pm(\xi) = \sqrt{\nu_{\RR,k}^\pm(\xi)}\,,
	\]
	and $\vartheta_{k, \RR_+}^\pm(\xi)$ is the only solution $\alpha>0$ of
	\[
	\nu_{\RR_+,k}^\pm(\alpha,\xi) = \alpha^2\,,
	\]
	where $\vartheta_{k, \RR}^\pm(\xi)$ and $\vartheta_{k, \RR_+}^\pm(\xi)$ are defined in Notation \ref{not.disperCurve}.
By translation invariance, we have
\[
\nu_{\RR,k}^\pm(\xi) = \nu_{\RR,k}^\pm(0)\,\mbox{ for }k\geq 1 \mbox{ and }\xi\in \RR\,.
\]
These are related with the eigenvalues of the harmonic oscillator and Point \eqref{pt.harmonicOscillator} of Theorem \ref{thm.dispertionCurve} follows.
\subsection{About the dispersion curves $\nu_{\RR_+,k}^\pm$}
In this section, we prove the following proposition concerning the dispersion curves $\nu_{\RR_+,k}^\pm$.
\begin{proposition}\label{prop.dispersionCurveNu}
	Let $\alpha>0$ and $k\geq 1$. 
	\begin{enumerate}[\rm (i)]
		\item The function $\xi\mapsto \nu^+_{\RR_+,k}(\alpha, \xi)$ is smooth, has no critical point, is increasing, and tends to $+\infty$ as $\xi\to+\infty$ and to $2(k-1)$ as $\xi\to-\infty$.
		\item The function $\xi\mapsto \nu^-_{\RR_+,k}(\alpha,\xi)$ is smooth, has a unique critical point, which is a minimum $\xi_\alpha$, decreases on $(-\infty,\xi_\alpha)$ and increases on $(\xi_\alpha,+\infty)$, tends to $+\infty$ as $\xi\to-\infty$ and to $2k$ as $\xi\to+\infty$. Moreover, $\nu^-_{\RR_+,1}(\alpha, \xi_\alpha)<2$.
	\end{enumerate}
\end{proposition}
\begin{remark}
	Actually, to prove $\nu^-_{\RR_+,1}(\alpha, \xi_\alpha)<2$,	one could avoid the asymptotic analysis by using the knowledge of the de Gennes function. Consider $\xi=\alpha>0$. Then,
	\[\nu_{\RR_+,1}^-(\alpha,\alpha)=\mu(\alpha)+1\,,\]
	where $\mu$ is the celebrated de Gennes function. We know that, on $\mathbb{R}_+$, $\mu<1$. Thus, for all $\alpha>0$,
	\[\nu_{\RR_+,1}^-(\alpha,\xi_\alpha)\leq\mu(\alpha)+1<2=(\sqrt{2})^2\,.\]
\end{remark}
Let $\alpha>0$ and $k\geq 1$ be fixed.
In this part, we remove the subscript $\RR_+$ for the sake of notation simplicity.
By the analytic pertubation theory, we know that $\nu_k^\pm(\cdot,\cdot)$ are analytic functions of $\alpha$ and $\xi$. 
Let $\xi\mapsto \nu^\pm(\alpha,\xi)$ be an analytic branch of eigenvalues of $\mathscr{M}_{\alpha,\xi}^\pm$ and $u_{\alpha,\xi}^\pm$ a corresponding normalized eigenfunction.

The following elementary lemma will be used many times in this section.
\begin{lemma}\label{lem.ipp}
	We have
	\[\langle -\psi'',\varphi\rangle=\langle\psi,-\varphi''\rangle+\psi'(0)\varphi(0)-\psi(0)\varphi'(0)\,.\]	
\end{lemma}
\begin{proof}
	We have
	\[\langle -\psi'',\varphi\rangle=\langle\psi',\varphi'\rangle+\psi'(0)\varphi(0)=-\langle\psi,\varphi''\rangle+\psi'(0)\varphi(0)-\psi(0)\varphi'(0)\,.\]
\end{proof}
In the following lemmas, we compute derivatives of $\nu^\pm$ with respect to $\alpha$ and $\xi$.
\begin{lemma}\label{lem.nu'}
	We have
	\[\partial_\xi\nu^\pm(\alpha,\xi)=\int_0^{+\infty}2(\xi\pm t) u^2_{\alpha,\xi}(t)\dd t-u_{\alpha,\xi}(0)^2\,,\]
	and
	\[(\mathscr{M}_ {\alpha,\xi}^\pm-\nu^\pm)v_{\alpha,\xi}=(\partial_\xi\nu^\pm) u_{\alpha,\xi}-2(\xi\pm t)u_{\alpha,\xi}\,,\quad v_{\alpha,\xi}=\partial_\xi u_{\alpha,\xi}\,,\]
	with
	\[(\partial_t+\xi-\alpha)v_{\alpha,\xi}(0)=-u_{\alpha,\xi}(0)\,.\]	
\end{lemma}

\begin{proof}
	We have
	\[(\mathscr{M}^\pm_{\alpha,\xi}-\nu^\pm)u_{\alpha,\xi}=0\,.\]
	Then,
	\[(\mathscr{M}^\pm_{\alpha,\xi}-\nu^\pm)\partial_\xi u_{\alpha,\xi}+\partial_{\xi}\mathscr{M}^\pm_{\alpha,\xi}u_{\alpha,\xi}=\partial_\xi\nu^\pm(\xi)u_{\alpha,\xi}\,,\]
	so that
	\[\langle(\mathscr{M}^\pm_{\alpha,\xi}-\nu^\pm)\partial_\xi u_{\alpha,\xi},u_{\alpha,\xi}\rangle+2\int_0^{+\infty}(\xi\pm t)u^2_{\alpha,\xi}(t)\dd t=\partial_\xi\nu^\pm(\xi)\,.\]
	By integrations by parts, we have
	\[\langle(\mathscr{M}^\pm_{\alpha,\xi}-\nu^\pm)\partial_\xi u_{\alpha,\xi},u_{\alpha,\xi}\rangle=\partial_t\partial_\xi u_{\alpha,\xi}(0)u_{\alpha,\xi}(0)-\partial_\xi u_{\alpha,\xi}(0)\partial_tu_{\alpha,\xi}(0)\,.\]
	Note that
	\[\partial_\xi\partial_tu_{\alpha,\xi}(0)=-u_{\alpha,\xi}(0)+(\alpha-\xi)\partial_\xi u_{\alpha,\xi}(0)\,.\]
	Thus,
	\[\langle(\mathscr{M}^\pm_{\alpha,\xi}-\nu^\pm)\partial_\xi u_{\alpha,\xi},u_{\alpha,\xi}\rangle=-u_{\alpha,\xi}(0)^2\,.\]
	
\end{proof}
In the following lemma, we focus on the second derivative in $\xi$. 
\begin{lemma}\label{lem.wxi}
	We have
	\[(\mathscr{M}^\pm_ {\alpha,\xi}-\nu^\pm)w_{\alpha,\xi}=2(\partial_\xi\nu^\pm) v_{\alpha,\xi}+\partial^2_\xi \nu^\pm u_{\alpha,\xi}-4(\xi\pm t)v_{\alpha,\xi}-2u_{\alpha,\xi}\,,\quad w_{\alpha,\xi}=\partial^2_\xi u_{\alpha,\xi}\,,\]
	with
	\[(\partial_t+\xi-\alpha)w_{\alpha,\xi}(0)=-2v_{\alpha,\xi}(0)\,.\]
	Moreover,
	\[-\frac{\partial^2_\xi\nu}{2}+1+2\langle(\xi\pm t)v_{\alpha,\xi},u_{\alpha,\xi}\rangle=u_{\alpha,\xi}(0)v_{\alpha,\xi}(0)\,.\]	
\end{lemma}
\begin{proof}
	We have
	\[\langle(\mathscr{M}^\pm_ {\alpha,\xi}-\nu^\pm)w_{\alpha,\xi},u_{\alpha,\xi}\rangle=\partial_t w_{\alpha,\xi}(0)u_{\alpha,\xi}(0)-w_{\alpha,\xi}(0)\partial_t u_{\alpha,\xi}(0)=-2 u_{\alpha,\xi}(0)v_{\alpha,\xi}(0)\,.\]
	In addition,
	\[\langle(\mathscr{M}^\pm_ {\alpha,\xi}-\nu^\pm)w_{\alpha,\xi},u_{\alpha,\xi}\rangle=\partial^2_\xi\nu^\pm-2-4\langle(\xi\pm t)v_{\alpha,\xi},u_{\alpha,\xi}\rangle\,.\]
\end{proof}

In the next lemma, we explicitely use the $\alpha$-dependence of the eigenfunctions.
\begin{lemma}\label{lem.nu'a}
	We have
	\[\partial_\alpha\nu^\pm(\alpha,\xi)=u^2_{\alpha,\xi}(0)\,,\]
	and
	\[(\mathscr{M}^\pm_{\alpha,\xi}-\nu^\pm(\alpha,\xi))\partial_\alpha u_{\alpha,\xi}=\partial_\alpha\nu^\pm(\alpha,\xi)u_{\alpha,\xi}\,,\]
	with
	\[\partial_\alpha\partial_t u_{\alpha,\xi}(0)=u_{\alpha,\xi}(0)+(\alpha-\xi)\partial_{\alpha}u_{\alpha,\xi}(0)\,.\]	
\end{lemma}
\begin{proof}
	We get
	\[\partial_\alpha\nu^\pm(\alpha,\xi)=\langle(\mathscr{M}^\pm_{\alpha,\xi}-\nu^\pm(\alpha,\xi))\partial_\alpha u_{\alpha,\xi},u_{\alpha,\xi}\rangle=\partial_t\partial_\alpha u_{\alpha,\xi}(0)u_{\alpha,\xi}(0)-\partial_\alpha u_{\alpha,\xi}(0)\partial_t u_{\alpha,\xi}(0)\,,\]
	and the conclusion follows.
\end{proof}

A consequence (which will be used later) of the previous lemma 
 is the following.
\begin{lemma}\label{lem.xi-t}
	We have
	\[
		\partial^2_{\alpha}\nu^\pm(\alpha,\xi)=-2 u_{\alpha,\xi}(0)v_{\alpha,\xi}(0)+4\langle(\xi\pm t) u_{\alpha,\xi},\partial_\alpha u_{\alpha,\xi}\rangle\,,
		 \quad v_{\alpha,\xi} := \pa_\xi u_{\alpha,\xi}\,.\]
\end{lemma}
\begin{proof}
	Thanks to Lemma \ref{lem.nu'}, and using that $\langle u_{\alpha,\xi},\partial_\alpha u_{\alpha,\xi}\rangle=0$, we have
	\[\langle(\mathscr{M}^\pm_ {\alpha,\xi}-\nu^\pm)v_{\alpha,\xi},\partial_\alpha u_{\alpha,\xi}\rangle=-2\langle (\xi\pm t)u_{\alpha,\xi},\partial_\alpha u_{\alpha,\xi}\rangle\,.\]
	Integrating by parts, we get using Lemmas \ref{lem.nu'} and \ref{lem.nu'a},
	\[\begin{split}
	-2\langle& (\xi\pm t)u_{\alpha,\xi},\partial_\alpha u_{\alpha,\xi}\rangle
	=\partial_t v_{\alpha,\xi}(0)\partial_\alpha u_{\alpha,\xi}(0)-v_{\alpha,\xi}(0)\partial_t\partial_\alpha u_{\alpha,\xi}(0)\\
	&=((\alpha-\xi)v_{\alpha,\xi}(0)-u_{\alpha,\xi}(0))\partial_\alpha u_{\alpha,\xi}(0)
	-v_{\alpha,\xi}(0)((\alpha-\xi)\partial_\alpha u_{\alpha,\xi}(0)+u_{\alpha,\xi}(0))\\
	&=-u_{\alpha,\xi}(0)\partial_\alpha u_{\alpha,\xi}(0)-u_{\alpha,\xi}(0)v_{\alpha,\xi}(0)\,.
	\end{split}\]
	Lemma \ref{lem.nu'a} gives
	\[2u_{\alpha,\xi}(0)\partial_\alpha u_{\alpha,\xi}(0)=\partial^2_\alpha\nu^\pm(\alpha,\xi)\,,\]
	and the conclusion follows.
\end{proof}
Let us study the critical points of $\xi\mapsto \nu^\pm(\alpha,\xi)$. Formula \eqref{eq.C1} below and its proof are reminiscent of \cite{DH93} (see also \cite[Section 3.2]{FH11} in relation with the de Gennes operator).
\begin{proposition}\label{prop.C3}
	We have
	\begin{equation}\label{eq.C1}
	\partial_\xi\nu^\pm(\alpha,\xi)=\pm	\left(\nu^\pm(\alpha,\xi)+\alpha^2-2\alpha\xi \right)u^2_{\alpha,\xi}(0)\,.
	\end{equation}
	In particular, if $\xi_\alpha$ is a critical point of $\nu^\pm(\alpha,\cdot)$, we have
	\begin{equation}\label{eq.C2}
	\nu^\pm(\alpha,\xi_\alpha)=-\alpha^2+2\alpha\xi_\alpha\,.
	\end{equation}	
	Moreover,
	\begin{equation}\label{eq.C3}
	\partial^2_\xi\nu^\pm(\alpha,\xi_\alpha)=\mp 2\alpha u^2_{\alpha,\xi_\alpha}(0)\,.
	\end{equation}
	All the critical points of $\nu^-$ are local non-degenerate minima, and all the critical points of $\nu^+$ are local non-degenerate maxima. 
	In particular, there is at most one critical point. If such a point exists for $\nu^+(\alpha,\cdot)$, then $\nu^+(\alpha,\cdot)$ is bounded from above. 
\end{proposition}
\begin{proof}
	With Lemma \ref{lem.nu'}, we get
	\[\begin{split}
	\partial_\xi\nu^\pm(\alpha,\xi)&=\pm\int_0^{+\infty}\partial_t[(\xi\pm t)^2] u^2_{\alpha,\xi}(t)\dd t-u_{\alpha,\xi}(0)^2\\
	&=\mp 2\int_0^{+\infty}(\xi-t)^2 u_{\alpha,\xi}(t)u'_{\alpha,\xi}(t)\dd t\mp \xi^2 u_{\alpha,\xi}(0)^2-u_{\alpha,\xi}(0)^2\\
	&=\mp 2 \int_0^{+\infty}(u''_{\alpha,\xi}(t)+(\nu^\pm(\alpha,\xi)\pm 1)u_{\alpha,\xi})u'_{\alpha,\xi}(t)\dd t\mp \xi^2 u_{\alpha,\xi}(0)^2-u_{\alpha,\xi}(0)^2\\
	&=\mp\int_0^{+\infty}\partial_t\left((u'_{\alpha,\xi})^2+(\nu^\pm(\alpha,\xi)\pm 1) u^2_{\alpha,\xi}\right)\dd t\mp\xi^2 u_{\alpha,\xi}(0)^2-u_{\alpha,\xi}(0)^2\\
	&=\left(\pm(\nu^\pm(\alpha,\xi)\pm 1)\pm(\alpha-\xi)^2\mp\xi^2-1 \right)u_{\alpha,\xi}(0)^2\\
	&=\left(\pm\nu^\pm(\alpha,\xi)\pm\alpha^2\mp 2\alpha\xi \right)u_{\alpha,\xi}(0)^2\,.
	\end{split}\]	
	We get \eqref{eq.C2}. Taking the derivative of \eqref{eq.C2}, we deduce \eqref{eq.C3}. The last sentence follows from \eqref{eq.C2} and \eqref{eq.C3}.
\end{proof}
Next, we show that $\nu^-(\alpha,\cdot)$ has always a critical point.
\begin{corollary}\label{cor.C4}
	For all $j\geq 1$, the function $\nu^-_j(\alpha,\cdot)$ has a unique critical point $\xi^-_{j,\alpha}$, and it is a non-degenerate minimum. The function $\xi\mapsto \nu^-_j(\alpha,\xi)$ is decreasing on $(-\infty,\xi^-_{j,\alpha})$ and is increasing on $(\xi^-_{j,\alpha},+\infty)$.	
\end{corollary}

\begin{proof} 
	If $\nu^-_j(\alpha,\cdot)$ has no critical points, then it is decreasing (it is decreasing on $(-\infty,0)$ by Proposition \ref{prop.C3}). From Proposition \ref{prop.C3}, we deduce that, for all $\xi\geq 0$,
	\[-\nu^-_j(\alpha,\xi)-\alpha^2+2\alpha\xi\leq 0\,,\]
	and that $\lim_{\xi\to+\infty}\nu^-_j(\alpha,\xi)=+\infty$. This is in contradiction with the function being decreasing. This shows that $\nu^-_j(\alpha,\cdot)$ has a unique critical point. It is a \emph{local} non-degenerate minimum. Since there is only one critical point, this shows that it is a global minimum.
\end{proof}

Let us study the asymptotic behavior of $\nu_k^\pm$.
\begin{lemma}\label{lem.asympCurveNu}
	Let $k\geq 1$.
	We have that 
	\[
	\lim_{\xi\to+\infty}\nu_k^+(\alpha,\xi) = \lim_{\xi\to-\infty}\nu_k^-(\alpha,\xi) = +\infty\,,
	\]
	and
	\[
	\lim_{\xi\to-\infty}\nu_k^+(\alpha,\xi) = 2(k-1)\,,\quad \lim_{\xi\to+\infty}\nu_k^-(\alpha,\xi) =2k\,.
	\]
\end{lemma}
\begin{proof}
	Let us first remark that for $\xi<0$ and $u\in B^1(\RR_+)$, we have
	\[
	q_{\alpha,\xi}^-(u)
	\geq (\xi^2-1)\|u\|_{\RR_+}^2\,,
	\]
	and
	\[
	\lim_{\xi\to-\infty}\nu_k^-(\alpha,\xi) = +\infty\,.
	\]
	For $\xi>0$ and $u\in B^1(\RR_+)\setminus\{0\}$, we denote by $\tau = \xi t$ and $v(\tau) = u(t)$ so that
	\[
	\begin{split}
	\frac{q_{\alpha,\xi}^+(u)}{\|u\|_{\RR_+}^2}
	&=\frac{
		\int_{\mathbb{R}_+} \left(|\partial_{t} u|^2 + |(\xi+ t)u|^2- |u|^2\right)\dd t+(\alpha-\xi)|u(0)|^2
	}{\|u\|_{\RR_+}^2}
	\\&=
	\frac{
		\int_{\mathbb{R}_+} \left(\xi^2|\partial_{\tau} v|^2 + |(\xi+ \frac{\tau}{\xi})v|^2- |v|^2\right)\dd \tau+\xi(\alpha-\xi)|v(0)|^2
	}{\|v\|_{\RR_+}^2}
	\\&=\xi^2\frac{
		\int_{\mathbb{R}_+} \left(|\partial_{\tau} v|^2 + |(1+ \frac{\tau}{\xi^2})v|^2- \xi^{-2}|v|^2\right)\dd \tau+(\frac{\alpha}{\xi}-1)|v(0)|^2
	}{\|v\|_{\RR_+}^2}
	\\&\geq\xi^2\left(1- \xi^{-2}+
	\frac{
		\int_{\mathbb{R}_+} |\partial_{\tau} v|^2\dd \tau+(\frac{\alpha}{\xi}-1)|v(0)|^2
	}{\|v\|_{\RR_+}^2}
	\right)\,.
	\end{split}
	\]
	Let us consider the Robin Laplacian associated with the quadratic form 
	\[v\mapsto\int_{\mathbb{R}_+} |\partial_{\tau} v|^2\dd \tau+\left(\frac{\alpha}{\xi}-1\right)|v(0)|^2\,.\] 
	Its essential spectrum is $[0,+\infty)$ and the only point in the negative spectrum is the eigenvalue $-(1-\alpha/\xi)^2$ whose eigenspace is spanned by $\tau\mapsto e^{-(1-\alpha/\xi)\tau}$.  
	We get 
	\[
	\frac{q_{\alpha,\xi}^+(u)}{\|u\|_{\RR_+}^2}
	\geq 
	\xi^2\left(1- \xi^{-2}-(1-\alpha/\xi)^2\right) 
	=2\alpha\xi - (1+\alpha^2)\,,
	\]
	and thus
	\[
	\lim_{\xi\to+\infty}\nu_k^+(\alpha,\xi) = + \infty\,.
	\]
	Let $\xi<0$. Let us consider $u\in B^1(\RR_+)$, we define $u(t) = v(\xi+t)$
	\[
	\begin{split}
	\frac{q_{\alpha,\xi}^+(u)}{\|u\|_{\RR_+}^2}
	&=\frac{\int_{\mathbb{R}_+} \left(|\partial_{t} u|^2 + |(\xi+ t)u|^2- |u|^2\right)\dd t+(\alpha-\xi)|u(0)|^2}{\|u\|_{\RR_+}^2} 
	\\&=\frac{\int_{\xi}^{+\infty} \left(|\partial_{\tau} v|^2 + |\tau v|^2- |v|^2\right)\dd t+(\alpha-\xi)|v(\xi)|^2}{\|v\|_{(\xi,+\infty)}^2}
	\,.
	\end{split}\]
	Using truncated Hermite's functions as test functions, we get that
	\begin{equation}\label{eq.upperNuPlus}
	\limsup_{\xi\to-\infty} \nu_{k}^+(\alpha,\xi) \leq2(k-1)\,.
	\end{equation}
	Since
	\[
	\begin{split}
	\frac{q_{\alpha,\xi}^+(u)}{\|u\|_{\RR_+}^2}
	\geq\frac{\int_{\xi}^{+\infty} \left(|\partial_{\tau} v|^2 + |\tau v|^2- |v|^2\right)\dd t}{\|v\|_{(\xi,+\infty)}^2}
	\,,
	\end{split}\]
	the eigenvalue $\nu_{k}^+(\alpha,\xi)$ is larger than the $k$-th eigenvalue of the operator $-\partial^2_\tau+\tau^2-1$ with Neumann boundary condition on $(\xi,+\infty)$. In other words,
	\[
	\nu_{k}^+(\alpha,\xi)\geq \mu_k(\xi)-1\,,
	\]
	where $\mu_k$ is the $k$-th dispersion curve of the de Gennes operator.
	It is well-known that \footnote{see, for instance, \cite[Prop. 3.2.2 \& 3.2.4]{FH11}}
	\[
	\lim_{\xi\to-\infty}\mu_k(\xi) = 2k-1\,.
	\] 
	Thus,
	\[\liminf_{\xi\to -\infty}\nu_{k}^+(\alpha,\xi) \geq2(k-1)\,.\]
	As in \eqref{eq.upperNuPlus}, using Hermite's functions, we get that
	\begin{equation}\label{eq.upperbound2nu-}
	\limsup_{\xi\to+\infty} \nu_{k}^-(\alpha,\xi) \leq2k\,.
	\end{equation}
	Let $\xi>\alpha/2$. 
	Let us consider $u,v\in B^1(\RR_+)$ such that  $u(t) = v(\xi t)$. We have
	\[
	\begin{split}
	q_{\alpha,\xi}^-(u)=&\int_{\mathbb{R}_+} \left(|\partial_{t} u|^2 + |(t-\xi)u|^2+ |u|^2\right)\dd t+(\alpha-\xi)|u(0)|^2
	\\
	=&\int_{ \alpha/4}^{+\infty} \left(|\partial_{t} u|^2 + |(t-\xi)u|^2+ |u|^2\right)\dd t
	\\
	&+\xi^{-1}\int_0^{\xi\alpha/4} \left(\xi^2|\partial_{\tau} v|^2 + |(\tau/\xi-\xi)v|^2+ |v|^2\right)\dd \tau+(\alpha-\xi)|v(0)|^2 \,,
	\end{split}\]
	and
	\[
	\begin{split}
	&\xi^{-1}\int_0^{\alpha\xi/4} \left(\xi^2|\partial_{\tau} v|^2 + |(\tau/\xi-\xi)v|^2+ |v|^2\right)\dd \tau+(\alpha-\xi)|v(0)|^2
	\\&\geq
	\xi\left(\int_0^{\alpha\xi/4}|\partial_{\tau} v|^2 \dd \tau+(\alpha/\xi-1)|v(0)|^2\right)
	+
	\xi^{-1}(1 + (\xi-\alpha/4)^2)\int_0^{\alpha\xi/4}|v|^2\dd \tau
	\,.
	\end{split}\]
	By studying the spectrum of the Robin-Neumann Laplacian  whose quadratic form is 
	\[
	v\mapsto \int_0^{\alpha\xi/4}|\partial_{\tau} v|^2 \dd \tau+(\alpha/\xi-1)|v(0)|^2\,,
	\]
	we get,
	\[
	\xi\left(\int_0^{\alpha\xi/4}|\partial_{\tau} v|^2 \dd \tau+(\alpha/\xi-1)|v(0)|^2\right)
	\geq 
	\xi(-1+\alpha/\xi + o(\xi^{-1}))\int_0^{\alpha\xi/4}|v|^2\dd \tau\,.
	\]
	%
	We deduce that
	\[\begin{split}
	\xi^{-1}\int_0^{\xi\alpha/4} &\left(\xi^2|\partial_{\tau} v|^2 + |(\tau/\xi-\xi)v|^2+ |v|^2\right)\dd \tau+(\alpha-\xi)|v(0)|^2
	\\&\geq (\alpha/2 + o(1))\int_0^{\alpha\xi/4}|v|^2\dd \tau\,,
	\end{split}\]
	so that
	\begin{equation}\label{estilowerNu-}
	q_{\alpha,\xi}^-(u)
	\geq
	\int_{ \alpha/4}^{+\infty} \left(|\partial_{t} u|^2 + |(t-\xi)u|^2+ |u|^2\right)\dd t
	+\xi(\alpha/2 + o(1))\int_0^{\alpha/4}|u|^2\dd t\,.
	\end{equation}
	Let $(u_{1,\xi},\dots, u_{k,\xi})$ be an orthonormal family of eigenfunctions of $\mathscr{M}_{\alpha,\xi}^-$ associated with the $k$ first eigenvalues.
	
	We have, for all $u\in\underset{1\leq j\leq k}{\mathrm{span}} u_{j,\xi}$,
	\begin{equation*}
	\begin{split}
	\nu_k^-(\alpha,\xi)\|u\|^2\geq q_{\alpha,\xi}^-(u)
	&\geq
	\int_{\frac\alpha 4}^{+\infty} \left(|\partial_{t} u|^2 + |(t-\xi)u|^2+ |u|^2\right)\dd t
	+\xi\left(\frac\alpha 2 + o(1)\right)\int_0^{\frac\alpha 4}|u|^2\dd t\\
	&\geq \xi\left(\frac\alpha 2 + o(1)\right)\int_0^{\alpha/4}|u|^2\dd t\,.
	\end{split}
	\end{equation*}
	
	By \eqref{estilowerNu-} and the upper bound \eqref{eq.upperbound2nu-}, we get that for $\xi$ large enough, the family of the restrictions of $u_{1,\xi},\dots, u_{k,\xi}$ to the interval $(\alpha/4,+\infty)$ is of dimension $k$ and
	\[
	(1+o(1))\nu_k^-(\alpha,\xi)\geq \inf_{\tiny\begin{array}{c}
		V\subset B^1(\alpha/4,+\infty)\\
		\dim V = k
		\end{array}}
	\sup_{\tiny\begin{array}{c}
		u\in V\\
		\|u\|_{(\alpha/4,+\infty)} = 1
		\end{array}}
	\int_{ \alpha/4}^{+\infty} \left(|\partial_{t} u|^2 + |(t-\xi)u|^2+ |u|^2\right)\dd t\,.
	\]
	In the same way as for the de Gennes operator ($\alpha=0$), we get
	\[
	\lim_{\xi\to +\infty}\nu_{k}^-(\alpha,\xi)=2k\,,
	\]
	and the conclusion follows.
\end{proof}
\noindent{\it End of the proof of Proposition \ref{prop.dispersionCurveNu}.}
By Proposition \ref{prop.C3} and Lemma \ref{lem.asympCurveNu}, we get that $\xi\mapsto \nu^+_k(\alpha,\xi)$ has no critical point. Using again Lemma \ref{lem.asympCurveNu}, we deduce that $\xi\mapsto \nu^+_k(\alpha,\xi)$ is increasing.


The behavior of $\xi\mapsto \nu^-_k(\alpha,\xi)$ is described in Corollary \ref{cor.C4}.
The monotonicity of 
$\xi\mapsto \nu^-_1(\alpha,\xi)$ and its limit in $+\infty$ ensure that $\nu^-_{1}(\alpha, \xi_\alpha)<2$.

\subsection{Proof of points \eqref{pt.harmonicOscillator2} and \eqref{pt.harmonicOscillator3} of Theorem \ref{thm.dispertionCurve}}
In this part, we remove the subscript $\RR_+$ for the sake of notation simplicity.
Let $k\geq 1$ and $\xi\in \RR$. By the min-max characterization of $\vartheta_{k}^\pm(\xi)$, we have that 
\[
\begin{split}
&0 < \nu_k^{\pm}(\alpha,\xi)-\alpha^2\,,\mbox{ for all }0<\alpha<\vartheta_{k}^\pm(\xi),
\\&
0 > \nu_k^{\pm}(\alpha,\xi)-\alpha^2\,,\mbox{ for all }\alpha>\vartheta_{k}^\pm(\xi),
\\&
0 = \nu_k^{\pm}(\vartheta_{k}^\pm(\xi),\xi)-\vartheta_{k}^\pm(\xi)^2\,.
\end{split}
\]
\subsubsection{Limits of $\vartheta_k^\pm$}
Let $\alpha>0$ be fixed.
By Proposition \ref{prop.dispersionCurveNu}, there is $M\in \RR$ such that for all $\xi>M$
\[
\nu_k^{+}(\alpha,\xi)-\alpha^2>0\,.
\]
Hence, $\alpha<\vartheta_{k}^+(\xi)$. This shows that 
\[
\lim_{\xi\to+\infty}\vartheta_{k}^+(\xi) = +\infty\,.
\]
The same kind of arguments ensure that
\[
\lim_{\xi\to-\infty}\vartheta_{k}^+(\xi) = \sqrt{2(k-1)}\, 
\]
\[
\lim_{\xi\to-\infty}\vartheta_{k}^-(\xi) = +\infty\, 
\mbox{ and }
\lim_{\xi\to+\infty}\vartheta_{k}^-(\xi) = \sqrt{2k}\,.
\]
\subsubsection{Regularity of $\vartheta^\pm_k$}
We have, for all $\xi\in\R$,
\[\nu_k^{\pm}(\vartheta_{k}^\pm(\xi),\xi)-\vartheta_{k}^\pm(\xi)^2=0\,.\]
Let us explain why $\vartheta^\pm_k$ is smooth. Consider the function
\[F(\alpha,\xi)=\nu_k^\pm(\alpha,\xi)-\alpha^2\,.\]
We have
\[\partial_\alpha F(\alpha,\xi)=\partial_\alpha\nu^\pm_k(\alpha,\xi)-2\alpha\,.\]
By Lemma \ref{lem.nu'a}, we get
\[\partial_\alpha F(\alpha,\xi)=[u^\pm_{k, \alpha,\xi}(0)]^2-2\alpha\,.\]
Let us analyze the sign of $\partial_\alpha F(\vartheta_k^\pm(\xi),\xi)$. 
Notice that $P :\alpha\mapsto q^\pm_{\alpha,\xi}(u^\pm_{k,\vartheta^\pm_k(\xi),\xi})-\alpha^2$ is a polynomial of degree $2$, which is zero at $\alpha=\vartheta^\pm_k(\xi)$. Moreover, by the min-max principle, we have
\[P(\alpha)\geq \nu^\pm_k(\alpha,\xi)-\alpha^2\,,\]
so that, for all $\alpha\in(0,\vartheta^\pm_k(\xi))$, $P(\alpha)>0$. It follows that $P'(\vartheta^\pm_k(\xi))<0$. Since
\[P'(\vartheta^\pm_k(\xi))=[u^\pm_{k, \alpha,\xi}(0)]^2-2\alpha=\partial_\alpha F(\vartheta_k^\pm(\xi),\xi)\,,\]
we get that $\partial_\alpha F(\vartheta_k^\pm(\xi),\xi)<0$. With the Implicit Function Theorem, we deduce that $\vartheta^\pm_k$ is a smooth function (since $F$ is smooth). Moreover, the derivative of the implicit function is given by the usual formula
\begin{equation}\label{eq.crit-theta}
\partial_\xi\vartheta^\pm_k(\xi)=-\frac{\partial_\xi F(\vartheta^\pm_{k}(\xi),\xi)}{\partial_\alpha F(\vartheta^\pm_{k}(\xi),\xi)}= \frac{
	\pa_\xi\nu_k^{\pm}(\vartheta_{k}^\pm(\xi),\xi)
}{
	2\vartheta_{k}^\pm(\xi)-\pa_\alpha\nu_k^{\pm}(\vartheta_{k}^\pm(\xi),\xi)
}\,,\end{equation}
where we see that the denominator of the last expression is a positive function.

\subsubsection{Critical points of $\vartheta^\pm_k$}
We deduce that $\partial_\xi\vartheta^\pm_k(\xi)$
has the sign of $\pa_\xi\nu_k^{\pm}(\vartheta_{k}^\pm(\xi),\xi)$.
Thus, by Proposition \ref{prop.dispersionCurveNu}, $\xi\mapsto \vartheta_{k}^+(\xi)$ increases and has no critical points.

Moreover, $\xi_0$ is a critical point of $\xi\mapsto \vartheta_{k}^-(\xi)$ if and only if $\xi_0$ is a (actually \emph{the}) critical point of $\xi\mapsto \nu_k^{-}(\vartheta_{k}^-(\xi_0),\xi)$.
Let $\xi_0\in \RR$ be such a critical point. By \eqref{eq.crit-theta} and Proposition \ref{prop.C3}, we have
\[
\pa_\xi^2\vartheta_{k}^-(\xi_0) 
= \frac{
	[\pa_\xi^2\nu_k^{-}](\vartheta_{k}^-(\xi_0),\xi_0)
}{
	2\vartheta_{k}^-(\xi_0)-\pa_\alpha\nu_k^{-}(\vartheta_{k}^-(\xi_0),\xi_0)
}
>0
\,.
\]
Hence, there is at most one critical point which is a non-degenerate minimum. The function $\vartheta_{k}^-$ has exactly one minimum $\xi_k$, increases on $(\xi_k,+\infty)$ and decreases on $(-\infty,\xi_k)$.
For $\alpha = \vartheta_{1}^-(\xi_1)$, we also get that $\xi_1$ is the minimum of the function
$
\xi\mapsto\nu^-_{1}(\alpha,\xi)\,.
$
By Proposition \ref{prop.C3}, 
\[
0 = \nu^-_{1}(\alpha,\xi_1)-\alpha^2 = -2\alpha^2 + 2\alpha\xi_1\,,
\]
so that $\xi_1 = \alpha = \vartheta_{1}^-(\xi_1)$. Using again Proposition \ref{prop.C3} and Lemma \ref{lem.nu'a}, we get 
\[
\pa_\xi^2\vartheta_{1}^-(\xi_1) 
= \frac{
	2\alpha u_{\alpha,\alpha}(0)^2
}{
	2\alpha-u_{\alpha,\alpha}(0)^2
}
>0
\,.
\]
%

\subsection{Numerical illustrations}\label{sec.C3}
By using naive finite difference method and dichotomy method, it is a possible to compute the eigenvalues $\nu_{\RR_+,k}^\pm(\alpha,\cdot)$ and $\vartheta_{k,\RR_+}^{\pm}$ by using a short Python script. Subfigures \subref{fig1} and \subref{fig2} of Figure \ref{figmain1} below present $\nu_{\RR_+,k}^-(\alpha,\cdot)$ in colored lines. The horizontal dashed lines represent the Landau levels and the dotted affine line of \subref{fig2} is the graph of the function $\xi\mapsto -\alpha^2+2\alpha\xi$ (see Proposition \ref{prop.C3}).

\begin{figure}[ht!]
	\centering
	\begin{subfigure}[t]{0.47\textwidth}
		\centering
		\includegraphics[width=\textwidth]{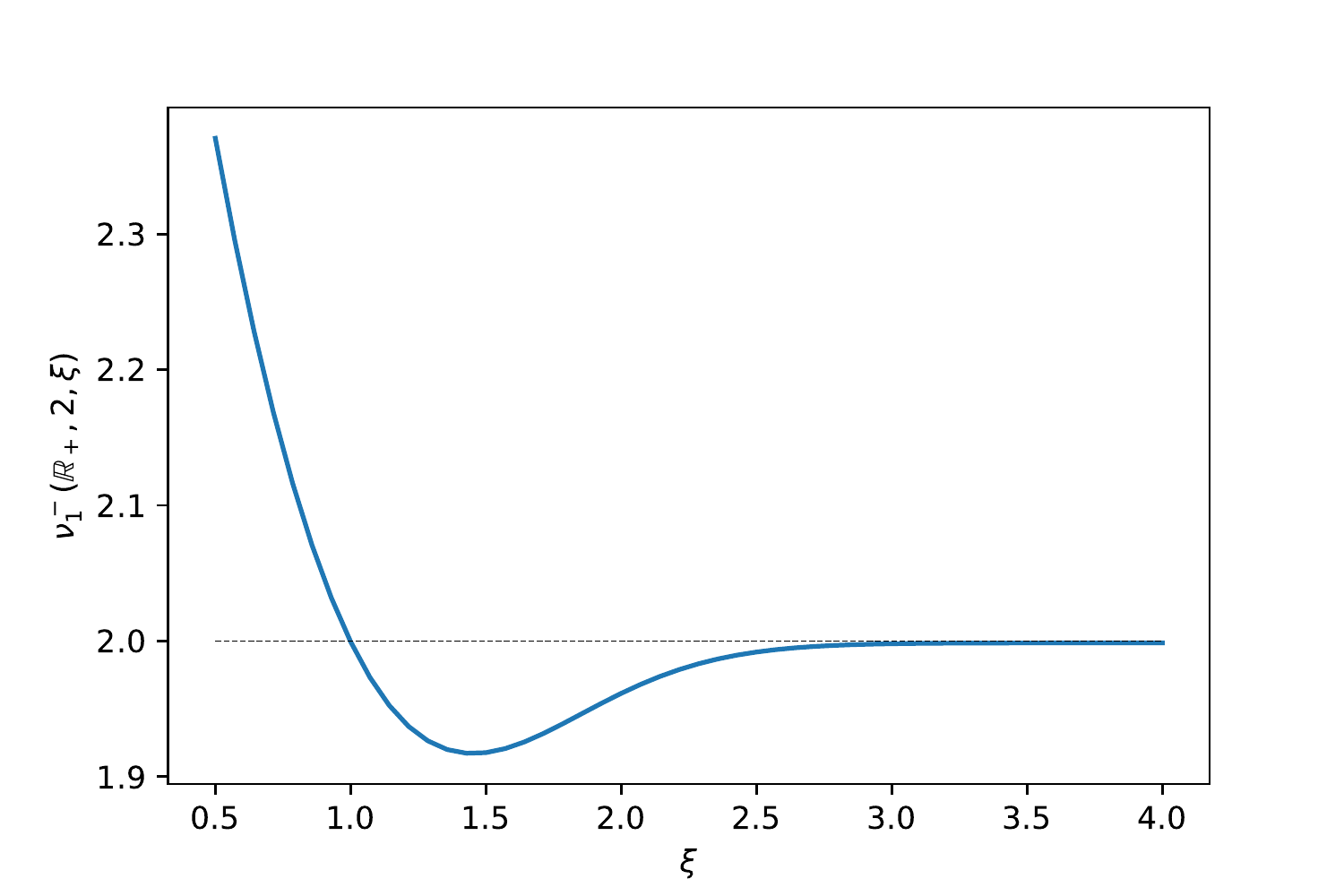}
		\caption{The function $\nu_{\RR_+,1}^-(2,\cdot)$}
		\label{fig1}
	\end{subfigure}
	\hfill
	\begin{subfigure}[t]{0.47\textwidth}
		\centering
		\includegraphics[width=\textwidth]{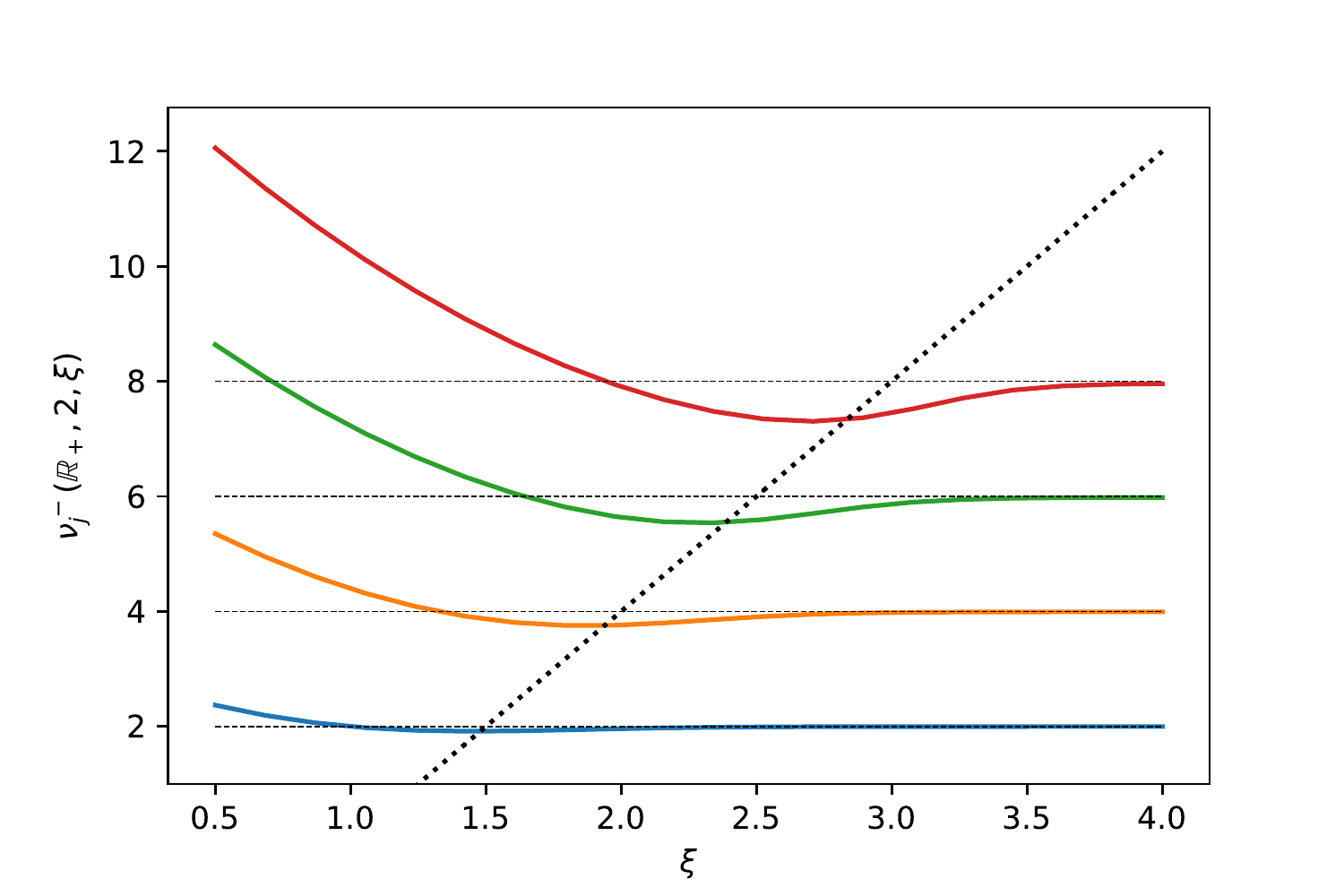}
		\caption{Functions $\nu_{\RR_+,k}^-(2,\cdot)$ and the function $\xi\mapsto -\alpha^2+2\alpha\xi$}
		\label{fig2}
	\end{subfigure}
	\caption{Dispersion curves $\nu_{\RR_+,k}^-$\label{figmain1}}
\end{figure}
Figure \ref{fig3} displays the increasing behavior of the $\nu^+_{\RR_+,j}(\alpha,\cdot)$ and the associated Landau levels $2(j-1)$ for $j\geq 1$.
\begin{figure}[ht!]
	\centering
	\includegraphics[width=0.6\textwidth]{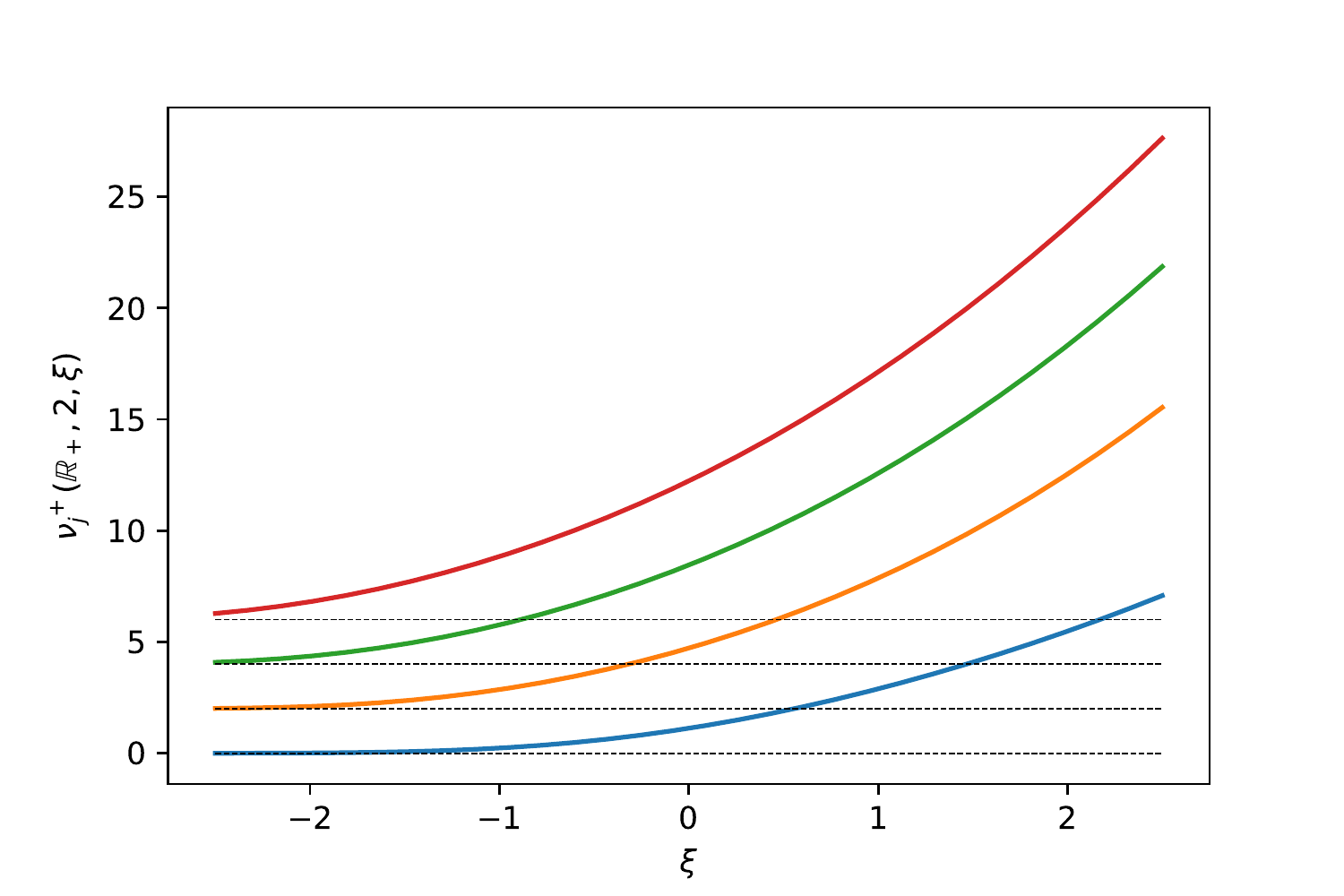}
	\caption{The functions $\nu_{\RR_+,j}^+(2,\cdot)$}\label{fig3}
\end{figure}

Figure \ref{fig4} shows the dispersion curves $\pm\vartheta^\pm_{\RR_+,j}$ representing the spectrum of the fibered Dirac operators $\mathscr{D}_{\xi, \RR_+}$ with the associated Landau levels $\pm\sqrt{2k}$ as dashed lines.

\begin{figure}[ht!]
	\centering
	\includegraphics[width=0.6\textwidth]{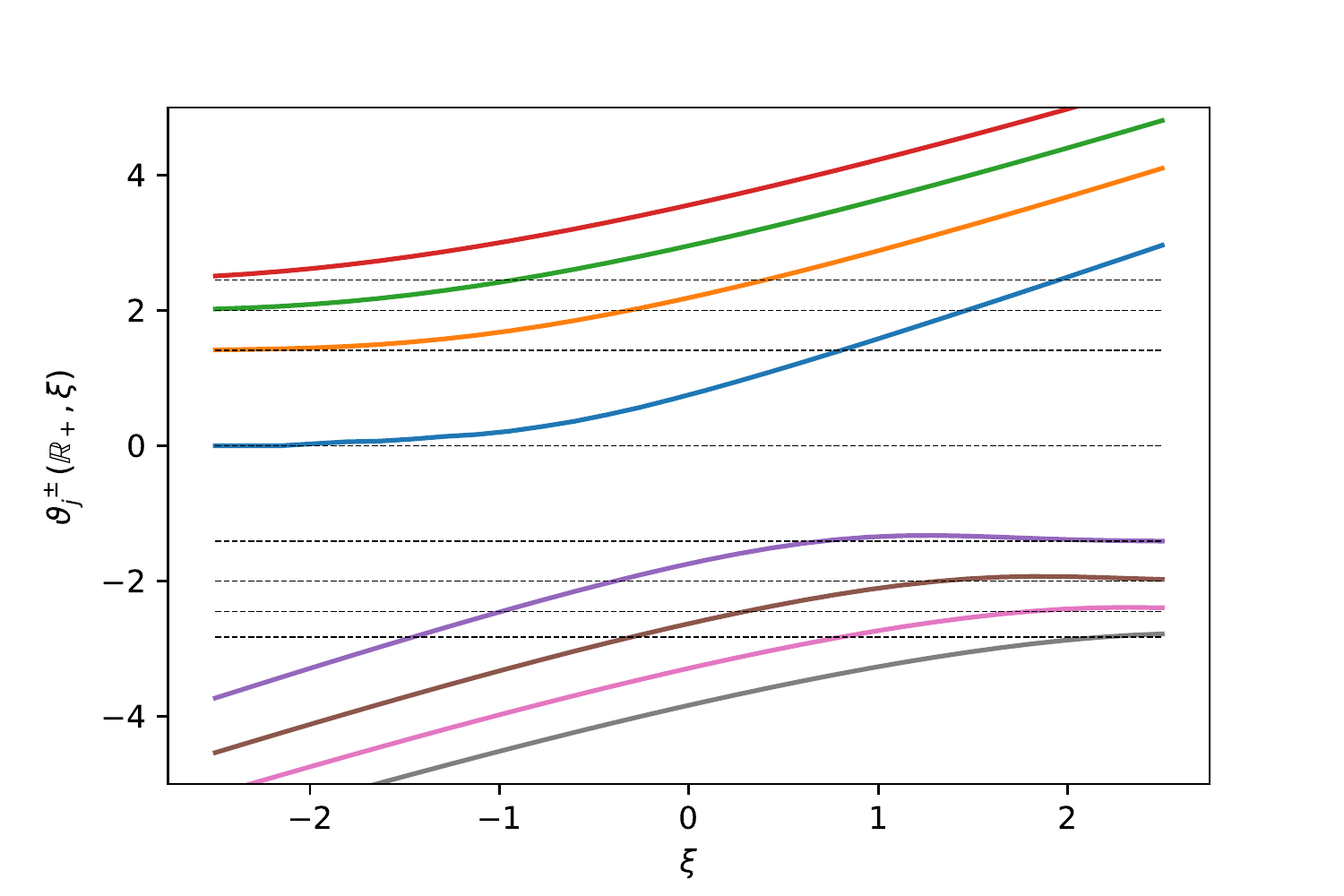}
	\caption{Functions $\pm\vartheta_{\RR_+,j}^{\pm}$}\label{fig4}
\end{figure}
All these simulations agree with all our theoretical results.

\newpage

\subsection{On the function $\nu$ and the different characterizations of $a_0$}\label{sec.C4}
We recall that $\nu$ is defined in \eqref{eq.nuc}. 
\begin{proposition}\label{prop.nu}
	The function $\nu$ is non-negative on $[0,+\infty)$, increasing, concave and it satisfies \[\nu(0)=0\,,\quad \nu(+\infty)=2\,,\quad \liminf_{\alpha\to 0^+}\frac{\nu(\alpha)}{\alpha}>0\,.\]
	In particular, the constant $a_0$ defined in Theorem \ref{thm.dispertionCurve} is the unique positive solution of the equation $\nu(\alpha)=\alpha^2$. Moreover, $\xi_{a_0}=a_0$, where, for all $\alpha>0$, $\xi_\alpha$ is the unique $\xi$ such that $\nu(\alpha)=\nu_1^-(\alpha,\xi)$ and
	\[
	 a_0 = \inf_{\underset{u\neq 0}{u\in \mathfrak{H}^2_{-\mathbf{A}_0}(\mathbb{R}^2_+)}}\!\!\!\frac{
		\|u\|^2_{\pa \mathbb{R}^2_+} + \sqrt{\|u\|^4_{\pa \mathbb{R}^2_+} + 4\|u\|^2\int_{\mathbb{R}^2_+} |(-i\partial_{s}-\tau+i(-i\partial_{\tau})) u|^2\dd s\dd \tau}	
	}{2 \|u\|^2}\,.
	\]
\end{proposition}

\begin{proof}
	The function $\nu$ is concave as an infimum of linear functions. The equality $\nu(0)=0$ follows by considering the zero modes\footnote{We can also check that $\nu$ is right continuous at $0$.}, and $\nu(+\infty)=2$ comes from the fact that, when $\alpha\to+\infty$, $\nu(\alpha)$ converges to the groundstate energy on the half-space with Dirichlet boundary condition. Then, the concavity implies that 
	\begin{equation}\label{eq.liminfnuc}
	\liminf_{\alpha\to 0^+}\frac{\nu(\alpha)}{\alpha}>0\,.
	\end{equation}
	Let us explain why $\nu$ is a smooth function on $(0,+\infty)$. Let us recall that by Proposition \ref{prop.C3},
	\begin{equation}\label{eq.nucfiber}
	\nu(\alpha)=\min_{\xi\in\mathbb{R}}\nu^-_1(\alpha,\xi)=\nu^-_1(\alpha,\xi_\alpha)<2\,,
	\end{equation}
	and that, for all $\alpha> 0$, $\xi_\alpha>0$ is the unique solution of 
	\begin{equation}\label{eq.critical}
	\partial_\xi\nu^-_1(\alpha,\xi)=0\,.
	\end{equation}
	For all $\alpha>0$, we have $\partial^2_\xi\nu^-_1(\alpha,\xi_\alpha)>0$, and thus the analytic implicit function theorem applied to \eqref{eq.critical} implies that $\alpha\mapsto \xi_\alpha$ is analytic. Since $\nu^-_1$ is analytic, we deduce that $\alpha\mapsto\nu(\alpha)$ is analytic. We notice that
	\[\nu'(\alpha)=\partial_\alpha\nu^-_1(\alpha,\xi_\alpha)+\partial_\xi\nu^-_1(\alpha,\xi_\alpha)\frac{\dd\xi_\alpha}{\dd \alpha}=\partial_\alpha\nu^-_1(\alpha,\xi_\alpha)\,.\]
	Thanks to Lemma \ref{lem.nu'a}, we get
	\begin{equation}\label{eq.nu'cu20}
	\nu'(\alpha)=u^2_{\alpha,\xi_\alpha}(0)>0\,.
	\end{equation}
	Let us now consider the function
	\[f(\alpha)=\nu(\alpha)-\alpha^2\,.\]
	From \eqref{eq.liminfnuc}, we see that $f$ is positive on some interval $(0,a)$ with $a>0$. Then, by $\nu(+\infty)=2$, we see that $f$ is negative on some interval $(b,+\infty)$. By the Intermediate Value Theorem, we deduce that $f$ has at least one zero in $(0,+\infty)$. Let us prove that there is only one zero. Consider $\alpha>0$ such that $f(\alpha)=0$. We have $f'(\alpha)=\nu'(\alpha)-2\alpha$. 
	Due to \eqref{eq.C2}, we have $\xi_\alpha=\alpha$, and with Lemma \ref{lem.nu'}, we get
	\[2\alpha-u^2_{\alpha,\xi_\alpha}(0)=\int_0^{+\infty}t u^2_{\alpha,\xi_\alpha}(t)\dd t>0\,.\]
	This, with \eqref{eq.nu'cu20}, implies that $f'(\alpha)<0$. We deduce that $f$ has at most one positive zero (and thus exactly one, denoted by $a_1$).
By \eqref{eq.nucfiber}, we get that
\[
\nu(a_1) = \nu_1^-(a_1,a_1) = a_1^2\,,
\]
so that $a_1 = \vartheta^-_1(a_1)$. By \eqref{eq.crit-theta} and Theorem \ref{thm.dispertionCurve}, $a_1$ is the only critical point of $ \vartheta^-_1$ so that $a_0 = a_1$.

Let us denote by $a_2$ the constant introduced in Remark \ref{rem.defa0} :
\[
	 a_2 := \inf_{\underset{u\neq 0}{u\in \mathfrak{H}^2_{-\mathbf{A}_0}(\mathbb{R}^2_+)}}\!\!\!\rho(u)\,,
\]
with
\[
\rho(u):=\frac{
		\int_{\mathbb{R}}|u(s,0) |^2\dd s + \sqrt{\left(\int_{\mathbb{R}}|u(s,0) |^2\dd s\right)^2 + 4\|u\|^2\int_{\mathbb{R}^2_+} |(-i\partial_{s}-\tau+i(-i\partial_{\tau})) u|^2\dd s\dd \tau}	
	}{2 \|u\|^2}\,.
	\]
Let $u\in \mathfrak{H}^2_{-\mathbf{A}_0}(\mathbb{R}^2_+)\setminus\{0\}$ such that $\|u\|=1$. $\rho(u)$ is the only positive root of the second order polynomial
\[
Q_u\colon\alpha \mapsto \int_{\mathbb{R}^2_+} |(-i\partial_{x_1}-x_2+i(-i\partial_{x_2})) u|^2\dd x_1\dd x_2
	+\alpha\int_{\mathbb{R}}|u(x_1,0) |^2\dd x_1-\alpha^2\,.
\]
By \eqref{eq.nuc}, we get that $0 = Q_u(\alpha)\geq \nu(\alpha)-\alpha^2$ with $\alpha = \rho(u)$. Hence, $a_0\leq \alpha$ and taking the infimum over $u$ ensures that $a_0\leq a_2$. Let $\varepsilon>0$. We have $\nu(\alpha)-\alpha^2<0$ with $\alpha = a_0+\varepsilon$. Therefore, there exists a normalized function $u\in \mathfrak{H}^2_{-\mathbf{A}_0}(\mathbb{R}^2_+)$ such that 
\[
Q_u(\alpha)< 0\,.
\]
By definition of $\rho(u)$, we get that $\rho(u)<\alpha$. By definition of $a_2$, we get $a_2< a_0+\varepsilon$ and the result follows.

	
\end{proof}

\section{Curvature related formulas for $\nu^-_1$}\label{sec.moreformula}

In this section, we study other properties of the eigenfunctions associated with the negative eigenvalues. They will be used in the fine structure splitting of the negative eigenvalues, see Theorem \ref{thm.main2}. In particular, Lemmata \ref{lem.magie} and \ref{lem.final} will be crucial.

Let $\alpha>0$, $\xi\in \RR$, $\nu = \nu^-_1$and $u_{\alpha,\xi}$ is a regular branch of real normalized eigenfunctions associated with the eigenvalue $\nu(\alpha,\xi)$: 
\[
\begin{split}
	&\mathscr{M}_{\alpha,\xi}^-u_{\alpha,\xi} = \nu(\alpha,\xi)u_{\alpha,\xi}\,, \text{ on }\RR_+\,,\\
	&(\pa_\tau + \xi-\alpha)u_{\alpha,\xi}(0) = 0\,.
\end{split}
\]
\subsection{About the momenta of $u_{\alpha,\xi}$ and $\mathscr{C}_\xi$}
We consider the operator
\begin{equation}\label{eq.Cxi}
	\mathscr{C}_\xi=
	2\left(\tau \mathscr{M}_{\alpha,\xi}^-+\xi-\partial_\tau+\tau^2(\xi-\tau)\right)
\,,\end{equation}
which will appear in the computation of the asymptotics of the negative eigenvalues.
\begin{lemma}\label{lem.Cxisym}
The operator $\mathscr{C}_\xi$ is symmetric on $\mathscr{S}(\overline{\mathbb{R}_+})$.
\end{lemma}
\begin{proof}
For shortness, we let $n_0=\mathscr{M}_{\alpha,\xi}^-$. We write
\[\begin{split}
\langle(\tau n_0-\partial_\tau) u,v\rangle&=\langle n_0 u,\tau v\rangle-\langle u',v\rangle\\
&=\langle u,n_0(\tau v)\rangle+u'(0)(\tau v)(0)-u(0)(\tau v)'(0)-\langle u',v\rangle\,,\end{split}\]
so that
\[\langle(\tau n_0-\partial_\tau) u,v\rangle=\langle n_0 u,\tau v\rangle-u(0)v(0)+\langle u,v'\rangle+u(0)v(0)\,,\]
and
\[\langle(\tau n_0-\partial_\tau) u,v\rangle=\langle u,(n_0\tau+\partial_\tau)v\rangle=\langle u,(\tau n_0-\partial_\tau)v\rangle\,.\]
\end{proof}
We let
\[M_j=\int_0^{+\infty} (\xi-\tau)^ju^2_{\alpha,\xi}\,\dd\tau\,,\]
and 
\[P_j(\tau)=(\xi-\tau)^j\,.\]
In order to compute the momenta $M_j$, the following lemma will be convenient.
\begin{lemma}\label{lem.p}
		Let $p$ be any polynomial in the $\tau$ variable. We have
	\[\begin{split}
	&\left(\mathscr{M}_{\alpha,\xi}^--\nu(\alpha,\xi_{})\right)(2pu'_{\alpha,\xi}-p'u_{\alpha,\xi})
	\\&=\left(p^{(3)}-4((\xi-\tau)^2 +1-\nu(\alpha,\xi_{}))p'+4(\xi-\tau)p\right)u_{\alpha,\xi}\,.
	\end{split}\]
	Moreover, we have
	\begin{equation*}
	\begin{split}
	&\langle\left(\mathscr{M}_{\alpha,\xi}^--\nu(\alpha,\xi_{})\right)(2pu'_{\alpha,\xi}-p'u_{\alpha,\xi}), u_{\alpha,\xi}\rangle\\
	=&u^2_{\alpha,\xi}(0)\left( - p''(0) +2p'(0)(\alpha-\xi_{}) +2p(0)(\xi^2 +1-\nu(\alpha,\xi_{})-(\alpha-\xi)^2) \right)\,.
	\end{split}
	\end{equation*}
	\end{lemma}
\begin{proof}
	The first part follows by a straightforward computation. For the second identity, we use Lemma \ref{lem.ipp} and the equation and the boundary condition satisfied by $u_{\alpha,\xi}$,
		\begin{equation*}
	\begin{split}
	&\langle\left(\mathscr{M}_{\alpha,\xi}^--\nu(\alpha,\xi_{})\right)(2pu'_{\alpha,\xi}-p'u_{\alpha,\xi}), u_{\alpha,\xi}\rangle\\
	=&-(2pu'_{\alpha,\xi}-p'u_{\alpha,\xi}) u'_{\alpha,\xi}(0)+(2pu'_{\alpha,\xi}-p'u_{\alpha,\xi})'(0)  u_{\alpha,\xi}(0)\\
	=&-2p(0)(\xi-\alpha)^2u^2_{\alpha,\xi}(0)+ p'(0)(\alpha-\xi_{})u^2_{\alpha,\xi}(0)\\
	&+(2p(0)u''_{\alpha,\xi}(0)+2p'(0)(\alpha-\xi)u_{\alpha,\xi}(0)-p'' (0)u_{\alpha,\xi}(0)-p'(0)(\alpha-\xi)u_{\alpha,\xi}(0))  u_{\alpha,\xi}(0)\\
	=&u^2_{\alpha,\xi}(0)\left( -2p(0)(\xi-\alpha)^2 +2p'(0)(\alpha-\xi_{}) +2p(0)(\xi^2 +1-\nu(\alpha,\xi_{})) - p''(0)\right)\,.
	\end{split}
	\end{equation*}
	\end{proof}
	In the following lemma, we compute the first momenta $M_j$.
	
\begin{lemma}\label{lem.momenta}
	We have
%
\[\begin{split}
M_1 &= 
\frac{u^2_{\alpha,\xi}(0)}{2}(\xi^2 +1-\nu(\alpha,\xi_{})-(\alpha-\xi)^2)\,,
\\
M_2 &=
\frac{\nu(\alpha,\xi)-1}{2}
+
\frac{u^2_{\alpha,\xi}(0)}{4}\left(-(\alpha-\xi) + \xi(\xi^2 +1-\nu(\alpha,\xi_{})-(\alpha-\xi)^2)\right)\,,
\\
M_3 &=
	\frac{u^2_{\alpha,\xi}(0)}{12}\left(
		(4\nu(\alpha,\xi_{})-4+2\xi^2)(\xi^2 +1-\nu(\alpha,\xi_{})-(\alpha-\xi)^2) - 2 -4\xi(\alpha-\xi_{}) \right)\,,
\end{split}\]
and 
	\[\begin{split}
	M_4&=
	\frac{3}{8} + \frac{3}{8}(\nu-1)^2
	%
	+\frac{u^2_{\alpha,\xi}(0)}{16}\left(
	- 6\xi
	+
	\left(
	3(1-\nu)-6\xi^2
	\right)(\alpha-\xi) 
	\right)
	\\&
	+\frac{u^2_{\alpha,\xi}(0)}{16}
	(2\xi^3+3(\nu-1)\xi)(\xi^2 +1-\nu-(\alpha-\xi)^2) 
	\end{split}\]
	Assume that $\alpha=\xi = a_0$. Then, we have
\[\begin{split}
M_1&=\frac{u^2_{\alpha,\xi}(0)}{2}\,,\\
M_2&=\frac{\xi^2-1}{2}+\frac{\xi u^2_{\alpha,\xi}(0)}{4}\,,\\
M_3&=\frac{\xi^2-1}{2}u^2_{\alpha,\xi}(0)\,,\\
M_4&=\frac38+\frac38(\xi^2-1)^2+\frac{u^2_{\alpha,\xi}(0)}{16}(5\xi^3-9\xi)\,.
\end{split}\]
%
	\end{lemma}

\begin{remark}
The assumption on $\alpha$ and $\xi$ comes from that fact that we will need the momenta with $\xi=\xi_\alpha$ when $\alpha$ satisfies $\nu(\alpha)=\alpha^2$. In this case, $\xi=\alpha = a_0$. See \eqref{eq.C2}. 
\end{remark}

\begin{proof}
	In Lemma \ref{lem.p}, we take $p = 1$ and get
	\[
	4M_1 = 2u^2_{\alpha,\xi}(0)(\xi^2 +1-\nu(\alpha,\xi_{})-(\alpha-\xi)^2)\,.
	\]
	We recover the results of Lemma \ref{lem.nu'} and Proposition \ref{eq.C1}.
	Taking $p = (\xi-\tau)$, we get
	\[
	8M_2 + 4M_0(1-\nu(\alpha,\xi)) = u^2_{\alpha,\xi}(0)\left(-2(\alpha-\xi) + 2\xi(\xi^2 +1-\nu(\alpha,\xi_{})-(\alpha-\xi)^2)\right)\,,
	\]
	and the result follows.
	Taking $p = (\xi-\tau)^2$, we get
	\[\begin{split}
	8(1-\nu(\alpha,\xi_{}))&M_1 + 12M_3
	\\&=
	u^2_{\alpha,\xi}(0)\left( - 2 -4\xi(\alpha-\xi_{}) +2\xi^2(\xi^2 +1-\nu(\alpha,\xi_{})-(\alpha-\xi)^2) \right)\,,
	\end{split}
	\]
	and
	\[\begin{split}
	M_3
	=
	\frac{u^2_{\alpha,\xi}(0)}{12}&\left(
		(4\nu(\alpha,\xi_{})-4+2\xi^2)(\xi^2 +1-\nu(\alpha,\xi_{})-(\alpha-\xi)^2)
		%
	 	- 2 -4\xi(\alpha-\xi_{}) \right)\,.
	\end{split}\]
	Finally, taking $p = (\xi-\tau)^3$, we get 
	\[\begin{split}
	-6+16M_4
	&+ 12(1-\nu(\alpha,\xi_{}))M_2
	\\&=
	u^2_{\alpha,\xi}(0)\left( - 6\xi +2(-3\xi^2)(\alpha-\xi_{}) +2\xi^3(\xi^2 +1-\nu(\alpha,\xi_{})-(\alpha-\xi)^2) \right)\,,
	\end{split}\]
	and
	\[\begin{split}
	M_4&=
	\frac{6}{16}
	\\&+\frac{12}{16}(\nu-1)\left(
		\frac{\nu-1}{2}
+
\frac{u^2_{\alpha,\xi}(0)}{4}\left(-(\alpha-\xi) + \xi(\xi^2 +1-\nu-(\alpha-\xi)^2)\right)
		\right)
	\\&+
	\frac{u^2_{\alpha,\xi}(0)}{16}\left( - 6\xi -6\xi^2(\alpha-\xi_{}) +2\xi^3(\xi^2 +1-\nu(\alpha,\xi_{})-(\alpha-\xi)^2) \right)\,.
	\end{split}\]
	The remaining identities follows from the fact that $\nu(a_0,a_0) = a_0^2$ (see Theorem \ref{thm.dispertionCurve} and Notation \ref{not.dispersionCurves}).
	%
%
%
	\end{proof}

The following quantities appear in the construction of the effective operator on the boundary that we obtain expanding the operator (written in appropriate coordinates) in powers of $h$.
%
%

\begin{lemma}\label{lem.magie}
	We have 
	\[
	\langle\mathscr{C}_\xi u_\xi,u_\xi\rangle=2B(\alpha,\xi)u^2_{\alpha,\xi}(0)\,,\]
	where
	\[\begin{split}
	B(\alpha,\xi)=&-\frac{\nu(\alpha,\xi)}{2}(\xi^2+1-\nu(\alpha,\xi)-(\alpha-\xi)^2)+\frac12\\
	&+\frac{1}{12}\left((4\nu(\alpha,\xi)-4+2\xi^2)(\xi^2+1-\nu(\alpha,\xi)-(\alpha-\xi)^2)-2-4\xi(\alpha-\xi)\right)\\
	&-\frac{\xi}{2}\left(-(\alpha-\xi)+\xi(\xi^2+1-\nu(\alpha,\xi)-(\alpha-\xi)^2)\right)\\
	&+\frac{\xi^2}{2}\left(\xi^2+1-\nu(\alpha,\xi)-(\alpha-\xi)^2\right)\,.
	\end{split}\]
	If $\alpha=\xi = a_0$, then
	\[\langle\mathscr{C}_\xi u_{\alpha,\xi},u_{\alpha,\xi}\rangle=0\,,\quad 
	\partial_\xi\langle\mathscr{C}_\xi u_{\alpha,\xi},u_{\alpha,\xi}\rangle=-\frac{\partial^2_\xi\nu(\alpha,\xi)}{2}\,.\]
\end{lemma}

\begin{proof}
We let $n_0=\mathscr{M}_{\alpha,\xi}^-$. Let us write
\begin{multline*}
\langle\left(\tau n_0-\xi-\partial_\tau+\tau^2(\xi-\tau)\right) u_{\alpha,\xi},u_ {\alpha,\xi}\rangle\\
=\nu(\alpha,\xi)\int_0^{+\infty} \tau u_{\alpha,\xi}^2\dd \tau-\xi+\frac{u_{\alpha,\xi}^2(0)}{2}+\int_0^{+\infty}(\xi-\tau-\xi)^2(\xi-\tau)u_{\alpha,\xi}^2\dd\tau\,,
\end{multline*}
so that
\begin{multline*}
\langle\left(\tau n_0-\xi-\partial_\tau+\tau^2(\xi-\tau)\right) u_{\alpha,\xi},u_ {\alpha,\xi}\rangle\\
=\xi(\nu(\alpha,\xi)-1)-\nu(\alpha,\xi) M_1+\frac{u_{\alpha,\xi}^2(0)}{2}+M_3-2\xi M_2+\xi^2M_1\,.
\end{multline*}
Therefore,
\[\langle\mathscr{C}_\xi u_{\alpha,\xi},u_{\alpha,\xi}\rangle=2\left(A(\alpha,\xi)+B(\alpha,\xi)u_{\alpha,\xi}^2(0)\right)\,,\]
where
\[
A(\alpha,\xi)=\xi(\nu(\alpha,\xi)-1)-\xi(\nu(\alpha,\xi)-1)=0\,.
\]
When $\alpha=\xi$ and $\nu(\alpha,\xi)=\xi^2$, we easily check that
\[B(\alpha,\xi)=0\,,\quad \langle\mathscr{C}_\xi u_{\alpha,\xi},u_{\alpha,\xi}\rangle=0\,.\]
Let us compute the derivative $\partial_\xi B (\alpha,\xi)$ when $\alpha=\xi$ and $\nu(\alpha,\xi)=\xi^2$. We have
\[\begin{split}
\partial_\xi B(\alpha,\xi)&=-\xi\nu(\alpha,\xi)+\frac{1}{12}\left(4\xi+2\xi(6\xi^2-4)+4\xi\right)-\frac{\xi}{2}-\frac{\xi}{2}(1+1+2\xi^2)+\xi+\xi^3\\
&=-\frac{\xi}{2}\,.
\end{split}\]
Thus,
\[\partial_\xi\langle\mathscr{C}_\xi u_{\alpha,\xi},u_{\alpha,\xi}\rangle=-\xi u_{\alpha,\xi}^2(0)=-\frac{\partial^2_\xi\nu(\alpha,\xi)}{2}\,.\]
\end{proof}
\begin{remark}
	Note that, when $\alpha=\xi = \xi_\alpha$, we have
	\[u_{\xi_{\alpha}}'(0)=0\,,\]
	and this means that 
	$\nu(\alpha,\xi_{\alpha})$ is an eigenvalue of the de Gennes operator with parameter $\xi_{\alpha}>0$. This eigenvalue lies on the parabola $\xi\mapsto \xi^2$.
\end{remark}
%
%
%

%
The functions $g_{\alpha,\xi}$ and $k_{\alpha,\xi}$ studied in the next two sections appear in the construction of the symbol of the effective operator on the boundary.

%
\subsection{About the function $g_{\alpha,\xi}$}

\begin{lemma}\label{lem.gxi}
	For all $\alpha>0$, $\xi\in\mathbb{R}$, there exists a unique $g_{\alpha,\xi}$ such that
	\[(\mathscr{M}_ {\alpha,\xi}^--\nu(\alpha,\xi)) g_{\alpha,\xi}=u^2_ {\alpha,\xi}(0) u_ {\alpha,\xi}\,,\quad (\xi-\alpha+\partial_\tau)g_{\alpha,\xi}(0)=u_ {\alpha,\xi}(0)\,,\quad \langle g_{\alpha,\xi},u_ {\alpha,\xi}\rangle=0\,.\]
	Moreover,
	\[u_ {\alpha,\xi}(0) v_ {\alpha,\xi}(0) -2\langle(\xi-\tau) g_{\alpha,\xi},u_ {\alpha,\xi}\rangle=-u_ {\alpha,\xi}(0) g_{\alpha,\xi}(0)\,,\]
	and
	\[(\mathscr{M}_ {\alpha,\xi}^--\nu(\alpha,\xi)) \partial_\xi g_{\alpha,\xi}=\partial_\xi\nu(\alpha,\xi)g_{\alpha,\xi}+2u_ {\alpha,\xi}(0) v_ {\alpha,\xi}(0) u_ {\alpha,\xi}+u^2_ {\alpha,\xi}(0) v_ {\alpha,\xi}-2(\xi-\tau) g_{\alpha,\xi}\,.\]
\end{lemma}
\begin{proof}
	For shortness, we write $u$, $v$ and $g$ instead of $u_ {\alpha,\xi}$, $v_ {\alpha,\xi}$ and $g_ {\alpha,\xi}$. For any $g$ satisfying the boundary condition, we have
	\[\langle(\mathscr{M}_ {\alpha,\xi}^--\nu(\alpha,\xi)) g,u\rangle=g'(0)u(0)-g(0)u'(0)=(g'(0)-g(0)(\alpha-\xi))u(0)=u^2(0)\,.\]	
	We have
	\[(\mathscr{M}_ {\alpha,\xi}^--\nu(\alpha,\xi)) \partial_\xi g=\partial_\xi\nu(\alpha,\xi)g+2u(0) v(0) u+u^2(0) v-2(\xi-\tau) g\,.\]
	Then, using the fact that $0 = \partial_\xi \|u\|^2 = 2\langle u, v\rangle$, and Lemma \ref{lem.ipp},
	\[\begin{split}
	\langle(\mathscr{M}_ {\alpha,\xi}^--\nu(\alpha,\xi)) \partial_\xi g,u\rangle&=2u(0) v(0) -2\langle(\xi-\tau) g,u\rangle\\
	&=\partial_\tau\partial_\xi g(0)u(0)-\partial_\xi g(0)u'(0)\\
	&=\left(\partial_\tau\partial_\xi g(0)+(\xi-\alpha)\partial_\xi g(0)\right)u(0)\,.
	\end{split}\]
	Note that
	\[(\xi-\alpha+\partial_\tau)\partial_\xi g(0)=v(0)-g(0)\,.\]
	Thus
	\[2u(0) v(0) -2\langle(\xi-\tau) g,u\rangle=u(0)v(0)-u(0) g(0)\,.\]
\end{proof}
\begin{lemma}\label{lem.v+g}
	We have
	\[(\mathscr{M}_ {\alpha,\xi}^--\nu(\alpha,\xi)) (v_{\alpha,\xi}+g_{\alpha,\xi})=\partial_\xi\nu(\alpha,\xi) u_ {\alpha,\xi}-2(\xi-t)u_ {\alpha,\xi}+u^2_ {\alpha,\xi}(0)u_ {\alpha,\xi}\,,\]
	with the boundary condition 
	\[(\partial_t+\xi-\alpha)(v_{\alpha,\xi}+g_{\alpha,\xi})(0)=0\,,\]
	and $\langle v_{\alpha,\xi}+g_{\alpha,\xi},u_{\alpha,\xi}\rangle=0$.
\end{lemma}	
\begin{proof}
	It follows from Lemmata \ref{lem.nu'} and \ref{lem.gxi}.
\end{proof}

\begin{lemma}\label{lem.ug}
	We have
	\[u_{\alpha,\xi}(0)g_{\alpha,\xi}(0)=\frac{1}{2}\partial^2_{\alpha}\nu(\alpha,\xi)\,.\]
\end{lemma}
\begin{proof}
	For shortness, we write $u$, $v$ and $g$ instead of $u_ {\alpha,\xi}$, $v_ {\alpha,\xi}$ and $g_ {\alpha,\xi}$.
	Thanks to Lemma \ref{lem.nu'a}, we have
	\[(\mathscr{M}_{\alpha,\xi}^--\nu(\alpha,\xi)) g=\partial_\alpha\nu(\alpha,\xi) u\,.\]
	Taking the derivative with respect to $\alpha$ and then the inner product with $u$, we get
	\[\langle (\mathscr{M}_{\alpha,\xi}^--\nu(\alpha,\xi)) \partial_\alpha g ,u\rangle=\partial_\alpha\nu(\alpha,\xi)\langle g,u\rangle+\partial^2_\alpha\nu(\alpha,\xi)+\partial_\alpha\nu(\alpha,\xi)\langle\partial_\alpha u,u\rangle = \partial^2_\alpha\nu(\alpha,\xi)\,,\]
	and an integration by parts provides us with
	\[\partial_t\partial_\alpha g(0) u(0)-\partial_\alpha g(0)\partial_t u(0)=\partial^2_\alpha\nu(\alpha,\xi)\,.\]
	We have
	\[(\partial_t+\xi-\alpha)\partial_\alpha g(0)=\partial_{\alpha} u(0)+g(0)\,,\]
	so that
	\[\partial_\alpha u(0)u(0)+u(0)g(0)=\partial^2_\alpha\nu(\alpha,\xi)\,.\]
	\[
	\pa_\alpha^2\nu(\alpha,\xi) = 2u(0)\pa_\alpha u(0)\,,
	\]
	so that
	\[
	u(0)g(0)=\frac{\partial^2_\alpha\nu(\alpha,\xi)}{2}\,.\]
\end{proof}

\subsection{About the function $k_{\alpha,\xi}$}
Let $k_{\alpha,\xi}$ be the unique solution orthogonal to $u_{\alpha,\xi}$ of
\begin{equation}\label{eq.k}
(\mathscr{M}_{\alpha,\xi}^--\nu(\alpha,\xi))k_{\alpha,\xi}=-\Pi^\perp( \mathscr{C}_\xi u_{\alpha,\xi})\,.
\end{equation}
with
\[(\xi-\alpha+\partial_\tau)k_{\alpha,\xi}(0)=0\,,\]
 where $\Pi^\perp$ is the orthogonal projection on the orthogonal of $u_{\alpha,\xi}$,

\begin{lemma}\label{lem.Cuv}
	We have
	\[\langle\mathscr{C}_\xi u_{\alpha, \xi},v_{\alpha, \xi}\rangle=-k_{\alpha, \xi}(0)u_{\alpha,\xi}(0)+2\langle(\xi-\tau)k_{\alpha, \xi},u_{\alpha,\xi}\rangle\,.\]	
\end{lemma}
\begin{proof}
	We recall \eqref{eq.k}, and Lemma \ref{lem.nu'}. Then, since $\langle u_{\alpha,\xi},v_{\alpha,\xi}\rangle=0$ and $\langle u_{\alpha,\xi},k_{\alpha,\xi}\rangle=0$, we get, by an integration by parts,
	\[\begin{split}
	-\langle\mathscr{C}_\xi u_{\alpha,\xi},v_{\alpha,\xi}\rangle
	&=-\langle\Pi^\perp\mathscr{C}_\xi u_{\alpha,\xi},v_{\alpha,\xi}\rangle
	=\langle\left(\mathscr{M}_{\alpha,\xi}^--\nu(\alpha,\xi)\right)k_{\alpha,\xi},v_{\alpha,\xi}\rangle
	\\&=k'_{\alpha,\xi}(0)v_{\alpha,\xi}(0)-k_{\alpha,\xi}(0)v'_{\alpha,\xi}(0)+\langle k_{\alpha,\xi},\left(\mathscr{M}_{\alpha,\xi}^--\nu(\alpha,\xi)\right)v_{\alpha,\xi}\rangle\\
	&=u_{\alpha,\xi}(0)k_{\alpha,\xi}(0)-2\langle(\xi-\tau)u_{\alpha,\xi},k_{\alpha,\xi}\rangle\,.
	\end{split}\]
\end{proof}
In fact the function $k_{\alpha,\xi}$ can be computed explicitly when $\alpha=\xi$ and $\nu(\alpha,\xi)=\alpha^2$.

\begin{lemma}\label{lem.kexplicit}
Assume that $\alpha=\xi = a_0$.  	
Consider the function
\[k_0=\left(-\frac\xi2+\frac23 P_1\right)u_{\alpha,\xi}+\left(\frac{2}{3}(1-\xi^2)+\xi P_1-\frac13 P_2\right)u'_{\alpha,\xi}\,.\]
It solves
\[(\mathscr{M}_{\alpha,\xi}-\nu(\alpha,\xi))k_0=-\mathscr{C}_\xi u_{\alpha,\xi}\,,\]
and satisfies the Neumann condition.

\end{lemma}
\begin{proof}
We take $\alpha=\xi=a_0$. For shortness, we remove the reference to $(\alpha,\xi)$. We want to solve the equation
\[(\mathscr{M}-\nu)k=-\mathscr{C}_\xi u\,.\]
We have
\[(\mathscr{M}-\nu)(pu+qu')=-p'' u-q'' u'-2p'u'-2q'u''+q(\mathscr{M}-\nu)u'\,,\]
and
\[(\mathscr{M}-\nu)u'-2(\xi-\tau)u=0\,,\quad -u''=(\nu-1)u-(\xi-\tau)^2u\,.\]
Thus,
\[(\mathscr{M}-\nu)(pu+qu')=(-p''+2q'((\nu-1)-P_2)+2qP_1)u+(-2p'-q'')u'\,.\]
Looking at the expression of $\mathscr{C}_\xi$ in \eqref{eq.Cxi}, we want that
\[-2p'-q''=2\,,\]
so that we take
\[p=-\tau-\frac{q'}{2}\,.\]
Let us now determine the function $q$. We want that
\[-\frac{q^{(3)}}{2}+2q'((\nu-1)-P_2)+2qP_1=-2\nu\tau+2\xi-2\tau^2P_1\,.\]
We get
\[-\frac{q^{(3)}}{2}+2q'((\nu-1)-P_2)+2qP_1=2\nu P_1-2\xi\nu+2\xi-2(\xi-\tau-\xi)^2P_1\,,\]
which can be written as
\[-\frac{q^{(3)}}{2}+2q'((\nu-1)-P_2)+2qP_1=-2\xi(\nu-1)+(2\nu-2\xi^2) P_1+4\xi P_2-2P_3\,.\]
We look for $q$ in the form
\[q=\alpha_3 P_3+\alpha_2P_2+\alpha_1 P_1+\alpha_0\,.\]
We have $q^{(3)}=-6\alpha_3$ and
\[q'=-3\alpha_3 P_2-2\alpha_2P_1-\alpha_1\,.\]
We get
\[
\begin{split}
3\alpha_3-2\alpha_1(\nu-1)&=-2\xi(\nu-1)\,,\\
-4\alpha_2(\nu-1)+2\alpha_0&=2(\nu-\xi^2)=0\,,\\
2\alpha_1-6\alpha_3(\nu-1)+2\alpha_1&=4\xi\,,\\
4\alpha_2+2\alpha_2&=-2\,,\\
6\alpha_3+2\alpha_3&=0\,.
\end{split}
\]
This gives
\[\alpha_3=0\,,\quad \alpha_2=-\frac13\,,\quad \alpha_1=\xi\,,\quad \alpha_0=\frac{2}{3}(1-\xi^2)\,.\]
Thus, 
\[q=\frac{2}{3}(1-\xi^2)+\xi P_1-\frac13 P_2\,,\]
and 
\[p=-\tau-\frac{q'}{2}=P_1-\xi-\frac12(-\xi+\frac23 P_1)=-\frac\xi2+\frac23 P_1\,.\]

	Let us check that $k_0$ satisfies the Neumann condition. Notice that $u'(0)=0$ and $u''(0)=u(0)$. We have
	\[\partial_\tau k_0(0)=\left(-\frac23+\frac23(1-\xi^2)+\xi^2-\frac{\xi^2}{3}\right)u(0)=0\,.\]
	\end{proof}
The following two lemmata will be used in the proof of Theorem \ref{thm.main2}, especially to find the constant $-\frac{1}{12}$ in \eqref{eq.Qeff}.
\begin{lemma}\label{lem.Cukexpl1}
Assume that $\alpha=\xi = a_0$. Then,
\begin{equation*}
\langle \mathscr{C}_\xi k_{\alpha,\xi},u_{\alpha,\xi}\rangle=-\frac{3}{4}+\frac{11}{4}\xi^2-\frac{19}{8}\xi^4+\left(-\frac{37}{48}\xi+\frac{19}{16}\xi^3\right)u_{\alpha,\xi}^2(0)\,.
\end{equation*}	
\end{lemma}
\begin{proof}
We drop the index $(\alpha,\xi)$. We have
\[\langle \mathscr{C}k,u\rangle=\langle\mathscr{C}k_0,u\rangle=\langle k_0,\mathscr{C}u\rangle=2\langle k, \left(\xi(\xi^2-1)-2\xi P_2+P_3-\partial_\tau\right)u\rangle\,,\]
where we used Lemmata \ref{lem.Cxisym} and \ref{lem.magie}, and the explicit expression of $\mathscr{C}_\xi$ in \eqref{eq.Cxi}. Therefore, with Lemma \ref{lem.kexplicit}, we have to estimate
\begin{multline*}
I=\left\langle\left(-\frac\xi2+\frac23 P_1\right)u+\left(\frac23(1-\xi^2)+\xi P_1-\frac13 P_2\right)u',(\xi(\xi^2-1)-2\xi P_2+P_3)u-u'\right\rangle\,.
\end{multline*}
It can be written as
\[I=I_1+I_2+I_3+I_4\,,\]
where
\[\begin{split}
I_1&=\left\langle\left(-\frac\xi2+\frac23 P_1\right)(\xi(\xi^2-1)-2\xi P_2+P_3)u,u\right\rangle\,,\\
I_2&=-\left\langle\left(\frac23(1-\xi^2)+\xi P_1-\frac13 P_2\right)u',u'\right\rangle\,,\\
I_3&=-\left\langle(-\frac\xi2+\frac23 P_1)u,u'\right\rangle\,,\\
I_4&=\left\langle\left(\frac23(1-\xi^2)+\xi P_1-\frac13 P_2\right)u',(\xi(\xi^2-1)-2\xi P_2+P_3)u\right\rangle\,.
\end{split}
\]
We have
\[\left(-\frac\xi2+\frac23 P_1\right)(\xi(\xi^2-1)-2\xi P_2+P_3)=-\frac{\xi}{2}(\xi^3-\xi)+\frac23(\xi^3-\xi)P_1+\xi^2P_2-\frac{11}{6}\xi P_3+\frac23 P_4\,,\]
so that
\begin{equation}\label{eq.I1}
I_1=-\frac{\xi}{2}(\xi^3-\xi)+\frac23(\xi^3-\xi)M_1+\xi^2M_2-\frac{11}{6}\xi M_3+\frac23 M_4\,.
\end{equation}
Then, by integrating by parts and using $u'(0)=0$,
\begin{equation*}
I_2=\left\langle\left(\left(\frac23(1-\xi^2)+\xi P_1-\frac13 P_2\right)u'\right)',u\right\rangle
\end{equation*}
We get
\[I_2=\left\langle \left(-\xi+\frac23 P_1\right)u',u\right\rangle+\left\langle\left(\frac23(1-\xi^2)+\xi P_1-\frac13 P_2\right)u'',u\right\rangle\,.\]
By integration by parts,
\[\left\langle (-\xi+\frac23 P_1)u',u\right\rangle=\frac13+\frac{\xi}{6}u^2(0)\,.\]
We recall that $u''=(P_2+1-\xi^2)u$ so that
\[\left\langle\left(\frac23(1-\xi^2)+\xi P_1-\frac13 P_2\right)u'',u\right\rangle=\langle(\frac23(1-\xi^2)+\xi P_1-\frac13 P_2)(P_2+1-\xi^2)u,u\rangle\,,\]
and then
\[\left\langle\left(\frac23(1-\xi^2)+\xi P_1-\frac13 P_2\right)u'',u\right\rangle=\frac23(1-\xi^2)^2+\xi(1-\xi^2)M_1+\frac13(1-\xi^2)M_2+\xi M_3-\frac13 M_4\,.\]
We deduce that
\begin{equation}\label{eq.I2}
I_2=\frac13+\frac{\xi}{6}u^2(0)+\frac23(1-\xi^2)^2+\xi(1-\xi^2)M_1+\frac13(1-\xi^2)M_2+\xi M_3-\frac13 M_4\,.
\end{equation}
Integrating by parts, we get
\begin{equation}\label{eq.I3}
I_3=-\frac13+\frac{\xi}{12}u^2(0)\,.
\end{equation}
Finally, we again integrate by parts to find
\begin{multline*}
I_4=\frac{\xi}{3}u^2(0)-\frac12\int_0^{+\infty} \left[\left(\frac23(1-\xi^2)+\xi P_1-\frac13 P_2\right)(\xi(\xi^2-1)-2\xi P_2+P_3)\right]'u^2 \dd\tau\,,
\end{multline*}
so that
\begin{equation}\label{eq.I4}
I_4
=\frac{\xi}{3}u^2(0)-\frac12\left(-\xi^2(\xi^2-1)+2\xi(1-\xi^2)M_1+(8\xi^2-2)M_2-\frac{20\xi}{3}M_3+\frac53M_4\right)\,.
\end{equation}
It remains to notice that $\langle \mathscr{C}k,u\rangle=2I$, to use \eqref{eq.I1}, \eqref{eq.I2}, \eqref{eq.I3}, \eqref{eq.I4}, and to remember Lemma \ref{lem.momenta}.	
\end{proof}

\begin{lemma}\label{lem.Cukexpl2}
	Assume that $\alpha=\xi= a_0$. Let
	\[
	\mathscr{C}_{\xi,2} =-4\tau(\partial_\tau+\xi-\tau)+2\tau^2 n_0+\frac83(\xi-\tau)\tau^3-4\tau^2+\tau^4  \,,\quad n_0=\mathscr{M}_{\alpha,\xi}\,.
	\]
	We have
\[\langle\mathscr{C}_{\xi,2}u_{\alpha,\xi},u_{\alpha,\xi}\rangle=\frac34-\frac{11}{4}\xi^2+\frac{19}{8}\xi^4+\left(\frac{15}{16}\xi-\frac{19}{16}\xi^3\right)u_{\alpha,\xi}^2(0)\,.\]
\end{lemma}
\begin{proof}
We have
\[\begin{split}
\tau^2&=(\tau-\xi+\xi)^2=P_2-2\xi P_1+\xi^2\,,\\
\tau^3&=(\tau-\xi+\xi)^3=-P_3+3\xi P_2-3\xi^2P_1+\xi^3\,,\\
\tau^4&=(\xi-\tau-\xi)^4=P_4-4\xi P_3+6\xi^2P_2-4\xi^3P_1+\xi^4\,.
\end{split}\]
Thus,
\begin{multline*}
\mathscr{C}_{\xi,2}=-4\tau\partial_\tau+4P_2-4\xi P_1+2(P_2-2\xi P_1+\xi^2)n_0+\frac83(-P_4+3\xi P_3-3\xi^2P_2+\xi^3P_1)\\
-4(P_2-2\xi P_1+\xi^2)+P_4-4\xi P_3+6\xi^2P_2-4\xi^3P_1+\xi^4\,.
\end{multline*}
Since we consider $\langle\mathscr{C}_{\xi,2}u_{\alpha,\xi},u_{\alpha,\xi}\rangle$, we can replace $n_0$ by $\nu$. Rearranging the terms, we get
\[\langle \mathscr{C}_{\xi,2} u_{\alpha,\xi},u_{\alpha,\xi}\rangle=2-\frac{5}{3}M_4+4\xi M_3+(4\xi-\frac{16}{3}\xi^3)M_1-4\xi^2+3\xi^4\,.\]
It remains to use Lemma \ref{lem.momenta}.

	\end{proof}
Lemmata \ref{lem.Cukexpl1} and \ref{lem.Cukexpl2} can be combined with Proposition \ref{prop.C3} to get the following.

\begin{lemma}\label{lem.final}
Assume that	$\alpha=\xi=a_0$. Then,
\[
\langle\mathscr{C}_\xi u_{\alpha,\xi},k_{\alpha,\xi}\rangle
+\left\langle\mathscr{C}_{\xi,2}u_{\alpha,\xi},u_{\alpha,\xi}\right\rangle
=\frac{\xi u^2_{\alpha,\xi}(0)}{6}=\frac{\partial^2_{\xi}\nu(\alpha,\xi)}{12}\,.
\]
\end{lemma}

\section{Semiclassical analysis of the first negative eigenvalue}\label{sec.5}

\subsection{About the proof of Theorem \ref{thm.main'}}
Thanks to the charge conjugation (see Remark \ref{chargec}), the negative eigenvalues $\lambda^-_k(h)$ can be characterized as follows. For $\lambda\geq 0$, consider the quadratic form
\[\widetilde {Q}_{\lambda}(u)=q_{\lambda,h}(u)-\lambda^2\|u\|^2\,,\quad q_{\lambda,h}(u)=\|d^\times_{h,-\mathbf{A}}u\|^2+\lambda h\|u\|^2_{\partial\Omega}\,.\]
Let us denote by $(\tilde\ell_k(\lambda))_{k\geq 1}$ the eigenvalues of the corresponding operator. As in Section \ref{sec.NLminmax}, for all $k\geq 1$, the equation $\tilde\ell_k(\lambda)=0$ has a unique positive solution; this solution is $\lambda_k^-(h)$. On the other hand, we have
\[\tilde\ell_k(\lambda)=\gamma_k(\lambda,h)-\lambda^2\,,\]
where the $(\gamma_k(\lambda,h))_{k\geq 1}$ are the eigenvalues of the operator associated with $q_{\lambda,h}$. Note that, by Lemma \ref{lem.lk}\eqref{eq.lkvi}, for all $\lambda>0$,
\begin{equation}\label{eq.ell'}
|\gamma_k(\lambda,h)-\lambda^2|=|\tilde\ell_k(\lambda)|\geq\lambda|\lambda-\lambda^-_k(h)|\,,
\end{equation}
and $\lambda_k^-(h)$ is the unique solution of
\[\gamma_k(\lambda,h)=\lambda^2\,.\]
We write $\lambda^-_k(h)=e_k(h)h^{\frac 12}$, and the equation becomes
\begin{equation}\label{eq.lambda-}
\gamma_1(e_k(h)h^{\frac 12},h)=e_k(h)^2h\,.
\end{equation}
Note that, by setting $\lambda=ah^{\frac 12}$ with $a>0$, we have the reformulation of \eqref{eq.ell'}: 
\begin{equation}\label{eq.ell'a}
|h^{-1}\gamma_k(ah^{\frac 12},h)-a^2|\geq a|a-e_k(h)|\,.
\end{equation}
The main goal of the next section is to establish the following estimate.
\begin{proposition}\label{prop.Lambdah}
	We have, for all $a>0$,
	\[\gamma_1(ah^{\frac 12},h)=h\Lambda(a)+o(h)\,,\quad \Lambda(a)=\min\left(
	2b_0, b'_{0}\nu(a(b'_{0})^{-1/2})
	\right)\,.
	\]	
\end{proposition}
Proposition \ref{prop.Lambdah} implies Theorem \ref{thm.main'}. Indeed, observe that, substituing this asymptotic expansion into \eqref{eq.ell'a}, we get
\[|\Lambda(a)-a^2+o(1)|\geq a|a-e_1(h)|\,.\]
Notice that, if $a>0$ is such that $\Lambda(a)=a^2$, then 
\[
e_1(h) = a +o(1)\,.
\]
Actually, there is a unique positive $a$ such that
\[\min\left(
2b_0, b'_{0}\nu(a(b'_{0})^{-1/2})
\right)=a^2\,,\]
which is given by 
\[a=\min(\sqrt{2b_0},a_0\sqrt{b'_0})\,,\]
where $a_0$ is the unique positive solution of $\nu(\alpha)=\alpha^2$, see Proposition \ref{prop.nu}. We deduce that
\[\lim_{h\to 0} e_1(h)=\min(\sqrt{2b_0},a_0\sqrt{b'_0})\,,\]
or equivalently
\[\lambda_1^-(h)=h^{\frac 12}\min(\sqrt{2b_0},a_0\sqrt{b'_0})+o(h^{\frac 12})\,.\]

\subsection{Ground energy of a Pauli-Robin type operator}

Let $a>0$. We consider the quadratic form
\[\mathscr{Q}_{a,h}(u)=q_{ah^{1/2},h}(u)=\|d^\times_{h,-\mathbf{A}}u\|^2+ah^{\frac 32}\|u\|^2_{\partial\Omega}\,,\]
and we have
\[
\gamma_1(ah^{\frac12},h)= \inf_{\underset{u\neq 0}{u\in \mathfrak{H}^2_{-\mathbf{A}}(\Omega)}}\frac{\mathscr{Q}_{a,h}(u)}{\|u\|^2}\,.
\]

\subsubsection{Localization formula}
Let $\rho\in\left(0,\frac 12\right)$. Let us consider a semiclassical partition of the unity $(\chi_j)_{j\in\mathbb{Z}^2}$ with $\mathrm{supp}\,\chi_j\subset D(x_j,h^\rho)$, and such that
\[\sum_{j\in\mathbb{Z}^2}\chi_j^2=1\,,\quad \sum_{j\in\mathbb{Z}^2}|\nabla\chi_j|^2\leq Ch^{-2\rho}\,,\quad (C>1)\,.\]

\begin{lemma}\label{lem.loc}
	We have
	\[\mathscr{Q}_{a,h}(u)=\sum_{j\in\mathbb{Z}^2}\mathscr{Q}_{a,h}(\chi_ju)-h^2\sum_{j\in\mathbb{Z}^2}\|(\nabla\chi_j) u\|^2\,.\]	
	In particular,
	\[\mathscr{Q}_{a,h}(u)\geq\sum_{j\in\mathbb{Z}^2}\mathscr{Q}_{a,h}(\chi_ju)-Ch^{2-2\rho}\|u\|^2\,.\]	
\end{lemma}

\begin{proof}
	Let us write
	\[\begin{split}
	\|d^\times_{h,-\mathbf{A}}u\|^2&=\sum_{j\in\mathbb{Z}^2}\langle d^\times_{h,-\mathbf{A}} u,d^\times_{h,-\mathbf{A}}(\chi^2_j u)\rangle\\
	&=\sum_{j\in\mathbb{Z}^2}\left(\langle d^\times_{h,-\mathbf{A}} u,[d^\times_{h,-\mathbf{A}},\chi_j]\chi_j u\rangle+\langle \chi_j d^\times_{h,-\mathbf{A}} u,d^\times_{h,-\mathbf{A}}(\chi_j u)\rangle\right)\\
	&=\sum_{j\in\mathbb{Z}^2}\left(\langle \chi_j d^\times_{h,-\mathbf{A}} u,[d^\times_{h,-\mathbf{A}},\chi_j] u\rangle+\langle \chi_j d^\times_{h,-\mathbf{A}} u,d^\times_{h,-\mathbf{A}}(\chi_j u)\rangle\right)\\
	&=\sum_{j\in\mathbb{Z}^2}\left(-\|[d^\times_{h,-\mathbf{A}},\chi_j] u\|^2+\langle d^\times_{h,-\mathbf{A}} ( \chi_j u),[d^\times_{h,-\mathbf{A}},\chi_j] u\rangle+\langle \chi_j d^\times_{h,-\mathbf{A}} u,d^\times_{h,-\mathbf{A}}(\chi_j u)\rangle\right)\\
	&=\sum_{j\in\mathbb{Z}^2}\left(\|d^\times_{h,-\mathbf{A}}(\chi_j u)\|^2-\|[d^\times_{h,-\mathbf{A}},\chi_j] u\|^2+2i\mathrm{Im}\,\langle d^\times_{h,-\mathbf{A}} ( \chi_j u),[d^\times_{h,-\mathbf{A}},\chi_j] u\rangle\right)\,,
	\end{split}\]	
	where we used that the commutator $[d^\times_{h,-\mathbf{A}},\chi_j]=-2ih\partial_{\overline{z}}\chi_j$ is a function.
	Taking the real part, we get
	
	\[\|d^\times_{h,-\mathbf{A}}u\|^2=\sum_{j\in\mathbb{Z}^2}\left(\|d^\times_{h,-\mathbf{A}}(\chi_j u)\|^2-\|h(\nabla\chi_j) u\|^2\right)\,.\]
\end{proof}
\subsubsection{Lower bound}
Let $j$ be such that $\mathrm{supp}(\chi_j)\subset\Omega$. Then, we have
\begin{equation}\label{eq.lbint}
\mathscr{Q}_{a,h}(\chi_j u)=\|d^\times_{h,-\mathbf{A}}(\chi_j u)\|^2\geq 2hb_0\|\chi_j u\|^2\,,
\end{equation}
since the Dirichlet realization of
\[d_{h,-\mathbf{A}}d^\times_{h,-\mathbf{A}}=(-ih\nabla+\mathbf{A})^2+hB\]
is bounded from below by $2hb_0$.

Therefore, let us focus on the $j$ such that $\mathrm{supp}(\chi_j)\cap\partial\Omega\neq\emptyset$.
We may assume that $x_j\in\partial\Omega$.

Let us bound the local energy $\mathscr{Q}_{a,h}(\chi_j u)$ from below. 

\begin{proposition}
We have
	\begin{equation}\label{eq.lbbnd}
	\mathscr{Q}_{a,h}(\chi_j u)\geq 	\left[b'_{0}\nu(a(b'_{0})^{-1/2})h-Ch^{\frac 12+2\rho}\right]\|\chi_j u\|^2\,.
	\end{equation}
	
	\end{proposition}
\begin{proof}
Before starting the proof, let us say a few words about the strategy. The general idea is to approximate the magnetic field, on the support of $\chi_j$, by a constant magnetic field, and to flatten the boundary by means of tubular coordinates. Due to the lack of ellipticity of the Cauchy-Riemann operators, we cannot choose the canonical tubular coordinates (given by the curvilinear abscissa and the distance to the boundary). However, with the exponential coordinates \eqref{eq.polar}, we are able to avoid this problem for the disc, and then, by means of the Riemann mapping, for $\Omega$. This amounts to constructing \enquote{conformal} tubular coordinates for $\Omega$.

It is convenient to use the change of function
\[u=e^{\phi/h}v\,.\]
For notational simplicity, we let $u_j=\chi_j u$ and $v_j=\chi_j v$.
We have
\[\mathscr{Q}_{a,h}(u_j)=h^2\int_{\Omega}e^{2\phi/h}|2\partial_{\overline{z}} v_j|^2\dd x+ah^{\frac 32}\|v_j\|^2_{\partial\Omega}\,.\]
Let us use the Riemann biholomorphism $F : \mathbb{D}\to \Omega$. We let $w_j=v_j\circ F$. We get
\[\mathscr{Q}_{a,h}(u_j)=4h^2\int_{\mathbb{D}}e^{2\phi\circ F(y)/h}|\partial_{\overline{y}} w_j|^2\dd y+ah^{\frac 32}\int_{\partial\mathbb{D}}|w_j |^2|F'|\dd\sigma\,,\quad \partial_{\overline{y}}=\frac{1}{2}(\partial_{y_1}+i\partial_{y_2})\,.\]
Note that $w_j$ is supported in a neighborhood of order $h^\rho$ of $\partial\mathbb{D}$. Let us now use a change of coordinates near the boundary. Let $\delta>0$. Consider the \enquote{exponential polar coordinates}, $y=P(s,\tau)$, given by
\begin{equation}\label{eq.polar}
y_1=e^{-\tau}\cos s\,,\quad y_2=e^{-\tau}\sin s\,\quad (s,\tau)\in\mathcal{T}_\delta:=[0,2\pi)\times (0,\delta)\,.
\end{equation}
$P$ is a smooth diffeomorphism in a neighborhood of the boundary. We have
\[-e^{\tau}\partial_{s}=\sin s\partial_{y_1}-\cos s\partial_{y_2}\,,\quad -e^{\tau}\partial_{\tau}=\cos s\partial_{y_1}+\sin s\partial_{y_2}\,,\]
and we get
\[\partial_{y_1}+i\partial_{y_2}=ie^{\tau+is}(\partial_{s}+i\partial_{\tau})\,.\]
The coordinates of the center $x_j$ of the support of $\chi_j$ are denoted by $(s_j,0)$.

In terms of these new coordinates, we have
\begin{multline*}
\mathscr{Q}_{a,h}(u_j)=h^2\int_{\mathcal{T}_\delta} e^{2\phi\circ F(P(s,\tau))/h}|(\partial_{s}+i\partial_{\tau}) (w_j\circ P)|^2\dd s\dd \tau\\
+ah^{\frac 32}\int_{0}^{2\pi}|w_j\circ P(s,0) |^2|F'(e^{is})|\dd s\,.
\end{multline*}
We let $\check \phi=\phi\circ F\circ P$.
Since $\phi$ is zero at the boundary, we have that $e^{2\check\phi(s,0)/h}=1$.
Then, by using that $|F'(e^{is})|\geq (1-Ch^\rho)|F'(e^{is_j})|$, and by commuting the exponential with the Cauchy-Riemann derivative, we get
\[\begin{split}(1-Ch^\rho)^{-1}\mathscr{Q}_{a,h}(u_j)\geq \int_{\mathcal{T}_\delta} |(h\partial_{s}-\partial_s\check \phi+ih\partial_{\tau}-i\partial_\tau\check\phi) e^{\check\phi(s,\tau))/h}(w_j\circ P)|^2\dd s\dd \tau\\
+ah^{\frac 32}\int_{0}^{2\pi}|e^{\check\phi(s,0))/h} w_j\circ P(s,0) |^2|F'(e^{is_j})|\dd s\,. \end{split}\]
Then,
\begin{multline*}
	(1-Ch^\rho)^{-1}\mathscr{Q}_{a,h}(u_j)\geq \int_{\mathcal{T}_\delta} |(-ih\partial_{s}+\check A_1+i(-ih\partial_{\tau}+\check A_2)) W_j|^2\dd s\dd \tau\\
	+|F'(e^{is_j})|a h^{\frac 32}\int_{0}^{2\pi}|W_j(s,0) |^2\dd s\,, 
\end{multline*}
where $W_j=e^{\check\phi(s,\tau))/h}(w_j\circ P)$ and $\check A=\nabla\check\phi^{\perp}=(-\partial_\tau\check\phi,\partial_s\check \phi)$. Now, we have a magnetic Cauchy-Riemann problem on a flat space, with a uniform Robin condition.

A computation that uses the identity $(\partial_{s}+i\partial_{\tau})(e^{-\tau+is}) = 0$ and
\[\begin{split}\nabla\times\check A=(\partial_s^2+\partial_{\tau}^2)(\phi\circ F\circ P)&=e^{-2\tau}\Delta_y(\phi\circ F)(P(s,\tau))\\
&=e^{-2\tau}|F'(P(s,\tau))|^{2}B(F(P(s,\tau)))\\
&=\beta_j+\mathscr{O}(|s-s_j|+|\tau-\tau_j|)\,,
\end{split}\]
gives the new constant magnetic field $\beta_j=|F'(y_j)|^{2}B(x_j)$.

Using the Young inequality, we get
\begin{multline*}
	\int_{\mathcal{T}_\delta} |(-ih\partial_{s}+\check A_1+i(-ih\partial_{\tau}+\check A_2)) W_j|^2\dd s\dd \tau	\\
	\geq (1-\varepsilon)\int_{\mathcal{T}_\delta} |(-ih\partial_{s}+\check A_{1,j}+i(-ih\partial_{\tau}+\check A_{2,j})) W_j|^2\dd s\dd \tau-\varepsilon^{-1}\int_{\mathcal{T}_\delta}|\check A-\check A_j|^2|W_j|^2\dd s\dd\tau\,,
\end{multline*}
where $\check A_{j} = (\check A_{1,j},\check A_{2,j})$ is the Taylor approximation of $\check A$ at the order one at $(s_j,\tau_j)$:
\[|\check A-\check A_j|\leq Ch^{2\rho}\,,\]
on the support of $W_j$.
We get that
\begin{equation}\label{eq.LBQ}
(1-Ch^\rho)^{-1}\mathscr{Q}_{a,h}(u_j)\geq (1-\varepsilon)\mathcal{Q}_j(W_j)-Ch^{4\rho}\varepsilon^{-1}\int_{\mathcal{T}_\delta}|W_j|^2\dd s\dd\tau\,,	
\end{equation}
with
\begin{multline*}
	\mathcal{Q}_j(W)=\int_{\mathbb{R}^2_+} |(-ih\partial_{s}+\check A_{1,j}+i(-ih\partial_{\tau}+\check A_{2,j})) W|^2\dd s\dd \tau\\
	+|F'(e^{is_j})|ah^{\frac 32}\int_{\mathbb{R}}|W(s,0) |^2\dd s\,.
\end{multline*}
Let us remark that, by construction,
\[
\nabla \times \check A_{j} = \beta_j\,,
\]
so that after a change of gauge, we can assume that $\check A_{j} = (-\beta_j \tau,0)$.

Thus, we get a new quadratic form on $L^2(\mathbb{R}^2_+)$ which is associated with a new operator $\mathcal{L}_j$. We are interested in the bottom of its spectrum:
\[\inf \mathrm{sp}(\mathcal{L}_j)=\inf_{\underset{W\neq 0}{W\in \mathfrak{H}^2_{-\check A_j}(\mathbb{R}^2_+)}}\frac{\mathcal{Q}_j(W)}{\|W\|^2}\,.\]
Let us consider the rescaling
\[(s,\tau)=h^{\frac 12}\beta_{j}^{-\frac{1}{2}}(\tilde s,\tilde\tau)\,.\]
We get
\[\inf \mathrm{sp}(\mathcal{L}_j)=h\beta_j\mu_j\,,\quad \mu_j=\inf_{\underset{W\neq 0}{W\in \mathfrak{H}^2_{-\check A_j}(\mathbb{R}^2_+)}}\frac{\widetilde{\mathcal{Q}}_j(W)}{\|W\|^2}\,,\]
where
\begin{equation*}
	\widetilde{\mathcal{Q}}_j(W)=\int_{\mathbb{R}^2_+} |(-i\partial_{s}-\tau+i(-i\partial_{\tau})) W|^2\dd s\dd \tau\\
	+aB(x_j)^{-\frac 12}\int_{\mathbb{R}}|W(s,0) |^2\dd s\,.
\end{equation*}
Then,
\[
(1-Ch^\rho)^{-1}\mathscr{Q}_{a,h}(u_j)\geq \left[(1-\varepsilon)h\beta_{j}\mu_j-Ch^{4\rho}\varepsilon^{-1}\right]\|W_j\|^2\,.	
\]
We choose $\varepsilon$ such that
\[\varepsilon h=\varepsilon^{-1}h^{4\rho}\,,\]
so that
\[\varepsilon=h^{-\frac 12+2\rho}\,,\]
and
\[
(1-Ch^\rho)^{-1}\mathscr{Q}_{a,h}(u_j)\geq \left[h\beta_{j}\mu_j-Ch^{\frac 12+2\rho}\right]\|W_j\|^2\,.	
\]
In particular, we get
\[
\mathscr{Q}_{a,h}(u_j)\geq \left[h\beta_{j}\mu_j-Ch^{\frac 12+2\rho}-Ch^{1+\rho}\right]\|W_j\|^2\,.	
\]
Then,
\[\begin{split}
\mathscr{Q}_{a,h}(u_j)&\geq \left[h\beta_{j}\mu_{j}-Ch^{\frac 12+2\rho}\right]\int_{\mathcal{T}_\delta} e^{2\check\phi(s,\tau))/h}|(v_j\circ F\circ  P)|^2\dd s\dd\tau\\
&\geq \left[h\beta_{j}\mu_{j}-Ch^{\frac 12+2\rho}-Ch^{1+\rho}\right]\int_{\mathbb{D}} e^{2\phi(F(y)))/h}|(v_j\circ F(y))|^2\dd y\\
&= \left[h\beta_{j}\mu_{j}-Ch^{\frac 12+2\rho}-Ch^{1+\rho}\right]\int_{\Omega} e^{2\phi/h}|v_j(x)|^2|(F^{-1})'(x)|^{2}\dd x\\
&\geq |(F^{-1})'(x_j)|^{2} \left[h\beta_{j}\mu_{j}-Ch^{\frac 12+2\rho}-Ch^{1+\rho}\right]\int_{\Omega} e^{2\phi/h}|v_j(x)|^2\dd x\\
&\geq \left[ hB(x_j)\mu_{j}-Ch^{\frac 12+2\rho}-Ch^{1+\rho}\right]\int_{\Omega} e^{2\phi/h}|v_j(x)|^2\dd x\\
&\geq \left[hB(x_j)\mu_{j}-Ch^{\frac 12+2\rho}\right]\int_{\Omega} e^{2\phi/h}|v_j(x)|^2\dd x\\
&=  \left[hB(x_j)\mu_{j}-Ch^{\frac 12+2\rho}\right]\int_{\Omega} |\chi_ju(x)|^2\dd x\,.
\end{split}	
\]
Then, letting $\mathbf{A}_0=(-\tau,0)$, we have
\[\begin{split}
B(x_j)\mu_{j} &=\inf_{\underset{u\neq 0}{u\in \mathfrak{H}^2_{-\mathbf{A}_0}(\mathbb{R}^2_+)}} \frac{
B(x_j)\int_{\mathbb{R}^2_+} |(-i\partial_{s}-\tau+i(-i\partial_{\tau})) u|^2\dd s\dd \tau
+aB(x_j)^{\frac 12}\int_{\mathbb{R}}|u(s,0) |^2\dd s}{\|u\|^2}
\\
&\geq \inf_{\underset{u\neq 0}{u\in \mathfrak{H}^2_{-\mathbf{A}_0}(\mathbb{R}^2_+)}} \frac{
	b'_0\int_{\mathbb{R}^2_+} |(-i\partial_{s}-\tau+i(-i\partial_{\tau})) u|^2\dd s\dd \tau
	+a(b'_0)^{\frac 12}\int_{\mathbb{R}}|u(s,0) |^2\dd s}{\|u\|^2}
\\
&= 
b'_{0}\nu(a(b'_{0})^{-1/2})\,.
\end{split}\] 
The result follows.
\end{proof}
\begin{remark}\label{rem.reverse}
It is clear from the proof that we also have a reverse inequality of \eqref{eq.LBQ}:
\begin{equation}\label{eq.LBQ'}
(1-Ch^\rho)^{-1}\mathscr{Q}_{a,h}(u_j)\leq (1+\varepsilon)\mathcal{Q}_j(W_j)+Ch^{4\rho}\varepsilon^{-1}\int_{\mathcal{T}_\delta}|W_j|^2\dd s\dd\tau\,.	
\end{equation}
\end{remark}

Gathering the estimates \eqref{eq.lbint} and \eqref{eq.lbbnd}, and using Lemma \ref{lem.loc}, we find that
\[\mathscr{Q}_{a,h}(u)\geq \left[\Lambda(a) h-Ch^{\frac 12+2\rho}-Ch^{2-2\rho}\right]\|u\|^2 \,.\]
We choose $\rho$ such that
\[\frac 12+2\rho=2-2\rho\,.\]
Thus, $\rho=\frac 38$ and
\[\mathscr{Q}_{a,h}(u)\geq \left[\Lambda(a) h-Ch^{\frac 54}\right]\|u\|^2 \,.\]
The min-max principle implies the lower bound in Proposition \ref{prop.Lambdah}.

\subsubsection{Upper bound}
The upper bound in Proposition \ref{prop.Lambdah}  follows by inserting appropriate localized test functions in $\mathscr{Q}_{a,h}$. Let us provide the main lines of the strategy for this classical analysis.

We recall that
\[
\gamma_1(ah^{\frac12},h)= \inf_{\underset{u\neq 0}{u\in \mathfrak{H}^2_{-\mathbf{A}}(\Omega)}}\frac{\mathscr{Q}_{a,h}(u)}{\|u\|^2}\,.
\]
In particular, we have
\[
\gamma_1(ah^{\frac12},h)\leq \inf_{\underset{u\neq 0}{u\in H^1_0(\Omega)}}\frac{\mathscr{Q}_{a,h}(u)}{\|u\|^2}=\inf_{\underset{u\neq 0}{u\in H^1_0(\Omega)}}\frac{\|(-ih\nabla+\mathbf{A})u\|^2+\int_{\Omega}hB|u|^2\dd x}{\|u\|^2}\,.
\]
The last quantity is the groundstate energy of $(-ih\nabla+\mathbf{A})^2+hB$. Pick up a point $x_0\in\Omega$. We can always find a normalized test function\ $\varphi_h$ in $\mathscr{C}^\infty_0(\Omega)$, localized at the scale $h^{\frac 12}$ near $x_0$, and such that
\[\|(-ih\nabla+\mathbf{A})\varphi_h\|^2+\int_{\Omega}hB|\varphi_h|^2\dd x\leq 2B(x_0)h+o(h)\,.\]
Now, if $B$ attains its minimum inside at $x_0$, then we deduce that
\begin{equation}\label{eq.ubb0}
\gamma_1(ah^{\frac12},h)\leq 2b_0 h+o(h)\,.
\end{equation}
If not, for any $\varepsilon>0$, we may find $x_0\in\Omega$ such that $|B(x_0)-b_0|\leq\varepsilon$, and \eqref{eq.ubb0} is true as well.

On the other hand, let us consider $x_0\in\partial\Omega$ where the minimum of $B_{|\partial\Omega}$ is attained. Take a fixed cutoff function $\chi$ centered at $x_0$, and a minimizing sequence $(W_n)\subset \mathscr{S}(\overline{\mathbb{R}^2_+})$ associated with $\mu$. Then, we consider the function $\psi_h(s,\tau)=\chi(s,\tau)W_n((b'_0)^{\frac 12}h^{-\frac 12}(s,\tau))$ and its avatar $\varphi_h$ in the original coordinates (afer the maps $P$ and $F$). Using Remark \ref{rem.reverse} (where $u_j$ is replaced by $\varphi_h$), we get
\[\gamma_1(ah^{\frac12},h)\leq hb'_{0}\nu(a(b'_{0})^{-1/2})+o(h)\,.\]
This, together with \eqref{eq.ubb0}, gives the desired upper bound.

\section{A first normal form}\label{sec.6}
The aim of this section is to start the proof of Theorem \ref{thm.main2} by reducing the analysis to a tubular neighborhood of the boundary.

\subsection{Description of the operator}
We consider the closed quadratic form
\begin{equation}\label{eq:originalquadform}
\mathscr{Q}_{a,h}(u)=\|\dMhmX u\|^2+ah^{\frac 32}\|u\|^2_{\partial\Omega}\,,\quad \dMhmX=-2ih\partial_{\overline{z}}+A_1+iA_2\,,
\end{equation}
for $u\in \mathrm{Dom}(\mathscr{Q}_{a,h}) = \HhmA$ with
\[\begin{split}
	&\HhmA = H^1(\Omega)+\HtwohmA\,,
	\\
	&\HtwohmA=\{u\in L^2(\Omega) : \dMhmX u=0\,, u_{|\partial\Omega}\in L^2(\partial\Omega)\}\,.
\end{split}\]
We also let
\[
\dMmh = -2ih\partial_{z}+A_1-iA_2\,.
\]
Let us describe the associated self-adjoint operator $\mathscr{L}_{a,h}$. For that purpose, we will need the following lemma.
 \begin{lemma}\label{lem.intbP}
	For all $u,v\in H^1(\Omega)$,
 	\[\langle u,\dMhmX v\rangle=\langle \dMmh u,v\rangle+ih\int_{\partial\Omega} \overline{n} u\overline{v}\mathrm{d}\sigma\,.\]
 \end{lemma}
\begin{proof}
We recall that
\[\langle u,\partial_j v\rangle=-\langle \partial_ju,v\rangle+\int_{\partial\Omega} n_j u\overline{v}\mathrm{d}\sigma\,.\]
Thus,
\[\langle u,-2i\partial_{\overline{z}} v\rangle=\langle -2i\partial_{z}u,v\rangle+i\int_{\partial\Omega} \overline{n} u\overline{v}\mathrm{d}\sigma\,.\]
\end{proof}
By Proposition \ref{prop.H} and Lemma \ref{lem.intbP}, we write  for all $u\in\mathsf{Dom}(\mathscr{L}_{a,h})$, and all $v\in \HhmA$,
\[\begin{split}
&\langle \dMhmX u,\dMhmX v \rangle+ah^{\frac 32}\int_{\partial\Omega} u\overline{v}\mathrm{d}\sigma\\
=&\quad\langle \left(\dMmh\dMhmX\right) u,v \rangle+ih\int_{\partial\Omega} \dMhmX u \overline{n}\overline{v}\mathrm{d}\sigma
+ah^{\frac 32}\int_{\partial\Omega} u\overline{v}\mathrm{d}\sigma\,,
\end{split}
\]
so that the operator $\mathscr{L}_{a,h}$ acts on $L^2(\Omega)$ as
\begin{equation}\label{eq.originaloperator}
 \left(\dMmh\dMhmX\right)=(-ih\nabla+\mathbf{A})^2+hB\,,
\end{equation}
and the boundary condition is
\begin{equation}\label{eq.BC}
-i\overline{n}\dMhmX u=a h^{\frac 12}u\,,\quad \text{ on }\partial \Omega\,.
\end{equation}

\subsection{Localization near the boundary}
We can prove that the eigenfunctions associated with the low-lying eigenvalues have an exponential localization near the boundary, at the scale $h^{\frac 12}$.

\begin{proposition}\label{red.bound}
	Let $\varepsilon_0\in (0,2b_0)$ and $\gamma\in(0,\sqrt{\eps_{0}})$. There exist $C,h_0>0$ such that for all $h\in(0,h_0)$, all $a>0$ and all eigenvalue $\lambda\leq (2b_0-\eps_0)h$ of $\mathscr{L}_{a,h}$ and any eigenfunction $\psi_h$ of $\mathscr{L}_{a,h}$ associated with $\lambda$, we have
	\[
	\left\|\psi_h \exp\left({\frac{\gamma \mathrm{dist}(\cdot,\partial \Omega)}{h^{1/2}}}\right)\right\|^2_{L^2(\Omega)} + h^{-1}\left|\mathscr{Q}_{a,h}\left(\psi_h \exp\left({\frac{\gamma \mathrm{dist}(\cdot,\partial \Omega)}{h^{1/2}}}\right)\right)\right|\leq C\|\psi_h\|^2_{L^2(\Omega)}. 
	\]
\end{proposition}
Before giving the proof, let us recall the following lemma.
\begin{lemma}\label{lem:locformula}
	Let $h>0$, $\chi$ be a real Lipschitzian function on $\Omega$ and $\psi\in \HhmA $, we have
	\[
	\RE\braket{\dMhmX \psi, \dMhmX (\chi^2\psi)}_\Omega = \|\dMhmX (\chi\psi)\|^2
	- h^2\|\psi\nabla \chi\|^2.
	\]
\end{lemma}

Let us now give the proof of Proposition \ref{red.bound}.
\begin{proof}
	Let us define the following Lipschitzian functions
	\[\Omega\ni\x\mapsto \Phi(\x) = \gamma \mathrm{dist}(\x,\partial \Omega)\in \R\,,\]
	and
	\[
	\Omega\ni\x\mapsto \chi_h(\x) = e^{\Phi(\x)h^{-1/2}}\in \R\,.
	\]
	Since $H^1(\Omega)$ is dense in $\mathrm{Dom}(\mathscr{Q}_{a,h})$ and $\chi_h$ is Lipschitzian, we get that $\chi_h^2\psi_h$ belongs to $\mathrm{Dom}(\mathscr{Q}_{a,h})$.
	We have that
	\[\begin{split}
	\mathscr{Q}_{a,h}(\psi_h,\chi_h^2\psi_h) 
	&=\Re\braket{\mathscr{L}_{a,h}\psi_h,\chi_h^2\psi_h}_\Omega
	\\
	&=
	\Re\left\{
	\braket{\dMhmX\psi_h,\dMhmX(\chi_h^2\psi_h)}_{\Omega}
	+ah^{3/2}\|\chi_h\psi_h\|^2_{\partial \Omega}
	\right\}\,.
	\end{split}\]
	By Lemma \ref{lem:locformula}, 
	we get that
	\[\begin{split}
	&\mathscr{Q}_{a,h}(\psi_h,\chi_h^2\psi_h) 
	=
	\mathscr{Q}_{a,h}(\chi_h\psi_h)- h^2\|\psi_h\nabla \chi_h\|_{L^2(\Omega)}^2\,.
	\end{split}\]
	Recall that $\psi_h$ is an eigenfunction of $\mathscr{L}_{a,h}$ associated with the eigenvalue $\lambda$, so that
	\begin{equation}\label{eq:ineq.a}
	\mathscr{Q}_{a,h}(\chi_h\psi_h)- h^2\|\psi_h\nabla \chi_h\|_{L^2(\Omega)}^2
	=\lambda\|\chi_h\psi_h\|^2_{L^2(\Omega)}\,.
	\end{equation}
	Let $R\geq 1$ and $ c>1$.
	We introduce a quadratic partition of unity of $\Omega$,
	\[\chi_{1,h, R}^2+\chi_{2,h, R}^2=1\,, \]
	in order to study the asymptotic behavior of $\psi_h$ in the interior and near the boundary $\partial\Omega$ separately. 
	We assume that $\chi_{1,h, R}$ satisfies
	\[
	\chi_{1,h, R}(\x) = 
	\begin{cases}
	1&\mbox{ if }\mathrm{dist}(\x,\partial\Omega)\geq h^{1/2}R\\
	0&\mbox{ if }\mathrm{dist}(\x,\partial\Omega)\leq h^{1/2}R/2,
	\end{cases}
	\]
	and that,
	\[
	\max(|\nabla \chi_{1,h, R}(\x) |,|\nabla \chi_{2,h, R}(\x) |)\leq 2 ch^{-1/2}/R\,,
	\]
	for all $\x\in \Omega$.
	Using again Lemma \ref{lem:locformula}, we get 
	\[
	\mathscr{Q}_{a,h}(\chi_h\psi_h)
	=
	\sum_{k=1,2}\mathscr{Q}_{a,h}(\chi_{k,h,R}\chi_h\psi_h)
	-h^2\|\chi_h\psi_h\nabla\chi_{k,h,R}\|^2_{L^2(\Omega)}\,.
	\]
	We have $\mathscr{Q}_{a,h}(\chi_{1,h,R}\chi_h\psi_h)\geq 2b_0h\|\chi_{1,h,R}\chi_h\psi_h\|^2$ by support considerations.
	Let us also remark that
	\[
	h^2\|\chi_h\psi_h\nabla \chi_{k,h,R}\|^2_{L^2(\Omega)}
	\leq 
	4h c^2/R^2\|\chi_h\psi_h\|^2_{L^2(\Omega)}
	\]
	and
	\[
	h^2\|\psi_h\nabla \chi_h\|^2_{L^2(\Omega)}
	\leq 
	h\gamma^2\|\chi_h\psi_h\|^2_{L^2(\Omega)}\,.
	\]
 	We deduce from \eqref{eq:ineq.a} that
	\[
	\lambda\|\chi_h\psi_h\|^2_{L^2(\Omega)}
	\geq
\mathscr{Q}_{c,h}(\chi_{1,h,R}\chi_h\psi_h)+	\mathscr{Q}_{c,h}(\chi_{2,h,R}\chi_h\psi_h)
	-h\|\psi_h \chi_h\|_{L^2(\Omega)}^2
	\left(\gamma^2 + 8{c}^2R^{-2}\right)\,,
	\]
Since $\mathscr{Q}_{a,h}(\chi_{2,h,R}\chi_h\psi_h)\geq 0$, we get
	\[
\lambda\|\chi_h\psi_h\|^2_{L^2(\Omega)}
\geq
2b_0h\|\chi_{1,h,R}\chi_h\psi_h\|^2
-h\|\psi_h \chi_h\|_{L^2(\Omega)}^2
\left(\gamma^2 + 8{c}^2R^{-2}\right)\,,
\]
Then,
\begin{multline*}\left(2b_0h-\lambda-h\left(\gamma^2+\frac{8c^2}{R^2}\right)\right)\|\chi_{1,h,R}\chi_h\psi_h\|^2\\
\leq \lambda\|\chi_{2,h,R}\chi_h\psi_h\|^2+h\left(\gamma^2+\frac{8c^2}{R^2}\right)\|\chi_{2,h,R}\chi_h\psi_h\|^2\,.\end{multline*}
Using the assumption on $\lambda$ and the support of $\chi_{2,h,R}$, we get
\begin{multline*}h\left(\varepsilon_0-\left(\gamma^2+\frac{8c^2}{R^2}\right)\right)\|\chi_{1,h,R}\chi_h\psi_h\|^2\\
\leq \lambda\|\chi_{2,h,R}\chi_h\psi_h\|^2+h\left(\gamma^2+\frac{8c^2}{R^2}\right)\|\chi_{2,h,R}\chi_h\psi_h\|^2\leq C(R)h\|\psi_h\|^2\,.
\end{multline*}
With $\gamma\in(0,\sqrt{\varepsilon_0})$, and choosing $R$ such that
\[\varepsilon_0-(\gamma^2+\frac{8c^2}{R^2})>0\,,\]
we infer the existence of $\tilde c>0$ such that
\[\tilde c\|\chi_{1,h,R}\chi_h\psi_h\|^2\leq C(R)\|\psi_h\|^2\,.\]
	We deduce that%
	\[
	\|\chi_h\psi_h\|_{L^2(\Omega)}\leq C\|\psi_h\|_{L^2(\Omega)}\,,
	\]
	and the conclusion follows by coming back to \eqref{eq:ineq.a}.
\end{proof}

\begin{remark}\label{rem.B=b0}
When $B=b_0=1$ is constant, by Proposition \ref{prop.Lambdah}, for all $a>0$,
\[ \gamma_1(ah^{\frac{1}{2}},h)=h \Lambda(a)+o(h) = h \min (2,\nu(a))+o(h)\,.
\]
By Proposition \ref{prop.nu}, $\nu(a)<2$ for all $a>0$ so that $\Lambda(a) = \nu(a)$.
 
Thus, for any $0<\varepsilon<2-\nu(a)$,
there exist $h_0>0$ such that, for all $h\in(0,h_0)$,
\[\gamma_1(ah^{\frac{1}{2}},h)\leq (2b_0-\varepsilon)h\,.\]
With the same proof as the one of Theorem \ref{thm.main'}, we also have for all $n\geq 1$,
\[\gamma_n(ah^{\frac{1}{2}},h)=h \nu(a)+o(h)\,,
\]
and thus up to choosing a smaller $h_0>0$ we get for all $0<h<h_0$,
\[\gamma_n(ah^{\frac{1}{2}},h)\leq (2b_0-\varepsilon)h\,,\]
and the hypothesis of Proposition \ref{red.bound} are satisfied for the $n$-first eigenvalues of $\mathscr{L}_{a,h}$.
\end{remark}
From now on, we assume that $B=b_0=1$.
Let us fix $0<\varepsilon<2-\nu(a)$.
Consider
\[N_h(a)=\{n\geq 1 : \gamma_n(ah^{\frac{1}{2}},h)\leq (2b_0-\varepsilon)h\}\,.\]
 With Remark \ref{rem.B=b0}, we see that, 
for $h$ small enough, $N_h(a)$ contains any given $n\geq 1$.

Let us consider the operator $\widetilde{\mathscr{L}}_{a,h}$ acting on the square integrable functions of the small neighborhood of the boundary $\Omega_\delta=\varphi(S_\delta)$, with $S_\delta=\mathbb{R}/(|\partial\Omega|\mathbb{Z})\times(0,\delta)$, and where $x=\varphi(s,t)$ corresponds to the holomorphic tubular coordinates given in Section \ref{app.holo} (where we explain how to construct these coordinates in such a way that $|\partial_s\varphi(s,0)|=1$).
The operator $\widetilde{\mathscr{L}}_{a,h}$ acts on $L^2(\Omega_\delta)$ as
\begin{equation}\label{eq.originaloperatortrunc}
(-ih\nabla+\mathbf{A})^2+hB\,,
\end{equation}
and the boundary conditions are
\begin{equation}\label{eq.BCtrund}
\begin{split}
-i\overline{n}\dMhmX u=a h^{\frac 12}u\,,\quad &\text{ on }\partial \Omega\,,\\
u= 0\,,\quad &\text{ on }\partial\Omega_\delta\setminus\partial\Omega\,.
\end{split}
\end{equation}
We denote by $(\widetilde\gamma_n(ah^{\frac{1}{2}},h))_{n\geq 1}$ the increasing sequence of the eigenvalues of the operator $\widetilde{\mathscr{L}}_{a,h}$ counted with multiplicity.

We take $\delta=h^{\frac 12-\eta}$ with $\eta\in\left(0,\frac 12\right)$.

As a consequence of Proposition \ref{red.bound} and Remark \ref{rem.B=b0}, we have the following.
\begin{corollary}\label{cor.redtub}
	For all $n\geq 1$, we have 
	\[\gamma_n(ah^{\frac{1}{2}},h)\leq \widetilde\gamma_n(ah^{\frac{1}{2}},h)\,,\]
	and
	\[\widetilde\gamma_n(ah^{\frac{1}{2}},h)\leq \gamma_n(ah^{\frac{1}{2}},h)+\mathscr{O}(h^\infty)\,,\]
	uniformly in $a$, for $a$ in any interval $(0,M)$ with $M>0$.	
\end{corollary}

\subsection{An operator near the boundary}
In this section, we write the operator $\widetilde{\mathscr{L}}_{a,h}$ in holomorphic tubular coordinates.
\subsubsection{Tubular coordinates}
On the neighborhood $\Omega_\delta$, we use the holomorphic boundary coordinates
\[x=\varphi(s,t)\,,\quad  (s,t)\in S_\delta=\mathbb{R}/(|\partial\Omega|\mathbb{Z})\times(0,\delta)\,.\]

\begin{lemma}\label{lem.dzbarvarphi}
	We have
	\[-2ih\partial_{\overline{z}}+A_1+iA_2=\varphi'|\varphi'|^{-2}\left(hD_s+ihD_t+\tilde A_1+i\tilde A_2\right)\,,\quad 
	\tilde A(s,t)=(d\varphi)^{\mathrm{T}}\circ A\circ\varphi(s,t)\,,\]
	where $D=-i\partial$. Moreover, we have
	\[\partial_s\tilde A_2(s,t)-\partial_t\tilde A_1(s,t)=|\varphi'(s+it)|^2\,.\]
\end{lemma}
\begin{proof}
	We have
	\[\begin{pmatrix}
	\partial_s\\
	\partial_t
	\end{pmatrix}=\begin{pmatrix}
	\partial_s\varphi_1&\partial_s\varphi_2\\
	\partial_t\varphi_1&\partial_t\varphi_2
	\end{pmatrix}\begin{pmatrix}
	\partial_1\\
	\partial_2
	\end{pmatrix}\,,\]
	and also
	\[\begin{pmatrix}
	\partial_1\\
	\partial_2
	\end{pmatrix}=|\varphi'|^{-2}\begin{pmatrix}
	\partial_t\varphi_2&-\partial_s\varphi_2\\
	-\partial_t\varphi_1&\partial_s\varphi_1
	\end{pmatrix}\begin{pmatrix}
	\partial_s\\
	\partial_t
	\end{pmatrix}\,,\]
	where $\varphi'$ stands for $\varphi'(s+it)$.
	We deduce that
	\[2|\varphi'|^2\partial_{\overline{z}}=\partial_t\varphi_2\partial_s-\partial_s\varphi_2\partial_t+i(-\partial_t\varphi_1\partial_s+\partial_s\varphi_1\partial_t)=(\partial_t\varphi_2-i\partial_t\varphi_1)\partial_s+(-\partial_s\varphi_2+i\partial_s\varphi_1)\partial_t\,.\]
	Thus,
	\[2|\varphi'|^2\partial_{\overline{z}}=(\partial_s\varphi_1-i\partial_t\varphi_1)\partial_s+(-\partial_t\varphi_1+i\partial_s\varphi_1)\partial_t=(\partial_s\varphi_1+i\partial_s\varphi_2)\partial_s+i(\partial_s\varphi_1+i\partial_s\varphi_2)\partial_t\,,\]
	so that
	\[2|\varphi'|^2\partial_{\overline{z}}=\partial_s\varphi(\partial_s+i\partial_t)\,.\]
	Note also that $\partial_s\varphi(s,t)=\varphi'(s+it)$ and thus
	\[2\partial_{\overline{z}}=\varphi'|\varphi'|^{-2}(\partial_s+i\partial_t)\,.\]
	Recalling that
	\[(d\varphi)^{\mathrm{T}}=
	\begin{pmatrix}
	\partial_s\varphi_1&\partial_s\varphi_2\\
	\partial_t\varphi_1&\partial_t\varphi_2
	\end{pmatrix}
	\,,\]
	we have
	\[
	\begin{pmatrix}
	A_1\\
	A_2
	\end{pmatrix}=|\varphi'|^{-2}
	\begin{pmatrix}
	\partial_t\varphi_2&-\partial_s\varphi_2\\
	-\partial_t\varphi_1&\partial_s\varphi_1
	\end{pmatrix}
	\begin{pmatrix}
	\tilde A_1\\
	\tilde A_2
	\end{pmatrix}
	\,.\]
	Hence
	\[\begin{split}
	A_1+iA_2&=|\varphi'|^{-2}\left(\partial_t\varphi_2\tilde A_1-\partial_s\varphi_2\tilde A_2+i(-\partial_t\varphi_1\tilde A_1+\partial_s\varphi_1\tilde A_2)\right)\\
	&=|\varphi'|^{-2}\left((\partial_t\varphi_2-i\partial_t\varphi_1)\tilde A_1+(-\partial_s\varphi_2+i\partial_s\varphi_1)\tilde A_2\right)\\
	&=|\varphi'|^{-2}\partial_s\varphi\left(\tilde A_1+i\tilde A_2\right)=|\varphi'|^{-2}\varphi'\left(\tilde A_1+i\tilde A_2\right)\,.
	\end{split}\]
	The conclusion follows.
\end{proof}
We recall that that quadratic form associated with $\widetilde{\mathscr{L}}_{a,h}$ is given by
\[\mathscr{Q}_{a,h}(u)=\int_{\Omega_\delta}|(-2ih\partial_{\overline{z}}+A_1+iA_2)u|^2\dd x+ah^{\frac32}\int_{\partial\Omega}|u|^2\dd \sigma\,.\]
With Lemma \ref{lem.dzbarvarphi}, we deduce that
\[\mathscr{Q}_{a,h}(u)=\widetilde{\mathscr{Q}}_{a,h}(\tilde u)=\int_{S_\delta}|(hD_s+ihD_t+\tilde A_1+i\tilde A_2)\tilde u|^2\dd s\dd t+ah^{\frac32}\int_0^{|\partial\Omega|} |\tilde u(s,0)|^2\dd s\,,\]
where $\tilde u(s,t)=u\circ\varphi(s,t)$, and where we used that $|\partial_s\varphi(s,0)|=1$ to deal with the boundary term (see Appendix \ref{app.holo}). The ambiant Hilbert space is now $L^2(|\varphi'|^2\dd s\dd t)$. To go to the flat $L^2$-space, we let $\hat u=\varphi'\tilde u$. Since $\varphi'$ is holomorphic, it commutes with the $\partial_s+i\partial_t$, and we deduce that
\[\mathscr{Q}_{a,h}(u)=\widehat{\mathscr{Q}}_{a,h}(\hat u)=\int_{S_\delta}|\varphi'|^{-2}(hD_s+ihD_t+\tilde A_1+i\tilde A_2)\hat u|^2\dd s\dd t+ah^{\frac32}\int_0^{|\partial\Omega|} |\hat u(s,0)|^2\dd s\,.\]
The operators $\widetilde{\mathscr{L}}_{a,h}$ and $\widehat{\mathscr{L}}_{a,h}$ are unitarily equivalent.

\subsubsection{Change of gauge}
Let us now use an appropriate change of gauge to cancel $\tilde A_2$. Consider
\[\psi_1(s,t)=\int_0^t\tilde{A}_2(s,\tau)\dd\tau\,,\quad \widehat{\mathbf{A}}=\tilde{\mathbf{A}}-\nabla\psi_1\,. \]
Notice that
\[\hat A_1(s,t)=\tilde A_1(s,t)-\int_0^t\partial_s \tilde A_2(s,u)\dd u\,,\quad \hat A_2(s,t)=0\,.\]
Clearly, with Lemma \ref{lem.dzbarvarphi},
\[-\partial_t\hat A_1=\partial_s\tilde A_2-\partial_t\tilde A_1=|\varphi'|^2\,.\]
so that, we have
\[
\hat A_1(s,t)=
\hat A_1(s,0) - \int_0^t|\varphi'(s+ir)|^2\dd r\,.
\]
Now, consider 
\[\psi_2(s)=\int_0^s\left(\hat A_1(u,0) -\frac{1}{|\partial\Omega|}\int_0^{|\partial\Omega|}\hat A_1(v,0)\dd v\right)\dd u\,.\]
The function $\psi_2$ is $|\partial\Omega|$-periodic. We let 
\[\check{\mathbf{A}}=\widehat{\mathbf{A}}-\nabla\psi_2=\tilde{\mathbf{A}}-\nabla\psi\,,\]
where 
\[\psi=\psi_1+\psi_2\,.\]
We find that
\[\check A_1(s,t)
=\frac{1}{|\partial\Omega|}\int_0^{|\partial\Omega|}\hat A_1(v,0)\dd v -\int_0^t|\varphi'(s+ir)|^2\dd r\,,\quad 
\check A_2(s,t)=0\,.\]
By the Green-Riemann formula,
\[
\begin{split}
\int_0^{|\partial\Omega|}\hat A_1(v,0)\dd v
=\int_0^{|\partial\Omega|}\tilde A_1(v,0)\dd v=\int_{\partial\Omega}\mathbf{A}\cdot \gamma'\dd s
=\int_{\Omega}\mathrm{curl} \mathbf{A}\,\dd x=|\Omega|\,.
\end{split}
\]
Thus,
\[\check A_1(s,t)
=\gamma_0 -\int_0^t|\varphi'(s+ir)|^2\dd r\,,\quad 
\check A_2(s,t)=0\,,\quad \gamma_0=\frac{|\Omega|}{|\partial\Omega|}\,.\]
Letting $\hat u=e^{-i\psi/h}\check u$, we get
\[\widehat{\mathscr{Q}}_{a,h}(\hat u)=\int_{S_\delta}|\varphi'|^{-2}(hD_s+ihD_t+\check A_1)\check u|^2\dd s\dd t+ah^{\frac32}\int_0^{|\partial\Omega|} |\check u(s,0)|^2\dd s\,.\]
The associated operator is
\[
\mathfrak{M}_{a,h}=(-ih\partial_s-h\partial_t+\check A_1(s,t))|\varphi'|^{-2}(-ih\partial_s+h\partial_t+\check A_1(s,t))
\]
and the boundary conditions are
\[
\left(-ih\partial_s+\gamma_0
	+h\partial_t
	\right) u(s,0) = ah^{1/2} u(s,0)\,,\quad  u(s,\delta)=0\,.
\]
\begin{remark}
	By unitary equivalence, the eigenvalues of $\mathfrak{M}_{a,h}$ are $(\widetilde\gamma_n(ah^{\frac{1}{2}},h))_{n\geq 1}$.
\end{remark}

\subsubsection{Rescaling}
We let $t=\hbar \tau$ with $\hbar=h^{\frac 12}$. We also divide the operator by $h$ and we get the operator, called $\mathscr{M}_{a,\hbar}$, acting on the Hilbert space $L^2(S_{\delta\hbar^{-1}}, \dd s\dd\tau)$ as
\[
\mathscr{M}_{a,\hbar}
=
(-i\hbar\partial_s-\partial_\tau+\hbar^{-1}\check A_1(s,\hbar \tau))|\varphi'(s+i\hbar\tau)|^{-2}(-i\hbar\partial_s+\partial_\tau+\hbar^{-1}\check A_1(s,\hbar \tau))\,,
\]
and the boundary conditions are
\[
\left(-i\hbar\partial_s+\frac{\gamma_0}{\hbar} 
	+\partial_\tau
	\right) u(s,0) = a u(s,0)\,,\quad u(s,\delta\hbar^{-1})=0\,.
\]
\subsubsection{Another change of gauge}\label{sec.changeofgauge}
For all $m\in\mathbb{Z}$, the function $s\mapsto e^{2im\pi\frac{s}{|\partial\Omega|}}$ is $|\partial\Omega|$-periodic. Thus, the operators
\[\mathscr{M}_{a,\hbar, m}:=e^{2i\pi m\frac{s}{|\partial\Omega|}}\mathscr{M}_{a,\hbar} e^{-2i\pi m\frac{s}{|\partial\Omega|}}\,,\]
are unitarily equivalent to $\mathscr{M}_{a,\hbar}$ and act as
\[\mathscr{M}_{a,\hbar,m}=(-i\hbar\partial_s-\partial_\tau+A_{1,\hbar,m}(s,\tau))|\varphi'(s+i\hbar\tau)|^{-2}(-i\hbar\partial_s+\partial_\tau+A_{1,\hbar,m}(s,\tau))\,,\]
where
\[A_{1,\hbar,m}(s,\tau)=-\frac{2\pi m\hbar}{|\partial\Omega|}+\hbar^{-1}\check A_1(s,\hbar \tau)\,.\]

The boundary conditions are
\[
\left(-i\hbar\partial_s+\left(\frac{\gamma_0}{\hbar}-\frac{2\pi m\hbar}{|\pa\Omega|}\right)
	+\partial_\tau
	\right) u(s,0) = a u(s,0)\,, \quad u(s,\delta\hbar^{-1})=0\,.
\]
Let us make a particular choice of $m$. 
This choice is made so that 
$\frac{\gamma_0}{\hbar}-\frac{2m\pi\hbar}{|\partial\Omega|}$ is the closest possible to $0$. Consider
\[d_{\hbar}:=\min_{m\in\mathbb{Z}} \left|\frac{|\partial\Omega|}{2\pi\hbar}\frac{\gamma_0}{\hbar}-m\right|\in\left[0,\frac 12\right]\,.\]
Let us denote by $m_{\hbar}$ the smallest minimizer (there are at most two minimizers), and let
\[r_{\hbar}:=\frac{|\partial\Omega|}{2\pi\hbar}\frac{\gamma_0}{\hbar}-m_{\hbar}\,.\]
We have $d_{\hbar}=|r_{\hbar}|$ and we can write
\[\frac{\gamma_0}{\hbar}-\frac{2 m_{\hbar}\pi\hbar}{|\partial\Omega|}=\hbar\theta\,,\quad \theta=\frac{2\pi}{|\partial\Omega|}r_{\hbar}\,.\]
The constant $\theta$ is uniformly bounded:
\[|\theta|\leq \frac{\pi }{|\partial\Omega|} \,.\]
With this choice of $m_\hbar$, we get the new (self-adjoint) operator
\begin{equation}\label{eq.Nahbar}
\mathscr{N}_{a,\hbar}=(-i\hbar\partial_s-\partial_\tau+A_{1,\hbar,m_\hbar})|\varphi'(s+i\hbar\tau)|^{-2}(-i\hbar\partial_s+\partial_\tau+A_{1,\hbar,m_\hbar})\,,
\end{equation}
where
\[A_{1,\hbar,m_\hbar}(s,\tau)=\hbar\theta-\hbar^{-1}\int_0^{\hbar\tau}|\varphi'(s,u)|^2\dd u\,,
\]
and
\begin{equation}\label{eq.theta}
\theta=\frac{2\pi r_{\hbar}}{|\partial\Omega|}=\frac{|\Omega|}{\hbar^2|\partial\Omega|}-\frac{2m_{\hbar}\pi}{|\partial\Omega|}\,.
\end{equation}
The boundary conditions are
\[
\left(-i\hbar\partial_s+\hbar\theta
	+\partial_\tau
	\right) u(s,0) = a u(s,0)\,,\quad u(s,\delta\hbar^{-1})=0\,.
\]
\begin{remark}
	The eigenvalues of $\mathscr{N}_{a,\hbar}$ are the $(h^{-1}\widetilde\gamma_n(ah^{\frac{1}{2}},h))_{n\geq 1}$.
\end{remark}

\section{Microlocal dimensional reduction}\label{sec.7}
Let us now focus on the spectral properties of $\mathscr{N}_{a,\hbar}$ defined in \eqref{eq.Nahbar}, and its boundary conditions are
with the boundary conditions
\[\left(-i\hbar\partial_s+\partial_\tau\right)u(s,0)=\alpha u(s,0)\,,\quad u(s,\delta\hbar^{-1})=0\,,\quad \alpha:=a-\theta\hbar\,.\]
We underline that $\theta$ depends on $\hbar$ and $\alpha$, but that it is uniformly bounded.
In what follows, $\theta$ will be consider a parameter.

\subsection{Inserting cutoff functions and pseudo-differential interpretation}

\subsubsection{Cutoff with respect to the normal variable}\label{sec.}
We can prove that the first eigenfunctions of $\mathscr{N}_{a,\hbar}$ satisfy Agmon estimates with respect to $\tau$ (in Proposition \ref{red.bound}, $\mathrm{dist}(x,\partial\Omega)$ can essentially be replaced by $t(x)$ near the boundary since we have \eqref{eq.td}).

This leads to consider the operator, acting on $L^2(\mathbb{R}/(|\partial\Omega|\mathbb{Z})\times(0,+\infty))$,
\[\widetilde{\mathscr{N}}_{a,\hbar}=(-i\hbar\partial_s-\partial_\tau+A^\chi_{1,\hbar,m_\hbar})|\varphi'(s+i\hbar\chi_\hbar(\tau)\tau)|^{-2}(-i\hbar\partial_s+\partial_\tau+A^\chi_{1,\hbar,m_\hbar})\,,\]
with 
\[A^\chi_{1,\hbar,m_\hbar}(s,\tau)=\hbar\theta-\tau-\chi_\hbar(\tau)\hbar^{-1}\left(\int_0^{\hbar\tau}|\varphi'(s,u)|^2\dd u-\hbar\tau\right)\,\]
and $\chi_\hbar(\tau)=\chi(\tau\hbar^{\eta})$ where $\chi$ is a smooth cutoff function equaling 1 near 0, and $\eta>0$ being as small as necessary.

We can then check that the low-lying eigenvalues of $\mathscr{N}_{a,\hbar}$ coincide with those of $\widetilde{\mathscr{N}}_{a,\hbar}$ modulo $\mathscr{O}(\hbar^\infty)$. Let us drop the tildas to lighten the presentation.

\subsubsection{Pseudo-differential interpretation}
This operator $\widetilde{\mathscr{N}}_{a,\hbar}$ can be seen as a pseudo-differential operator with operator symbol (for more detail, the reader can consult the Ph.D. thesis by Keraval \cite{Keraval}, or the paper by Martinez \cite{M07}). To describe its symbol, let us consider some Taylor expansions (see Lemma \ref{lem.Taylorphi}):
\[|\varphi'(s+i\hbar\chi_\hbar(\tau)\tau)|^{-2}=1+2\hbar\kappa\chi_\hbar\tau+2\hbar\kappa^2\chi_\hbar^2\tau^2+o(\hbar^2)\,,\]
\[A^\chi_{1,\hbar,m_\hbar}(s,\tau)=\hbar\theta-\tau+\hbar\kappa\chi_\hbar\tau^2-\frac23\hbar^2\kappa^2\chi_\hbar\tau^3+o(\hbar^2)\,.\]
Then, we have
\begin{equation}\label{eq.Nahbar2}
\mathscr{N}_{a,\hbar}=\mathrm{Op}^{\mathrm{W}}_\hbar(n_\hbar)\,,\quad n_\hbar=n_0+\hbar n_1+\hbar^2 n_2+\ldots
\end{equation}
where the first symbols are given by
\begin{equation}\label{eq.nj}
\begin{split}
n_0=&(\xi-\partial_\tau-\tau)(\xi+\partial_\tau-\tau)=-\partial^2_\tau+(\xi-\tau)^2+1\,,\\
n_1=&\kappa\left(2(\xi-\tau-\partial_\tau)\chi_\hbar\tau(\xi-\tau+\partial_\tau)+\chi_{\hbar}\tau^2(\xi-\tau+\partial_\tau)+(\xi-\tau-\partial_\tau)(\chi_{\hbar}\tau^2)\right)\\
&+2\theta(\xi-\tau)\,,\\
n_2=&\kappa^2\Big[2(-\partial_\tau+\xi-\tau)\chi_{\hbar}^2\tau^2(\partial_\tau+\xi-\tau)\\
&+\frac43\left((-\partial_\tau+\xi-\tau)(\chi_{\hbar}\tau^3)+(\chi_{\hbar}\tau^3)(\partial_\tau+\xi-\tau)\right)+\chi_{\hbar}^2\tau^4\Big]\\
&+\chi_{\hbar}\kappa\theta\left(2\tau^2-2+4\tau(\xi-\tau)\right)+\theta^2-2\chi'_{\hbar}\kappa\theta\tau\,,
\end{split}
\end{equation}
where $n_0$ is equipped with the boundary condition 
\[(\partial_\tau+\xi-\alpha)\psi=0\,.\]
Actually, we have
\[n_1=\kappa\chi_{\hbar}\mathscr{C}_\xi+2\theta(\xi-\tau)-\kappa\chi'_{\hbar}(2\tau(\xi+\partial_\tau-\tau)+\tau^2)\,,\quad \mathscr{C}_\xi=2\left(\tau n_0-\xi-\partial_\tau+\tau^2(\xi-\tau)\right)\,,\]
We also notice that
\begin{multline*}
n_2=\kappa^2\left[-4\tau\chi_{\hbar}^2(\partial_\tau+\xi-\tau)+2\tau^2\chi^2_{\hbar} n_0+\frac83(\xi-\tau)\chi_{\hbar}\tau^3-4\chi_{\hbar}\tau^2+\chi_\hbar^2\tau^4\right]\\
+\chi_{\hbar}\kappa\theta\left(2\tau^2-2+4\tau(\xi-\tau)\right)+\theta^2-2\chi'_{\hbar}\kappa\theta\tau-4\chi_{\hbar}\chi'_{\hbar}\tau^2(\partial_\tau+\xi-\tau)-\frac43\chi'_{\hbar}\tau^3\,.
\end{multline*}
\begin{remark}
	In our periodic framework, the usual Weyl quantization formula on $\mathbb{R}$ may be expressed by means of Fourier series (on the torus $\mathbb{R}/(2L\mathbb{Z})$):
	\[\mathrm{Op}^{\mathrm{W}}_\hbar(p)\psi(x)=\sum_{(k,j)\in\mathbb{Z}^2}e^{ix(j+k)\omega}\widehat {p}\left(j,\omega\frac{j\hbar}{2}+\omega\hbar k\right)\widehat{\psi}(k)\,,\quad \omega=\frac{2\pi}{2L}\,,\]
	where the Fourier coefficient is defined by
	\[\widehat{\psi}(k)=\frac{1}{2L}\int_0^{2L}\psi(x)e^{-ik\omega x}\dd x\,.\]
	Let us briefly recall where this formula comes from. We have
	\[\begin{split}
	\mathrm{Op}^{\mathrm{W}}_\hbar(p)\psi(x)&=\frac{1}{2\pi\hbar}\int_{\R^{2}}e^{i(x-y)\eta/\hbar}p\left(\frac{x+y}{2},\eta\right)\psi(y)\dd y\dd \eta\\
	&=\frac{1}{2\pi\hbar}\sum_{(j,k)\in\mathbb{Z}^2}\widehat{\psi}(k)\int_{\R^2}e^{i(x-y)\eta/\hbar}\widehat{p}(j,\eta)e^{ij\omega\frac{x+y}{2}}e^{ik\omega y}\dd y\dd\eta\\
	&=\frac{1}{2\pi\hbar}\sum_{(j,k)\in\mathbb{Z}^2}e^{ij\omega\frac{x}{2}}\widehat{\psi}(k)\int_{\R}\dd y e^{i\omega y(k+\frac{j}{2})}\int_{\R}\dd \eta e^{-i(y-x)\eta/\hbar}\widehat{p}(j,\eta)\\
	&=\frac{1}{2\pi\hbar}\sum_{(j,k)\in\mathbb{Z}^2}e^{ij\omega\frac{x}{2}}\widehat{\psi}(k)\int_{\R}\dd y e^{i\omega y(k+\frac{j}{2})}\mathscr{F}\widehat{p}(j,\frac{y-x}{\hbar})\\
	&=\sum_{(j,k)\in\mathbb{Z}^2}e^{ij\omega\frac{x}{2}}e^{i\omega x(k+\frac{j}{2})}\widehat{\psi}(k)\frac{1}{2\pi}\int_{\R}\dd z e^{i\hbar\omega z(k+\frac{j}{2})}\mathscr{F}\widehat{p}(\eta,z)\\
	&=\sum_{(j,k)\in\mathbb{Z}^2}e^{ij\omega\frac{x}{2}}e^{i\omega x(k+\frac{j}{2})}\widehat{\psi}(k)\widehat{p}\left(j,\hbar\omega(k+\frac j2)\right)\,.
	\end{split}\]
	When $p$ only depends on $\xi$, $p(x,\xi)=\chi(\xi)$, this formula becomes
	\[\mathrm{Op}^{\mathrm{W}}_\hbar(p)\psi(x)=\sum_{k\in\mathbb{Z}}e^{ix k\omega}\chi(\omega\hbar k)\widehat{\psi}(k)\,.\]
\end{remark}

\subsubsection{Microlocal cutoff}\label{sec.microcut}
Then, we insert a \enquote{cutoff} function with respect to $\xi$, and we consider
\[\check{\mathscr{N}}_{a,\hbar}=\mathrm{Op}^{\mathrm{W}}_\hbar(\check n_\hbar)\,,\quad \check n_\hbar(s,\xi)=n_\hbar(s,\chi_0(\xi)\xi)\,,\]
where $\chi_0$ is a smooth function equalling $1$ near $\xi_{a}$ and so that $\xi\mapsto \nu(a,\xi\chi_0(\xi))$ has still a unique minimum at $\xi_a$. In this way, $\check n_\hbar$ belongs to a suitable symbol class (essentially, this means that everything is going as if $\check n_\hbar\in S(1)$, $S(1)$ being the class of bounded symbols), see \cite[Section 3]{BHR19} and \cite{Keraval} where similar classes are used.
\begin{proposition}
	The low-lying eigenvalues of $\mathscr{N}_{a,\hbar}$ and those of $\check{\mathscr{N}}_{a,\hbar}$ coincide module $\mathscr{O}(\hbar^\infty)$.	
\end{proposition}
\begin{proof}
	The key is to prove that the eigenfunctions of $\widetilde{\mathscr{N}}_{a,\hbar}$, which is a perturbation of $\mathrm{Op}^{\mathrm{W}}_\hbar n_0$, are microlocalized near $\xi_a$. This fact comes from the behavior of the principal (operator) symbol $n_0$ (its first eigenvalue, as a function of $\xi$ has a unique minimum, which is not attained at infinity). A similar analysis can be found in \cite[Section 5]{BHR19}.	
\end{proof}

\begin{corollary}\label{cor.approxlambdan}
	For all $n\geq 1$, we have
	\[h^{-1}\lambda_n(a,h)=\lambda_n(\check{\mathscr{N}}_{a,\hbar})+\mathscr{O}(\hbar^\infty)\,,\]
	uniformly with respect to $a$ in a bounded interval.	
\end{corollary}

\subsection{Construction of a parametrix}
Let us now work on the operator with the cutoff functions.

\begin{lemma}
	Consider the operator
	\[\mathscr{P}_0=\begin{pmatrix}\check n_0-z&\cdot \check u_{\alpha,\xi} \\\langle\cdot,\check u_{\alpha,\xi}\rangle&0 \end{pmatrix}\,,\quad \check u_{\alpha,\xi}=u_{\alpha,\xi\chi_0(\xi)}\,.\]
	For $z$ sufficently close to $\nu(a)=\nu(a,\xi_a)$,  $\mathscr{P}_0$ is bijective and
	\[\mathscr{Q}_0:=\mathscr{P}^{-1}_0=\begin{pmatrix}p_0^{-1}&\cdot \check u_{\alpha,\xi} \\\langle\cdot,\check u_{\alpha,\xi}\rangle&z-\nu(\alpha,\xi\chi_0(\xi)) \end{pmatrix}\,,\quad p_0^{-1}=(n_0-z)^{-1}\Pi^\perp\,.\]
\end{lemma}
The aim of this section is to prove the following proposition.
\begin{proposition}\label{prop.parametrix}
	We let
	\[\mathscr{P}_\hbar=\mathrm{Op}^{\mathrm{W}}_\hbar\begin{pmatrix}\check n_\hbar-z&\cdot \check u_{\alpha,\xi}\\\langle\cdot,\check u_{\alpha,\xi}\rangle&0 \end{pmatrix}\,,\]
	and, for $j\geq 1$,
	\[\mathscr{P}_j=\begin{pmatrix}\check n_j&0\\0&0 \end{pmatrix}\,.\]	
	Consider the operator symbols defined by
	\[\mathscr{Q}_1=-\mathscr{Q}_0\mathscr{P}_1\mathscr{Q}_0\,,\]
	and
	\[\mathscr{Q}_2=-\mathscr{Q}_0\mathscr{P}_2\mathscr{Q}_0-\mathscr{Q}_1\mathscr{P}_1\mathscr{Q}_0-\frac{1}{2i}\left(\partial_\xi\mathscr{Q}_0\cdot\partial_s\mathscr{P}_1-\partial_s\mathscr{Q}_1\cdot\partial_\xi\mathscr{P}_0\right)\mathscr{Q}_0\,.\]
	Then, we have, for some $N_3\in\mathbb{N}$,
	\[\mathrm{Op}^{\mathrm{W}}_\hbar(\mathscr{Q}_0+\hbar\mathscr{Q}_1+\hbar^2\mathscr{Q}_2)\mathscr{P}_\hbar=\mathrm{Id}+\hbar^3\mathscr{O}(\langle\tau\rangle^{N_3})\,.\]
	More generally, we can find $(\mathscr{Q}_j)_{1\leq j\leq J}$, and $N_J\in\mathbb{N}$ such that
	\begin{equation}\label{eq.inductionQJ}
	\mathrm{Op}^{\mathrm{W}}_\hbar(\mathscr{Q}_0+\hbar\mathscr{Q}_1+\hbar^2\mathscr{Q}_2+\ldots+\hbar^J\mathscr{Q}_J)\mathscr{P}_\hbar=\mathrm{Id}+\hbar^{J+1}\mathscr{O}(\langle\tau\rangle^{N_J})\,.
	\end{equation}
	Moreover, the bottom right coefficient of \[\mathscr{Q}_0+\hbar\mathscr{Q}_1+\hbar^2\mathscr{Q}_2\]
	is 
	\[q_h^{\pm}=z-\nu(\alpha,\xi)-\hbar \langle \check n_1 \check u_{\alpha,\xi},\check u_{\alpha,\xi}\rangle +\hbar^2\langle (\check n_1p_0^{-1}\check n_1-\check n_2(s,\xi\chi_0(\xi))) \check u_{\alpha,\xi},\check u_{\alpha,\xi} \rangle\,.\]
	
\end{proposition}

\begin{proof}
	We have 
	\[\mathscr{Q}_0\mathscr{P}_0=\mathrm{Id}\,.\]
	Let us consider the operator $\mathscr{Q}_1$ defined by the relation:	
	\[\mathscr{Q}_0\mathscr{P}_1+\mathscr{Q}_1\mathscr{P}_0+\frac{1}{2i}\{\mathscr{Q}_0,\mathscr{P}_0\}=0\,,
	\]
	Note that $\{\mathscr{Q}_0,\mathscr{P}_0\}=0$ since $\mathscr{P}_0$ does not depend on $s$.
	Then, consider also $\mathscr{Q}_2$ defined by
	\[\mathscr{Q}_0\mathscr{P}_2+\mathscr{Q}_1\mathscr{P}_1+\mathscr{Q}_2\mathscr{P}_0+\frac{1}{2i}\left(\{\mathscr{Q}_0,\mathscr{P}_1\}+\{\mathscr{Q}_1,\mathscr{P}_0\}\right)=0\,.
	\]
	With these choices, we have, thanks to the Weyl calculus,
	\[\mathrm{Op}^{\mathrm{W}}_\hbar\left(\mathscr{Q}_0+\hbar\mathscr{Q}_1+\hbar\mathscr{Q}_2\right)\mathscr{P}_\hbar=\mathrm{Id}+\mathscr{O}(\hbar^3)\,.\]
	Let us compute the bottom right coefficient of
	\[\mathscr{Q}_0+\hbar\mathscr{Q}_1+\hbar\mathscr{Q}_2\,.\]
	We have
	\[\mathscr{Q}_1=-\mathscr{Q}_0\mathscr{P}_1\mathscr{Q}_0=-\begin{pmatrix}
	p_0^{-1}\check n_1 p_0^{-1}&p_0^{-1}\check n_1 q_0^+\\
	q_0^- \check n_1 p_0^{-1}&q_0^{-}\check n_1 q_0^+
	\end{pmatrix}\,.\]
	Then,
	\[\mathscr{Q}_2=-\mathscr{Q}_0\mathscr{P}_2\mathscr{Q}_0-\mathscr{Q}_1\mathscr{P}_1\mathscr{Q}_0-\frac{1}{2i}\left(\partial_\xi\mathscr{Q}_0\cdot\partial_s\mathscr{P}_1-\partial_s\mathscr{Q}_1\cdot\partial_\xi\mathscr{P}_0\right)\mathscr{Q}_0\,.\]
	Note that
	\[\partial_\xi\mathscr{Q}_0\cdot\partial_s\mathscr{P}_1\cdot\mathscr{Q}_0=\kappa'\begin{pmatrix}\partial_\xi p_0^{-1}& \cdot  \partial_\xi \check u_{\alpha,\xi}\\\langle\cdot,\partial_\xi\check u_{\alpha,\xi}\rangle&-\partial_\xi\nu(\alpha,\xi) \end{pmatrix}\begin{pmatrix}\mathscr{C}_\xi&0\\0&0 \end{pmatrix}\begin{pmatrix}p_0^{-1}&\cdot \check u_{\alpha,\xi} \\\langle\cdot,\check u_{\alpha,\xi}\rangle&z-\nu(\alpha,\xi\chi_0(\xi)) \end{pmatrix}\,.\]
	Then, we have
	\[q_2^{\pm}=-q_0^- \check n_2 q_0^++q_0^-\check n_1p_0^{-1}\check n_1q_0^+-\frac{\kappa'}{2i}(\langle \mathscr{C}_\xi \check u_{\alpha,\xi}, \partial_\xi\check u_{\alpha,\xi}\rangle-\partial_s\mathscr{Q}_1\cdot\partial_\xi\mathscr{P}_0\cdot\mathscr{Q}_0)\,.\]
	Note that
	\[-\partial_s\mathscr{Q}_1\cdot\partial_\xi\mathscr{P}_0\cdot\mathscr{Q}_0=\partial_s\mathscr{Q}_1\cdot\mathscr{P}_0\cdot\partial_\xi\mathscr{Q}_0\,,\]
	and then
	\[\partial_s\mathscr{Q}_1\cdot\mathscr{P}_0\cdot \partial_\xi\mathscr{Q}_0 =-\kappa'\begin{pmatrix}
	p_0^{-1}\mathscr{C} p_0^{-1}&p_0^{-1}\mathscr{C}  q_0^+\\
	q_0^- \mathscr{C}  p_0^{-1}&q_0^{-}\mathscr{C}  q_0^+
	\end{pmatrix}\begin{pmatrix}\check n_0-z&\cdot \check u_{\alpha,\xi}\\\langle\cdot,\check u_{\alpha,\xi}\rangle&0 \end{pmatrix}\begin{pmatrix}\partial_\xi p_0^{-1}& \cdot \partial_\xi\check u_{\alpha,\xi}\\\langle\cdot,\partial_\xi\check u_{\alpha,\xi}\rangle&-\partial_\xi\nu(\alpha,\xi) \end{pmatrix}\,.\]
	We deduce that
	\[\begin{split}
	q_2^{\pm}&=-q_0^- \check n_2 q_0^++q_0^-\check n_1p_0^{-1}\check n_1q_0^+-\frac{\kappa'}{2i}(\langle \mathscr{C}_\xi \check u_{\alpha,\xi}, \partial_\xi\check u_{\alpha,\xi}\rangle-\langle \mathscr{C}_\xi \partial_\xi\check u_{\alpha,\xi}, \check u_{\alpha,\xi}\rangle)\\
	&=-q_0^- \check n_2 q_0^++q_0^-\check n_1p_0^{-1}\check n_1q_0^+\,,
	\end{split}\]
	where we used Lemma \ref{lem.Cxisym}.
	The existence of the $\mathscr{Q}_j$ in \eqref{eq.inductionQJ} can be obtained by induction.
\end{proof}

From now on, we fix $J\geq 2$.
\begin{proposition}\label{prop.distsp1}
	We have
	\[\mathrm{dist}(0,\mathrm{sp}(Q^\pm_{z,J}))\|\psi\|\leq C\left(1+\mathrm{dist}(0,\mathrm{sp}(Q^\pm_{z,J}))\right)(\|(\mathscr{N}_{a,\hbar}-z)\psi\|+\hbar^{J+1}\|\langle\tau\rangle^{N_J}\psi\|)\,.\]	
\end{proposition}
\begin{proof}
	We let
	\[\mathscr{Q}_\hbar=\mathrm{Op}^{\mathrm{W}}_\hbar(\mathscr{Q}_0+\hbar\mathscr{Q}_1+\hbar^2\mathscr{Q}_2+\ldots+\hbar^J\mathscr{Q}_J)=\begin{pmatrix}Q&Q^+\\Q^-&Q^{\pm}\end{pmatrix}\,,\]	
	and recall that
	\[\mathscr{P}_\hbar=\begin{pmatrix}
	\mathscr{N}_{a,\hbar}-z&\Pi^*\\
	\Pi&0
	\end{pmatrix}\,,\quad \Pi^*=\mathrm{Op}_\hbar^{\mathrm{W}}(\cdot u_\xi)\,,\quad \Pi=\mathrm{Op}_\hbar^{\mathrm{W}}\langle\cdot, u_\xi\rangle\,.\]
	We have
	\[\mathscr{Q}_\hbar\mathscr{P}_\hbar=\mathrm{Id}+\hbar^{J+1}\mathscr{O}(\langle\tau\rangle^{N_J})\,.\]
	This implies that
	\[Q^-(\mathscr{N}_{a,\hbar}-z)+Q^\pm \Pi=\hbar^{J+1}\mathscr{O}(\langle\tau\rangle^{N_J})\,,\]
	and
	\[Q(\mathscr{N}_{a,\hbar}-z)+Q^+\Pi=\mathrm{Id}+\hbar^{J+1}\mathscr{O}(\langle\tau\rangle^{N_J})\,.\]	
	We get that
	\[\|\psi\|\leq C\|\Pi\psi\|+C\|(\mathscr{N}_{a,\hbar}-z)\psi\|+C\hbar^{J+1}\|\langle\tau\rangle^{N_J}\psi\|\,,\]
	and
	\[\|Q^\pm\Pi\psi\|\leq C\|(\mathscr{N}_{a,\hbar}-z)\psi\|+C\hbar^{J+1}\|\langle\tau\rangle^{N_J}\psi\|\,.\]
	Then, since $Q^\pm$ is self-adjoint and by using the spectral theorem,
	\[\|\psi\|\leq C\left(1+\frac{1}{\mathrm{dist}(0,\mathrm{sp}(Q^\pm))}\right)(\|(\mathscr{N}_{a,\hbar}-z)\psi\|+\hbar^3\|\langle\tau\rangle^N\psi\|)\,,\]
	so that the conclusion follows.
\end{proof}
\begin{proposition}\label{prop.distsp2}
	We have
	\[\mathrm{dist}(\mathrm{sp}(\mathscr{N}_{a,\hbar}),z)\|\psi\|\leq C\|Q^\pm_{z,J}\psi\|+C\hbar^{J+1}\|\psi\|\,.\]
\end{proposition}

\begin{proof}
	We have
	\[\mathscr{P}_\hbar\mathscr{Q}_\hbar=\mathrm{Id}+\mathscr{O}(\hbar^{J+1})\,.\]
	It follows that
	\[(\mathscr{N}_\hbar-z)Q^++\Pi^*Q^\pm_{z,J}=\mathscr{O}(\hbar^{J+1})\,,\]
	and
	\[\Pi Q^+=\mathrm{Id}+\mathscr{O}(\hbar^{J+1})\,.\]
	We get
	\[\|(\mathscr{N}_{a,\hbar}-z)(Q^+\psi)\|\leq C\|Q^\pm_{z,J}\psi\|+C\hbar^{J+1}\|\psi\|\,,\]
	and
	\[\mathrm{dist}(\mathrm{sp}(\mathscr{N}_{a,\hbar}),z)\|Q^+\psi\|\leq C\|Q^\pm_{z,J}\psi\|+C\hbar^{J+1}\|\psi\|\,.\]
	Then,
	\[\mathrm{dist}(\mathrm{sp}(\mathscr{N}_\hbar),z)\|\psi\|\leq C\|Q^\pm_{z,J}\psi\|+C\hbar^{J+1}\|\psi\|\,.\]
\end{proof}
Let us consider the case when $J=2$, and consider the pseudo-differential operator whose Weyl symbol is
\begin{equation}\label{eq.peff0}
p^{\mathrm{eff}}_\hbar(s,\xi)=\nu(\alpha,\check \xi)+\hbar \langle \check n_1 u_{\alpha,\check\xi},u_{\alpha,\check\xi}\rangle -\hbar^2\langle (\check n_1(\mathscr{N}_{\alpha,\check\xi}-\nu(\alpha,\xi_\alpha))^{-1}\Pi^\perp \check n_1- \check n_2) u_{\alpha,\check\xi},u_{\alpha,\check\xi} \rangle\,,
\end{equation}
where $\check \xi=\xi\chi_0(\xi)$. 
Due to the fact that $\nu(\alpha,\xi)$ has a unique and non-degenerate minimum, we can check that the low-lying eigenfunctions are microlocalized near $\xi_\alpha$.
\begin{lemma}\label{lem.microloc}
	For all normalized eigenfunction $\psi$ of $\mathrm{Op}^{\mathrm{W}}_\hbar p^{\mathrm{eff}}_\hbar$ associated with an eigenvalue $\lambda\leq\nu(\alpha,\xi_\alpha)+C\hbar$, we have
	\[\mathrm{Op}^{\mathrm{W}}_\hbar (\chi_\hbar) \psi=\mathscr{O}(\hbar^\infty)\,,\]
	with $\chi_\hbar(\xi)=\chi(\hbar^{-\frac12+\eta}(\xi-\xi_\alpha))$, where $\chi$ is a smooth function equal to $1$ away from a fixed neighborhood of $0$, and equal to $0$ near $0$.
	
	Moreover, we have
	\[\lambda^{\mathrm{eff}}_n(\alpha,\hbar)=\nu(\alpha,\xi_\alpha)+\mathscr{O}(\hbar)\,.\]
\end{lemma}
In particular, Lemma \ref{lem.microloc} tells us that the first eigenvalues lie in $D(\nu(\alpha,\xi_\alpha),C\hbar)$.

From Propositions \ref{prop.distsp1} and \ref{prop.distsp2}, we deduce the following.
\begin{proposition}\label{prop.distsp}
	For $z\in D(\nu(\alpha,\xi_\alpha),C\hbar)$,
	\[\mathrm{dist}(z,\mathrm{sp}(\mathrm{Op}_\hbar^{\mathrm{W}}(p^{\mathrm{eff}}_\hbar)))\|\psi\|\leq C\left(1+\mathrm{dist}(z,\mathrm{sp}(\mathrm{Op}_\hbar^{\mathrm{W}}(p^{\mathrm{eff}}_\hbar)))\right)(\|(\mathscr{N}_{a,\hbar}-z)\psi\|+\hbar^3\|\langle\tau\rangle^{N_3}\psi\|)\,,\]
	and
	\[\mathrm{dist}(z,\mathrm{sp}(\mathscr{N}_{a,\hbar}))\|\psi\|\leq C\|(\mathrm{Op}_\hbar^{\mathrm{W}}(p^{\mathrm{eff}}_\hbar)-z)\psi\|+C\hbar^{3}\|\psi\|\,.\]
	In addition, we have
	\[|\lambda_n(\mathscr{N}_{a,\hbar})-\lambda^{\mathrm{eff}}_n(\alpha,\hbar)|\leq C\hbar^3\,,\quad \alpha=a-\theta\hbar\,,\]
	uniformly with respect to $a\in (a_0-\eta,a_0+\eta)$.	
\end{proposition}

\subsection{On the effective operator}\label{sec.peff}
Let us now consider an $a$ in the form
\[a=a_0+\hbar a_1+\hbar^2 a_2\,,\]
or equivalently
\[\alpha=\alpha_0+\hbar\alpha_1+\hbar^2\alpha_2\,,\quad \alpha_0=a_0\,,\quad\alpha_1=a_1-\theta\,,\quad\alpha_2=a_2\,.\]
For such a choice, let us perform the spectral analysis of $\mathrm{Op}_\hbar^{\mathrm{W}}(p^{\mathrm{eff}}_\hbar)$. Note that, thanks to Lemma \ref{lem.microloc}, we can remove the frequency cutoff $\chi_0(\xi)$, up to a remainder of order $\mathscr{O}(\hbar^\infty)$. We can also replace $\chi_\hbar$ by $1$ thanks to the exponential decay of $u_{\alpha,\xi}$ modulo $\mathscr{O}(\hbar^\infty)$.

Then, we expand the symbol
\[
p^{\mathrm{eff}}_{\hbar}(s,\xi)=\nu(a_0,\xi)+\hbar p_1^{\mathrm{eff}}(s,\xi)+\hbar^2 p^{\mathrm{eff}}_{2}(s,\xi)+\mathscr{O}(\hbar^3)\,,
\]
\begin{equation}\label{eq.p1eff}
\begin{split}
p^{\mathrm{eff}}_1(s,\xi)=&\partial_\alpha\nu(a_0,\xi)(a_1-\theta)+\langle  n_1u_{a_0,\xi},u_{a_0,\xi}\rangle\,,\\
p^{\mathrm{eff}}_2(s,\xi)=&\partial_\alpha\nu(a_0,\xi)a_2+\frac{\partial_\alpha^2\nu(a_0,\xi)}{2}(a_1-\theta)^2+(a_1-\theta)\partial_{\alpha}\langle  n_1 u_{\alpha,\xi},u_{\alpha,\xi}\rangle(a_0)\\
&-\langle ( n_1(\mathscr{M}_{a_0,\xi}-\nu(a_0,\xi_{a_0}))^{-1}\Pi^\perp  n_1-  n_2) u_{a_0,\xi},u_{a_0,\xi} \rangle\,.
\end{split}
\end{equation}
Lemma \ref{lem.microloc} invites us to write a Taylor expansion of $p^{\mathrm{eff}}_{\hbar}$ near $\xi_{a_0}$. We can write
\[\begin{split}
p^{\mathrm{eff}}_{\hbar}(s,\xi)=\nu(a_0,\xi_{a_0})+\hbar^2\partial_\alpha\nu(a_0,\xi_{a_0})a_2+\frac{\partial^2_\xi\nu(a_0,\xi_{a_0})}{2}(\xi-\xi_{a_0})^2\\
+\hbar(p_1^{\mathrm{eff}}(s,\xi_{a_0})+\partial_\xi p^{\mathrm{eff}}_1(s,\xi_{a_0})(\xi-\xi_{a_0}))
+\hbar^2 \tilde p^{\mathrm{eff}}_{2}(s,\xi_{a_0})+r_\hbar\,,
\end{split}\]
where $r_\hbar\in S(1)$ (\emph{i.e.}, $r_\hbar$ and its derivatives at any order are bounded) satisfies
\begin{equation}\label{eq.rhbar}
|r_{\hbar}|\leq C|\xi-\xi_{a_0}|^3+C\hbar|\xi-\xi_{a_0}|^2\,,\end{equation}
and
\begin{multline}\label{eq.tildep2eff}
\tilde p^{\mathrm{eff}}_{2}(s,\xi_{a_0})=\frac{\partial_\alpha^2\nu(a_0,\xi_{a_0})}{2}(a_1-\theta)^2+(a_1-\theta)\partial_{\alpha}\langle  n_1 u_{\alpha,\xi_{a_0}},u_{\alpha,\xi_{a_0}}\rangle(a_0)\\
-\langle ( n_1(\mathscr{M}_{a_0,\xi_{a_0}}-\nu(a_0,\xi_{a_0}))^{-1}\Pi^\perp  n_1-  n_2) u_{a_0,\xi_{a_0}},u_{a_0,\xi_{a_0}} \rangle
\end{multline}

Let us consider $p_1^{\mathrm{eff}}(s,\xi_{a_0})$. By using \eqref{eq.nj} (with $\chi_\hbar$ replaced by $1$), Proposition \ref{prop.C3}, Lemmata \ref{lem.nu'} and \ref{lem.nu'a}, we get
\[\begin{split}
p_1^{\mathrm{eff}}(s,\xi_{a_0})&=\partial_\alpha\nu(a_0,\xi_{a_0})(a_1-\theta)+\langle  n_1u_{a_0,\xi_{a_0}},u_{a_0,\xi_{a_0}}\rangle\\
&=\partial_\alpha\nu(a_0,\xi_{a_0})(a_1-\theta)+2\theta\langle(\xi-\tau)u_{a_0,\xi_{a_0}},u_{a_0,\xi_{a_0}}\rangle\\
&=a_1u^2_{a_0,\xi_{a_0}}(0) \,.
\end{split}\]
Therefore,
\begin{equation}\label{eq.peff}\begin{split}
p^{\mathrm{eff}}_{\hbar}(s,\xi)=\nu(a_0,\xi_{a_0})+ a_1u^2_{a_0,\xi_{a_0}}(0)\hbar+u^2_{a_0,\xi_{a_0}}(0)a_2 \hbar^2+\widehat {p}^{\mathrm{eff}}_{\hbar}(s,\xi)+r_\hbar\,,
\end{split}\end{equation}
where
\begin{equation}\label{eq.hatpeff}
\widehat {p}^{\mathrm{eff}}_{\hbar}(s,\xi)=\frac{\partial^2_\xi\nu(a_0,\xi_{a_0})}{2}(\xi-\xi_{a_0})^2
+\hbar\partial_\xi p^{\mathrm{eff}}_1(s,\xi_{a_0})(\xi-\xi_{a_0})
+\hbar^2 \tilde p^{\mathrm{eff}}_{2}(s,\xi_{a_0})\,.
\end{equation}
In particular, we have found our ultimate effective (differential) operator
\begin{multline*}
\widehat{\mathscr{P}}^{\mathrm{eff}}_\hbar=\mathrm{Op}^{\mathrm{W}}_\hbar(\widehat {p}^{\mathrm{eff}}_{\hbar})=\frac{\partial^2_\xi\nu(a_0,\xi_{a_0})}{2}(\hbar D_s-\xi_{a_0})^2\\
+\frac{\hbar}{2}\left(\partial_\xi p^{\mathrm{eff}}_1(s,\xi_{a_0})(\hbar D_s-\xi_{a_0})+(\hbar D_s-\xi_{a_0})\partial_\xi p^{\mathrm{eff}}_1(s,\xi_{a_0})\right)
+\hbar^2 \tilde p^{\mathrm{eff}}_{2}(s,\xi_{a_0})\,.
\end{multline*}
We can rewrite it in the following form
\begin{multline}\label{eq.Peff}
\widehat{\mathscr{P}}^{\mathrm{eff}}_\hbar=\frac{\partial^2_\xi\nu(a_0,\xi_{a_0})}{2}\left(\hbar D_s-\xi_{a_0}+\hbar\frac{\partial_\xi p^{\mathrm{eff}}_1(s,\xi_{a_0})}{\partial^2_\xi\nu(a_0,\xi_{a_0})}\right)^2\\
+\hbar^2\left(\tilde p^{\mathrm{eff}}_{2}(s,\xi_{a_0})-\frac{\left(\partial_\xi p^{\mathrm{eff}}_1(s,\xi_{a_0})\right)^2}{2\partial^2_\xi\nu(a_0,\xi_{a_0})}\right)\,.
\end{multline}

\begin{lemma}
	Let $(\widehat{\lambda}_n^{\mathrm{eff}}(a,\hbar))_{n\geq 1}$ be the non-decreasing sequence of the eigenvalues of $\widehat{\mathscr{P}}^{\mathrm{eff}}_\hbar$. We have, for all $n\geq 1$,
	\[\widehat{\lambda}_n^{\mathrm{eff}}(a,\hbar)=\mathscr{O}(\hbar^2)\,.\]
\end{lemma}

\begin{proposition}\label{prop.lambdaneff+-}
	There exist $C, \hbar_0>0$, $\eta\in\left(0,\frac 12\right)$ such that the following holds. Consider the operators $\widehat{\mathscr{P}}^{\mathrm{eff},\pm}_\hbar$ defined by
	\[\widehat{\mathscr{P}}^{\mathrm{eff},\pm}_\hbar=\mathrm{Op}^{\mathrm{W}}_\hbar \widehat{p}^{\mathrm{eff},\pm}_\hbar\,,\]
	where
	\begin{equation*}
	\widehat{p}^{\mathrm{eff},\pm}_\hbar=(1\pm C\hbar^{\frac 12-\eta})\frac{\partial^2_\xi\nu(a_0,\xi_{a_0})}{2}(\xi-\xi_{a_0})^2
	+\hbar\partial_\xi p^{\mathrm{eff}}_1(s,\xi_{a_0})(\xi-\xi_{a_0})
	+\hbar^2 \tilde p^{\mathrm{eff}}_{2}(s,\xi_{a_0})\pm C\hbar^{2+\eta}\,.
	\end{equation*}	
	For all $\hbar\in(0,\hbar_0)$, we have
	\[\nu(a_0,\xi_{a_0})+ a_1u^2_{a_0,\xi_{a_0}}(0)\hbar+a_2u^2_{a_0,\xi_{a_0}}(0) \hbar^2+\lambda_n^{\mathrm{eff},-}(a,\hbar)\leq\lambda_n^{\mathrm{eff}}(a,\hbar)\,,\]
	and
	\[\lambda_n^{\mathrm{eff}}(a,\hbar)\leq \nu(a_0,\xi_{a_0})+ a_1u^2_{a_0,\xi_{a_0}}(0)\hbar+a_2u^2_{a_0,\xi_{a_0}}(0) \hbar^2+ \lambda_n^{\mathrm{eff},+}(a,\hbar)\,.\]
	
\end{proposition}
\begin{proof}
Let us recall that $p^{\mathrm{eff}}_{\hbar}$ is given in \eqref{eq.peff} and that the remainder $r_\hbar$ is defined in \eqref{eq.rhbar}. We have

\[p^{\mathrm{eff}}_\hbar-\left(\nu(a_0,\xi_{a_0})+ a_1u^2_{a_0,\xi_{a_0}}(0)\hbar+a_2u^2_{a_0,\xi_{a_0}}(0) \hbar^2\right)=\widehat {p}^{\mathrm{eff}}_{\hbar}(s,\xi)+r_\hbar\,.\]
We have
\[\langle\mathrm{Op}^{\mathrm{W}}_\hbar (\widehat {p}^{\mathrm{eff}}_{\hbar}(s,\xi)+r_\hbar)\psi,\psi\rangle=\langle\mathrm{Op}^{\mathrm{W}}_\hbar \widehat {p}^{\mathrm{eff}}_{\hbar}(s,\xi)\psi,\psi\rangle+\langle\mathrm{Op}^{\mathrm{W}}_\hbar r_\hbar\psi,\psi\rangle\,.\]
Let us discuss the upper bound, the lower bound following from similar arguments.
We consider
\[\mathscr{E}_N(\hbar)=\underset{1\leq j\leq N}{\mathrm{span}} \psi_{j,\hbar}\,,\]
where $(\psi_{j,\hbar})$ is an orthonormal family of eigenfunctions of $\mathrm{Op}^{\mathrm{W}}_\hbar \widehat {p}^{\mathrm{eff}}_{\hbar}$ associated with the eigenvalues $(\widehat{\lambda}^{\mathrm{eff}}_j(a,\hbar))$. We can prove that, for all $\psi\in\mathscr{E}_N(\hbar)$, we have
\begin{equation}\label{eq.microloc}
\mathrm{Op}^{\mathrm{W}}_\hbar [\chi(\hbar^{-\frac{1}{2}+\eta}(\xi-\xi_{a_0}))]\psi=\psi+\mathscr{O}(\hbar^\infty)\|\psi\|\,,
\end{equation}
where $\chi$ is a smooth cutoff function equal to $1$ near $0$ with compact support, and $\eta\in\left(0,\frac12\right)$. Then, by using \eqref{eq.rhbar}, \eqref{eq.microloc}, and classical pseudo-differential estimates,	
\[|\langle\mathrm{Op}^{\mathrm{W}}_\hbar r_\hbar\psi,\psi\rangle|\leq C\hbar^{\frac12-\eta}\|(\hbar D_s-\xi_{a_0})\psi\|^2+C\hbar \|(\hbar D_s-\xi_{a_0})\psi\|^2+C\hbar^{2+\eta}\|\psi\|^2\,,\]
for some $\eta\in(0,\frac12)$. Indeed, we can write
\[r_\hbar(s,\xi)=(\xi-\xi_{a_0})^2\tilde r_\hbar(s,\xi)+\check r_\hbar\,,\]
where $\tilde r_\hbar$ and $\check r_\hbar$ belong to $S(1)$, with $|\tilde r_\hbar|\leq C|\xi-\xi_{a_0}|$, and $\check r_\hbar$ equals $0$ near $\xi_{a_0}$. Due to \eqref{eq.microloc} and support considerations, we get
\[\mathrm{Op}^{\mathrm{W}}_\hbar(\check r_\hbar)\psi=\mathscr{O}(\hbar^\infty)\|\psi\|\,.\]
Then, a computation gives
\[\mathrm{Op}^{\mathrm W}_\hbar(\xi-\xi_{a_0})\mathrm{Op}^{\mathrm W}_\hbar(\tilde r_\hbar)\mathrm{Op}^{\mathrm W}_\hbar(\xi-\xi_{a_0})=\mathrm{Op}^{\mathrm W}_\hbar\left((\xi-\xi_{a_0})^2\tilde r_\hbar+\frac{\hbar^2}{4}\partial^2_s\tilde r_\hbar\right)\,,\]
from which we deduce that
\[\langle\mathrm{Op}^{\mathrm{W}}_\hbar(\tilde r_\hbar(\xi-\xi_{a_0})^2)\psi,\psi\rangle\leq \langle\mathrm{Op}^{\mathrm{W}}_\hbar(\tilde r_\hbar)(\hbar D_s-\xi_{a_0})\psi,(\hbar D_s-\xi_{a_0})\psi\rangle+C\hbar^2\|\mathrm{Op}^{\mathrm{W}}_\hbar(\partial^2_s\tilde r_\hbar)\psi\|\|\psi\| \,.\]
The last term can be controlled by the G\aa rding inequality (in the class $S^\delta(1)$, with $\delta=\frac12-\eta$) and \eqref{eq.microloc}:
\[\|\mathrm{Op}^{\mathrm{W}}_\hbar(\partial^2_s\tilde r_\hbar)\psi\|\leq C\hbar^{\frac 12-\eta}\|\psi\|\,.\]

We deduce that
\begin{multline*}
\langle\mathrm{Op}^{\mathrm{W}}_\hbar (\widehat {p}^{\mathrm{eff}}_{\hbar}(s,\xi)+r_\hbar)\psi,\psi\rangle\\
\leq \langle\mathrm{Op}^{\mathrm{W}}_\hbar (\widehat {p}^{\mathrm{eff}}_{\hbar}(s,\xi))\psi,\psi\rangle+C\hbar^{\frac12-\eta}\|(\hbar D_s-\xi_{a_0})\psi\|^2+C\hbar \|(\hbar D_s-\xi_{a_0})\psi\|^2+C\hbar^{2+\eta}\|\psi\|^2\,,
\end{multline*}
which gives the desired upper bound after recalling \eqref{eq.hatpeff} and using the min-max principle.

The lower bound follows in the same way (by using the eigenfunctions of $\mathrm{Op}^{\mathrm{W}}_\hbar p^{\mathrm{eff}}_\hbar$).
\end{proof}

\begin{proposition}
	For all $n\geq 1$, we have
	\[\lambda_n^{\mathrm{eff},\pm}(a,\hbar)=\widehat{\lambda}_n^{\mathrm{eff}}(a,\hbar)+o(\hbar^2)\,.\]
\end{proposition}

\begin{proof}
This comes from the fact that, if $\psi$ is an eigenfunction of $\mathrm{Op}^{\mathrm{W}}_\hbar(\widehat{p}^{\mathrm{eff},\pm}_\hbar)$, resp. $\mathrm{Op}^{\mathrm{W}}_\hbar(\widehat{p}^{\mathrm{eff}}_\hbar)$, associated with $\lambda_n^{\mathrm{eff},\pm}(a,\hbar)=\mathscr{O}(\hbar^2)$, resp. $\widehat{\lambda}_n^{\mathrm{eff}}(a,\hbar)=\mathscr{O}(\hbar^2)$, we have $\|(\hbar D_s-\xi_{a_0})\psi\|^2\leq C\hbar^2\|\psi\|^2$.
	\end{proof}
With Corollary \ref{cor.approxlambdan} and Propositions \ref{prop.distsp} and \ref{prop.lambdaneff+-}, we deduce the following important corollary.
\begin{corollary}\label{cor.ultimate}
	For all $n\geq 1$, we have
	\[h^{-1}\lambda_n(a,h)= \nu(a_0,\xi_{a_0})+ a_1u^2_{a_0,\xi_{a_0}}(0)\hbar+a_2u^2_{a_0,\xi_{a_0}}(0) \hbar^2+\widehat{\lambda}_n^{\mathrm{eff}}(a,\hbar)+o(\hbar^2)\,.\]
\end{corollary}

\subsection{Proof of Theorem \ref{thm.main2}}\label{sec.end}

\subsubsection{A choice of $a$}\label{sec.choicea}
From Corollary \ref{cor.ultimate} and \eqref{eq.ell'a}, we deduce that
\[a|a-e_n(h)|\leq |\nu(a_0,\xi_{a_0})+ a_1u^2_{a_0,\xi_{a_0}}(0)\hbar+a_2u^2_{a_0,\xi_{a_0}}(0) \hbar^2+\widehat{\lambda}_n^{\mathrm{eff}}(a,\hbar)-a^2|+o(\hbar^2)\,.\]
It remains to make a clever choice of $a_1$ and $a_2$. We have
\[a=a_0+\hbar a_1+\hbar^2a_2\,.\]
We want that
\[\nu(a_0,\xi_{a_0})+ a_1u^2_{a_0,\xi_{a_0}}(0)\hbar+a_2u^2_{a_0,\xi_{a_0}}(0) \hbar^2+\widehat{\lambda}_n^{\mathrm{eff}}(a,\hbar)-(a_0+\hbar a_1+\hbar^2 a_2)^2\]
to be $o(\hbar^2)$. By the choice of $a_0$, we have
\[\nu(a_0,\xi_{a_0})=\nu(a_0)=a_0^2\,.\]
Then, we would like to have
\[a_1u^2_{a_0,\xi_{a_0}}(0) =2a_0 a_1\,.\]
By using Lemma \ref{lem.nu'}, we have
\[u^2_{a_0,\xi_{a_0}}(0)=2\int_{\R_0}(\xi_{a_0}-t)u^2_{a_0,\xi_{a_0}}(t)\dd t\,,\]
and by Proposition \ref{prop.C3} (with our choice of $a_0$), we have $\xi_{a_0}=a_0$ so that
\[u^2_{a_0,\xi_{a_0}}(0)-2a_0=-2\int_{\R_0}tu^2_{a_0,\xi_{a_0}}(t)\dd t<0\,.\]
Thus, we must choose $a_1=0$.

Now, we must choose $a_2$ so that
\[2a_0 a_2\hbar^2=a_2u^2_{a_0,\xi_{a_0}}(0) \hbar^2+\widehat{\lambda}_n^{\mathrm{eff}}(a,\hbar)+o(\hbar^2)\,.\]
Note that $\widehat{\lambda}_n^{\mathrm{eff}}(a,\hbar)$ only depends on $a_0$. This leads to choose
\[a_2=\hbar^{-2}(2a_0-u^2_{a_0,\xi_{a_0}}(0))^{-1}\widehat{\lambda}_n^{\mathrm{eff}}(a_0,\hbar)\,.\]
With this choice, we get
\[e_n(h)=a_0+\hbar^2 a_2+o(\hbar^2)\,.\]
It remains to describe $\widehat{\mathscr{P}}^{\mathrm{eff}}_\hbar$ in order to get Theorem \ref{thm.main2}.
\subsubsection{End of the proof}

Let us now analyze the dependence on $\theta$ of $\widehat{\mathscr{P}}^{\mathrm{eff}}_\hbar$ defined in \eqref{eq.Peff}. Let us look at the first term. We have
\[\begin{split}
\partial_\xi p_1^{\mathrm{eff}}(s,\xi_{a_0})&=-\theta\partial_{\alpha}\partial_\xi\nu(a_0,\xi_{a_0})+\partial_{\xi}\langle n_1 u_{a_0,\xi}, u_{a_0,\xi}\rangle\\
&=\theta\left(-\partial_{\alpha}\partial_\xi\nu(a_0,\xi_{a_0})+2 \partial_{\xi}\int_0^{+\infty}(\xi-\tau)u^2_{a_0,\xi}(\tau)\dd\tau\right)+\kappa\partial_\xi\langle\mathscr{C}_\xi u_{a_0,\xi}, u_{a_0,\xi}\rangle\\
&=\theta\partial^2_{\xi}\nu(a_0,\xi_{a_0})+\kappa\partial_\xi\langle\mathscr{C}_\xi u_{a_0,\xi}, u_{a_0,\xi}\rangle\,.
\end{split}\]
where we used $n_1=\kappa\mathscr{C}_\xi+2\theta(\xi-\tau)$ and Lemmata \ref{lem.nu'} and \ref{lem.nu'a}.

This shows that
\begin{equation}\label{eq.Peff1}
\begin{split}\left(\hbar D_s-\xi_{a_0}+\hbar\frac{\partial_\xi p^{\mathrm{eff}}_1(s,\xi_{a_0})}{\partial^2_\xi\nu(a_0,\xi_{a_0})}\right)^2&=\left(\hbar D_s-\xi_{a_0}+\hbar\theta+\hbar\kappa\frac{\partial_\xi\langle\mathscr{C}_\xi u_{a_0,\xi}, u_{a_0,\xi}\rangle}{\partial^2_\xi\nu(a_0,\xi_{a_0})}\right)^2\\
&=\left(\hbar D_s-\xi_{a_0}+\hbar\theta-\hbar\frac{\kappa}{2}\right)^2=\left(\hbar D_s-a_0+\hbar\theta-\hbar\frac{\kappa}{2}\right)^2\,,
\end{split}
\end{equation}
where we used Lemma \ref{lem.magie} and the fact that $\xi_{a_0}=a_0$. 

\begin{remark}
We recall that the expression of $\theta$ is given in \eqref{eq.theta}. By the Gauss-Bonnet formula,
\[\underline{\kappa}=\frac{1}{|\partial\Omega|}\int_0^{|\partial\Omega|}\kappa(s)\dd s=\frac{2\pi}{|\partial\Omega|}\,.\]
Thus, in \eqref{eq.Peff1}, due to the gauge invariance, $\frac{\kappa}{2}$ can be replaced by $\frac{\underline{\kappa}}{2}=\frac{\pi}{|\partial\Omega|}$. This, with Remark \ref{rem.gaugeth}, is consistent with the expression of $\mathfrak{t}_h$ in \eqref{eq.Qeff}.
\end{remark}
Let us now look at the second term in \eqref{eq.Peff}, \emph{i.e.},
\begin{equation}\label{eq.potential}\tilde p^{\mathrm{eff}}_{2}(s,\xi_{a_0})-\frac{\left(\partial_\xi p^{\mathrm{eff}}_1(s,\xi_{a_0})\right)^2}{2\partial^2_\xi\nu(a_0,\xi_{a_0})}\,,
\end{equation}
where we recall \eqref{eq.p1eff} and \eqref{eq.tildep2eff}.

Let us explain why this term does not depend on $\theta$. It depends a priori on $\theta$ quadratically. Let us gather the terms depending on $\theta$ linearly (they have all $\kappa$ as common factor):
\begin{multline}\label{eq.termtheta}
L=-\partial_{\alpha}\langle \mathscr{C}_{\xi_{a_0}} u_{\alpha, \xi_{a_0}},u_{\alpha,\xi_{a_0}}\rangle-4\langle(\xi-\tau)(\mathscr{M}-\nu(a_0))^{-1}\Pi^{\perp}\mathscr{C}_{\xi_{a_0}} u_{a_0,\xi_{a_0}}, u_{a_0,\xi_{a_0}}\rangle\\
-\partial_\xi\langle\mathscr{C}_\xi u_{a_0,\xi}, u_{a_0,\xi}\rangle+\langle (2\tau^2-2+4\tau(\xi-\tau))u_{a_0,\xi_{a_0}},u_{a_0,\xi_{a_0}}\rangle\,.
\end{multline}
Note that
\[\partial_\xi\mathscr{C}_\xi=2\tau^2-2+4\tau(\xi-\tau)\,,\]
so that \eqref{eq.termtheta} becomes
\begin{multline*}
L=-2\langle \mathscr{C}_\xi u_{\alpha, \xi_{a_0}},\partial_\alpha u_{\alpha,\xi_{a_0}}\rangle-4\langle(\xi-\tau)(\mathscr{M}-\nu(a_0))^{-1}\Pi^{\perp}\mathscr{C}_\xi u_{a_0,\xi_{a_0}}, u_{a_0,\xi_{a_0}}\rangle-2\langle\mathscr{C}_\xi u_{a_0,\xi_{a_0}}, v_{a_0,\xi_{a_0}}\rangle\,.
\end{multline*}
We recall that $\langle\mathscr{C}_{\xi_{a_0}}u_{a_0,\xi_{a_0}}, u_{a_0,\xi_{a_0}}\rangle=0$ so that $\Pi^{\perp}\mathscr{C}_{\xi_{a_0}} u_{a_0,\xi_{a_0}}=\mathscr{C}_{\xi_{a_0}} u_{a_0,\xi_{a_0}}$, and we use Lemma \ref{lem.Cuv} to get
\[\begin{split}
L&=-2\langle \mathscr{C}_{\xi_{a_0}} u_{\alpha, \xi_{a_0}},\partial_\alpha u_{\alpha,\xi_{a_0}}\rangle+4\langle(\xi-\tau)k_{a_0,\xi_{a_0}}, u_{a_0,\xi_{a_0}}\rangle-2\langle\mathscr{C}_\xi u_{a_0,\xi_{a_0}}, v_{a_0,\xi_{a_0}}\rangle\\
&=-2\langle \mathscr{C}_{\xi_{a_0}} u_{\alpha, \xi_{a_0}},\partial_\alpha u_{\alpha,\xi_{a_0}}\rangle+2k_{a_0,\xi_{a_0}}(0)u_{a_0,\xi_{a_0}}(0)\\
&=2\left[\langle (\mathscr{M}_{a_0,\xi_{a_0}}-\nu(a_0))k_{\alpha, \xi_{a_0}},\partial_\alpha u_{\alpha,\xi_{a_0}}\rangle+k_{a_0,\xi_{a_0}}(0)u_{a_0,\xi_{a_0}}(0)\right]\\
&=2\left[k'_{a_0, \xi_{a_0}}(0)\partial_\alpha u_{\alpha,\xi_{a_0}}(0)-k_{a_0, \xi_{a_0}}(0)\partial_\alpha\partial_\tau u_{\alpha,\xi_{a_0}}(0)+k_{a_0,\xi_{a_0}}(0)u_{a_0,\xi_{a_0}}(0)\right]\,,
\end{split}
\]
where we used Lemma \ref{lem.ipp}, and the fact that $(\mathscr{M}_{a_0,\xi_{a_0}}-\nu(a_0))\partial_\alpha u_{\alpha,\xi_{a_0}}=0$. Let us now recall that $\xi_{a_0}=a_0$, by Proposition \ref{prop.C3}. In addition, since $\partial_\tau u_{\alpha,\xi}(0)=(\alpha-\xi)u_{\alpha,\xi}(0)$, we get
\[\partial_\alpha\partial_\tau u_{\alpha,\xi_{a_0}}(0)=u_{a_0,\xi_{a_0}}(0)\,.\]
This implies that
\[L=0\,.\]
Let us now gather the terms depending on $\theta$ quadratically:
\begin{multline*}
C=\frac{\partial^2_\alpha\nu(a_0,\xi_{a_0})}{2}-2\partial_{\alpha}\langle(\xi-\tau)u_{\alpha,\xi_{a_0}},u_{\alpha,\xi_{a_0}}\rangle\\
+2\langle (\xi-\tau)(\mathscr{M}_{a_0,\xi_{a_0}}-\nu(a_0))^{-1}\Pi^{\perp}(-2(\xi-\tau)u_{\alpha,\xi_{a_0}}),u_{\alpha,\xi_{a_0}}\rangle+1-\frac{\partial^2_\xi\nu(a_0,\xi_{a_0})}{2}\,.
\end{multline*}
By Lemma \ref{lem.v+g}, we get
\begin{equation*}
\begin{split}
&(\mathscr{M}_{a_0,\xi_{a_0}}-\nu(a_0))^{-1}\Pi^{\perp}(-2(\xi-\tau)u_{\alpha,\xi_{a_0}})\\
&=(\mathscr{M}_{a_0,\xi_{a_0}}-\nu(a_0))^{-1}\Pi^{\perp}(-2(\xi-\tau)u_{\alpha,\xi_{a_0}}+u^2_{\xi_{a_0}}(0)u_{a_0,\xi_{a_0}})\\
&=v_{a_0,\xi_{a_0}}+g_{a_0,\xi_{a_0}}\,.
\end{split}
\end{equation*}
Thus, with Lemmata \ref{lem.wxi} and \ref{lem.gxi},
\begin{equation*}
\begin{split}
C=&\frac{\partial^2_\alpha\nu(a_0,\xi_{a_0})}{2}-2\partial_{\alpha}\langle(\xi-\tau)u_{\alpha,\xi_{a_0}},u_{\alpha,\xi_{a_0}}\rangle\\
&+2\langle (\xi-\tau)(v_{a_0,\xi_{a_0}}+g_{a_0,\xi_{a_0}}),u_{a_0,\xi_{a_0}}\rangle+1-\frac{\partial^2_\xi\nu(a_0,\xi_{a_0})}{2}\\
=&\frac{\partial^2_\alpha\nu(a_0,\xi_{a_0})}{2}-2\partial_{\alpha}\langle(\xi-\tau)u_{\alpha,\xi_{a_0}},u_{\alpha,\xi_{a_0}}\rangle+2\langle(\xi-\tau) g_{a_0,\xi_{a_0}},u_{a_0,\xi_{a_0}}\rangle\\
&+u_{a_0,\xi_{a_0}}(0)v_{a_0,\xi_{a_0}}(0)\\
=&\frac{\partial^2_\alpha\nu(a_0,\xi_{a_0})}{2}-2\partial_{\alpha}\langle(\xi-\tau)u_{\alpha,\xi_{a_0}},u_{\alpha,\xi_{a_0}}\rangle+2u_{a_0,\xi_{a_0}}(0)v_{a_0,\xi_{a_0}}(0)+u_{a_0,\xi_{a_0}}(0)g_{a_0,\xi_{a_0}}(0)\\
=&\frac{\partial^2_\alpha\nu(a_0,\xi_{a_0})}{2}-4\langle(\xi-\tau)u_{\alpha,\xi_{a_0}},\partial_\alpha u_{\alpha,\xi_{a_0}}\rangle+2u_{a_0,\xi_{a_0}}(0)v_{a_0,\xi_{a_0}}(0)+u_{a_0,\xi_{a_0}}(0)g_{a_0,\xi_{a_0}}(0)
\end{split}
\end{equation*}
Using Lemmata \ref{lem.xi-t} and \ref{lem.ug}, we get
\[C=0\,.\]
Therefore, in \eqref{eq.potential}, we can replace $\theta$ by $0$. Using again Lemma \ref{lem.magie}, we see that it remains to consider $\tilde p^{\mathrm{eff}}_{2}(s,\xi_{a_0})$ defined in \eqref{eq.tildep2eff} and given by
\begin{equation}
\tilde p^{\mathrm{eff}}_{2}(s,\xi_{a_0})=-\langle ( n_1(\mathscr{M}_{a_0,\xi_{a_0}}-\nu(a_0,\xi_{a_0}))^{-1}\Pi^\perp  n_1-  n_2) u_{a_0,\xi_{a_0}},u_{a_0,\xi_{a_0}} \rangle\,,
\end{equation}
where $n_1$ and $n_2$ can be replaced by\footnote{See Lemma \ref{lem.Cukexpl2} and \eqref{eq.nj}.}
\[n_1=\kappa\mathscr{C}_\xi\,,\quad n_2=\kappa^2\mathscr{C}_{\xi,2}\,.\]
Hence,
\[
\tilde p^{\mathrm{eff}}_{2}(s,\xi_{a_0})=\kappa^2\left(\langle  \mathscr{C}_{\xi_{a_0}}u_{a_0,\xi_{a_0}},k_{a_0,\xi_{a_0}}\rangle+\langle\mathscr{C}_{\xi,2}u_{a_0,\xi_{a_0}},u_{a_0,\xi_{a_0}}\rangle\right)\,.\]
Lemmata \ref{lem.final} and \ref{lem.magie} show that
\[\tilde p^{\mathrm{eff}}_{2}(s,\xi_{a_0})-\frac{\left(\partial_\xi p^{\mathrm{eff}}_1(s,\xi_{a_0})\right)^2}{2\partial^2_\xi\nu(a_0,\xi_{a_0})}=\frac{\kappa^2}{12}\partial^2_\xi\nu(a_0,\xi_{a_0})-\kappa^2\frac{\partial^2_\xi\nu(a_0,\xi_{a_0})}{8}=-\frac{\kappa^2}{12}\frac{\partial_\xi^2\nu(a_0,\xi_{a_0})}{2}\,.\]
Recalling \eqref{eq.Peff}, \eqref{eq.Peff1}, and the discussion in Section \ref{sec.choicea}, this ends the proof of Theorem \ref{thm.main2}.

\appendix

\section{The results under various local boundary conditions}\label{app.A}
For $\eta\in \mathbb{R}\,,$ and $\mathbf{n}$ is a unit vector, we define the boundary matrix 
\[
\mathcal{T}_{\eta,\mathbf{n}}=-i\sigma_3(\sigma\cdot\mathbf{n})\cos(\eta) + \sigma_3 \sin(\eta)\,.
\] 
$\mathcal{T}_{\eta,\mathbf{n}}$ is an unitary and Hermitian matrix so that its spectrum is $\{\pm 1\}$.
For any regular function $\eta\colon \partial \Omega\to\mathbb{R}$, we introduce the local boundary condition 
\[
\mathcal{T}_{\eta(s),\mathbf{n}(s)} \varphi(s)=\varphi(s)\,, \quad s\in \partial \Omega\,,
\]  
where $\mathbf{n}\colon \partial\Omega\to S^1$ is the outward pointing normal and $\varphi\colon \partial\Omega \to \CC^2$. The associated magnetic Dirac operator $(\mathscr{D}_{h,\mathbf{A},\eta}, \mathsf{Dom}(h,\mathscr{D}_{h,\mathbf{A},\eta}))$  acts as  $\DMh$ on
\[
\mathsf{Dom}(\mathscr{D}_{h,\mathbf{A},\eta}) = \left\{
\varphi\in  H^1(\Omega)^2\,,\quad \mathcal{T}_{\eta,\mathbf{n}} \varphi=\varphi\text{ on }\partial \Omega
\right\}\,.
\]
The case $\eta \equiv 0$ corresponds to the infinite-mass boundary condition.
Note that 
\[
\mathcal{T}_{\eta,\mathbf{n}} = \begin{pmatrix} \sin(\eta)&-i\overline{\mathbf{n}}\cos(\eta)\\i\mathbf{n}\cos(\eta)&-\sin(\eta)\end{pmatrix}\,,
\]
so that
the boundary condition reads
\[u_2=i\mathbf{n}\frac{\cos(\eta)}{1+\sin(\eta)}u_1\,,\]
where $\varphi = (u_1,u_2)^T$.
\begin{hyp}\label{ref.boundReg}
	We have $\eta\in \mathcal{C}^1\left(\partial \Omega, \left(-\frac{\pi}{2},\frac{3\pi}{2}\right)\setminus\{\frac\pi2\}\right)$.
\end{hyp}
In \cite{MR3626307}, the authors proved that  under Assumptions \ref{hyp.reg} and \ref{ref.boundReg}, $\mathscr{D}_{h,\mathbf{A},\eta}$ is self-adjoint.
We define 
\[
\gamma\colon s\in \partial\Omega\mapsto\frac{\cos(\eta(s))}{1+\sin(\eta(s))}=\tan\left(\frac{\pi}{4}-\frac{\eta(s)}{2}\right)\in \mathbb{R}_+\,.
\]
Since $\partial \Omega$ is compact, we get that 
\begin{equation}\label{eq.boundboundarycondition}
0<\inf_{\partial \Omega} \gamma\leq \gamma(s)\leq \sup_{\partial \Omega} \gamma<+\infty\,.
\end{equation}
\begin{notation}
	Let
	\[
	\|u\|_{\partial \Omega,\gamma}^2 = \int_{\partial \Omega }|u^2|\,\gamma \,\dd s\,,
	\]
	where $u\in L^2(\partial \Omega)$. By \eqref{eq.boundboundarycondition}, this norm is equivalent with the one introduced in Notation \ref{not.BH}.
\end{notation}
It is straightforward to see that the proofs of the min-max characterization and of Theorem \ref{thm.main} are exactly the same up to the replacement of the norm on the boundary. In particular, the constants in the asymptotic analysis are defined with respect to the 
corresponding  weighted Hardy norm on the boundary.

Theorem \ref{thm.main'} has also its counterpart in this context. Here, the proof has to be slightly adapted by  Taylor approximating $\gamma$ around each point of the boundary. We choose to present our proof for the infinite-mass condition only in order not to burden the reader with complicated notations that do not give more insight on the problem. More precisely, we get : 
\begin{theorem}\label{thm.mainbound}
	Under Assumptions \ref{hyp.reg}, \ref{hyp.posB}, and  \ref{ref.boundReg}:
	\begin{enumerate}[\rm (i)]
		\item Under the further assumption \ref{hyp.regphi} we have, 
		for all $k\geq 1$,
		\[\lambda^+_k(h)= \left(\frac{\mathrm{dist}_{\mathcal{H}}\left((z-z_{\min})^{k-1},\mathscr{H}^2_k(\Omega)\right)}{\mathrm{dist}_{\mathcal{B}}\left(z^{k-1},\mathcal{P}_{k-2}\right)}\right)^2 h^{1-k} e^{2\phi_{\min}/h}(1+o_{h\to 0}(1))\,,\]
		\item
		\[
		\lambda^-_1(h) = h^{\frac 12}\min\left(\sqrt{2b_0}, c_{\gamma(x)}\sqrt{B(x)}\,; x\in\partial \Omega\right)+o_{h\to 0}(h^{\frac 12})\,,
		\]
		where for any $x\in \partial \Omega$, $c_{\gamma(x)}>0$ is the unique positive solution of the equation $\nu_{\gamma(x)}(c) = c^2$ with
		\begin{equation*}
		\nu_{\gamma(x)}(c)= \inf_{\underset{u\neq 0}{u\in H^1(\mathbb{R}^2_+)}}\frac{\int_{\mathbb{R}^2_+} |(-i\partial_{s}-\tau+i(-i\partial_{\tau})) u|^2\dd s\dd \tau
			+c\gamma(x)\int_{\mathbb{R}}|u(s,0) |^2\dd s}{\|u\|^2}\,.
		\end{equation*}
	\end{enumerate}
\end{theorem}

\begin{remark}
	Using Remark \ref{chargec}, we also cover the case $\cos(\eta(s))<0$ for all $s\in \partial \Omega$. 
\end{remark}

\section{Negative eigenvalues and variable magnetic fields}\label{app.C}
Let us assume that $B$ is smooth and positive. As we saw in Sections \ref{sec.6} and \ref{sec.7}, the asymptotic analysis of the negative eigenvalues is related to the one of a Schrödinger operator with a Robin-like boundary condition (see the quadratic form \eqref{eq:originalquadform}). We proved (see Proposition \ref{prop.Lambdah}) the following one-term asymptotic expansion
\[\lambda_1(a,h)=h\min\left(
2b_0, b'_{0}\nu(a(b'_{0})^{-1/2})
\right)+o(h)\,,
\]
where $b_0=\min_{x\in\overline{\Omega}} B(x)$ and $b'_0=\min_{x\in\partial\Omega} B(x)$.

It implies that
\[\lambda_1^-(h)=h^{\frac 12}\min(\sqrt{2b_0},a_0\sqrt{b'_0})+o(h^{\frac 12})\,.\]

\subsection{Case of boundary localization}
As for the case with constant magnetic field, when $a_0\sqrt{b'_0}<\sqrt{2b_0}$, we can prove that the first eigenfunctions of $\mathscr{L}_{a,h}$ are exponentially localized near the boundary when $a$ is close enough to $a_0\sqrt{b'_0}$. In this case, we have 
\[\min\left(
2b_0, b'_{0}\nu(a(b'_{0})^{-1/2})
\right)=b'_{0}\nu(a(b'_{0})^{-1/2})\,.\]
Let us now explain how our strategy can be adapted to describe the asymptotic behavior of $\lambda_n(a,h)$. The key point is still the study of $\mathscr{N}_{a,\hbar}$, see \eqref{eq.Nahbar} and to use its interpretation as a pseudo-differential operator, see \eqref{eq.Nahbar2}. The main difference with the case of constant magnetic field appears in the principal operator symbol. We can check that $n_0$ is replaced by
\[-\partial^2_\tau+(\xi-b(s)\tau)^2+b(s)\,,\]
where $s\mapsto b(s)$ is the restriction of the magnetic to the boundary, and where $n_0$ is equipped with the Robin condition
\[(\partial_\tau+\xi-\alpha)\psi=0\,,\]
where we recall that $\alpha=a-\hbar\theta$. We also recall that $\theta$ is defined in \eqref{eq.theta} (see also Section \ref{sec.changeofgauge} where $\gamma_0$ has to be replaced by $\frac{\int_\Omega B(x)\mathrm{d}x}{|\partial\Omega|}$). Implemeting the Grushin method, we see that the principal symbol of the effective operator is the first eigenvalue of $n_0$, denoted by $\mu(s,\xi)$. It can be explictly described thanks to $\nu$. Consider the rescaling $\tau=b(s)^{-\frac12}\tilde\tau$, we get
\[\mu(s,\xi)=b(s)\nu\left(\frac{\alpha}{b(s)^{\frac12}},\frac{\xi}{b(s)^{\frac12}}\right)\,.\]
In order to study the function $\mu$ (and its critical points), we will need some lemmas.

The first lemma comes from the concavity of $\alpha\mapsto \nu(\alpha)$.
\begin{lemma}\label{lem.concavitynu}
	For all $\alpha>0$, 
	\[\nu'(\alpha)\leq \frac{\nu(\alpha)}{\alpha}\,.\]	
\end{lemma}

\begin{lemma}\label{lem.f}
	The function $f : (0,+\infty)\ni b\mapsto b\nu\left(\frac{\alpha}{\sqrt{b}}\right)$ is increasing and its derivative is positive.
\end{lemma}
\begin{proof}
	We have seen in Section \ref{sec.C4} that $\nu$ is analytic and that $\nu'>0$. We have
	\[f'(b)=\nu\left(\frac{\alpha}{\sqrt{b}}\right)-\frac{\alpha}{2\sqrt{b}}\nu'\left(\frac{\alpha}{\sqrt{b}}\right)\geq \frac{\alpha}{2\sqrt{b}}\nu'\left(\frac{\alpha}{\sqrt{b}}\right)>0\,,\]	
	where we used Lemma \ref{lem.concavitynu}. 
\end{proof}
\begin{proposition}\label{prop.B3}
	For all $s\in\partial\Omega$ and $\xi\in\R$, we have
	\[b'_0\nu\left(\frac{\alpha}{\sqrt{b'_0}}\right)\leq b(s)\nu\left(\frac{\alpha}{\sqrt{b}}\right)\leq\mu(s,\xi)\,.\]
	In particular, $b'_0\nu\left(\frac{\alpha}{\sqrt{b'_0}}\right)$ is the minimal value of $\mu$.
	
	Moreover, if $b$ has a unique minimum at $s_0$, then $\mu$ has also a unique minimum at $(s,\xi)=\left(s_0,\sqrt{b'_0}\xi_{\frac{\alpha}{\sqrt{b'_0}}}\right)$. 
\end{proposition}
\begin{proof}
	The inequality follows from the fact that
	\[\nu(\alpha)=\min_{\xi\in\R}\nu(\alpha,\xi)\,,\]
	and from Lemma \ref{lem.f}.	Taking $s_0$ a minimum of $b$ and $\xi=\sqrt{b'_0}\xi_{\frac{\alpha}{\sqrt{b'_0}}}$ (see Proposition \ref{prop.dispersionCurveNu} for the definition of $\alpha\mapsto\xi_\alpha$), we get the minimal value. By using the strict monotonicity in Lemma \ref{lem.f}, we get the conclusion about the unique minimum.
\end{proof}
Let us now make the following generic assumption.
\begin{assumption}
	$b$ has a unique minimum at $s_0$, which is non-degenerate.
\end{assumption}
\begin{proposition}
	The function $\mu$ has a unique minimum. This minimum is non-degenerate.
\end{proposition}
\begin{proof}
	From Proposition \ref{prop.B3}, the minimum is uniquely attained at the point $(s_0,\xi_0)=\left(s_0,\sqrt{b'_0}\xi_{\frac{\alpha}{\sqrt{b'_0}}}\right)$. We have just to check the non-degeneracy.
	
	We have
	\[\partial_\xi\mu(s,\xi)=\sqrt{b(s)}\partial_\xi\nu\left(\frac{\alpha}{b(s)^{\frac12}},\frac{\xi}{b(s)^{\frac12}}\right)\,,\]
	and
	\begin{multline*}
	\partial_s\mu(s,\xi)\\
	=b'(s)\nu\left(\frac{\alpha}{b(s)^{\frac12}},\frac{\xi}{b(s)^{\frac12}}\right)-\frac{b'(s)}{2\sqrt{b(s)}}\left(\alpha\partial_\alpha\nu\left(\frac{\alpha}{b(s)^{\frac12}},\frac{\xi}{b(s)^{\frac12}}\right)+\partial_\xi\nu\left(\frac{\alpha}{b(s)^{\frac12}},\frac{\xi}{b(s)^{\frac12}}\right)\right)\,.
	\end{multline*}
	At a critical point $(s,\xi)$, we have
	\[b'(s)\left(\nu\left(\frac{\alpha}{b(s)^{\frac12}},\frac{\xi}{b(s)^{\frac12}}\right)-\frac{1}{2\sqrt{b(s)}}\alpha\partial_\alpha\nu\left(\frac{\alpha}{b(s)^{\frac12}},\frac{\xi}{b(s)^{\frac12}}\right)\right)=0\,.\]
	At this point, we have\footnote{Recall that, for all $\alpha>0$, $\nu'(\alpha)=\partial_\alpha\nu(\alpha,\xi_\alpha)$.}
	\[\nu\left(\frac{\alpha}{b(s)^{\frac12}},\frac{\xi}{b(s)^{\frac12}}\right)=\nu\left(\frac{\alpha}{b(s)^{\frac12}}\right)\,\quad\partial_\alpha\nu\left(\frac{\alpha}{b(s)^{\frac12}},\frac{\xi}{b(s)^{\frac12}}\right)=\nu'\left(\frac{\alpha}{b(s)^{\frac12}}\right)\,.\]
	Lemma \ref{lem.f} shows that $b'(s)=0$ which is consistent with the fact that $b$ is assumed to be minimal at $s=s_0$.
	
	Let us compute the derivatives of order two at $(s,\xi)=(s_0,\xi_0)$.
	
	We have
	\begin{equation}\label{eq.d2ximu}
	\partial^2_\xi\mu(s_0,\xi_0)=\partial^2_\xi\nu\left(\frac{\alpha}{b(s_0)^{\frac12}},\frac{\xi_0}{b(s_0)^{\frac12}}\right)>0\,,
	\end{equation}
	\[\partial_s\partial_\xi\mu(s_0,\xi_0)=0\,,\]
	\begin{equation}\label{eq.d2smu}
	\partial^2_s\mu(s_0,\xi_0)=b''(s_0)\left(\nu\left(\frac{\alpha}{\sqrt{b(s_0)}}\right)-\frac{\alpha}{2\sqrt{b(s_0)}}\nu'\left(\frac{\alpha}{\sqrt{b(s_0)}}\right)\right)>0\,.
	\end{equation}
	This shows that the minimum is non-degenerate.
\end{proof}
We can prove that the eigenfunctions of $\mathscr{N}_{a,\hbar}$ are microlocalized near $(s_0,\xi_0)$. The localization in space near the minimum of $b$ allows to take $\theta=0$ by using an appropriate local (near $s_0$) change of gauge. Thus $\alpha=a$.

By using the Grushin reduction (and the harmonic approximation of $\mu$), we get the following asymptotic expansion.

\begin{proposition}
	We have, for some $d_0\in\R$ (\emph{a priori} depending on $a$),
	\[\lambda_n(a,h)=b'_0 h\nu\left(\frac{a}{\sqrt{b'_0}}\right)+h^{\frac32}\left((n-\frac12)\sqrt{\partial^2_s\mu(s_0,\xi_0)\partial^2_\xi\mu(s_0,\xi_0)}+d_0\right)+\mathscr{O}(h^2)\,,\]	
	where the second order derivatives are given in \eqref{eq.d2smu} and \eqref{eq.d2ximu}.
\end{proposition}
From this, we can deduce the following asympotic expansion of the negative eigenvalues.
\[\lambda^-_n(h)=a_0 h^{\frac12}\sqrt{b'_0}+\frac{h}{2a_0}\left((n-\frac12)\sqrt{\partial^2_s\mu(s_0,\xi_0)\partial^2_\xi\mu(s_0,\xi_0)}+\tilde d_0\right)+o(h)\,,\]
with $\tilde d_0\in\R$ and where $\alpha$ has to be replaced by $a_0$ in the expression of the second order derivatives.

\subsection{Case of interior localization}
When $a_0\sqrt{b'_0}>\sqrt{2b_0}$, and when $a$ is close enough to $\sqrt{2b_0}$, we can prove that the first eigenfunctions of $\mathscr{L}_{a,h}$ are exponentially localized near the set $\{B=b_0\}$. In this case, the boundary can essentially be forgotten and the model operator is the electro-magnetic Schrödinger operator
\[(-ih\nabla-\mathbf{A})^2+hB(x)\,.\]
Here $\mathbf{A}$ and $B$ are extended to $\R^2$ in such a way that the minimal level set of $B$ is not changed. Modulo $\mathscr{O}(h^\infty)$, this operator governs the spectral asymptotics of $\mathscr{L}_{a,h}$.
The spectral analysis of such an operator can be done by means of various methods, one of them being the Birkhoff normal form, see \cite{RVN15} and \cite[Section 4]{HKR20} where the \emph{electro-}magnetic case is tackled. In the generic case when $B$ has a unique and non-degenerate minimum at $x_0\in\Omega$, we have
\[\lambda_n(a,h)=2b_0 h+h^2\left(c_0(2n-1)+c_1\right)+\mathscr{O}(h^3)\,,\]
where $c_1\in\R$ and 
\[c_0=\frac{\sqrt{\det(\mathrm{Hess}_{x_0} B)}}{B(x_0)}\,.\]
From this, we deduce that
\[\lambda_n^-(h)=h^{\frac12}\sqrt{2b_0}+h^{\frac32}\left(\frac{c_0}{2\sqrt{2b_0}}(2n-1)+\frac{c_1}{2\sqrt{2b_0}}\right)+\mathscr{O}(h^{\frac 52})\,.\]

\section{Proof of Lemma \ref{lem.Hardy}}\label{sec.prooflemmahardy}
We use Remark \ref{rem.Hardy-disc} to consider the case when $\Omega=\mathbb{D}$. 
We let
\[\ell^2_w(\mathbb{N})=\left\{u\in\ell^2(\mathbb{N}) : \sum_{n\geq 0}(n+1)^{-1}|u_n|^2<+\infty\right\}\,.\] 
Thanks to the isomorphism expressed in \eqref{eq.L2D}, $(\mathscr{H}^2(\mathbb{D}),\langle\cdot,\cdot\rangle_{\partial\mathbb{D}})$ is a Hilbert space.

Consider 
\[K=\left\{u\in\ell^2(\mathbb{N}) : \sum_{n\geq 0}|u_n|^2\leq 1\right\}\,.\]
It is sufficient to show that $K$ is precompact in $\ell^2_w(\mathbb{N})$. Let $\varepsilon>0$. There exists $N\in\mathbb{N}$ such that, for all $u\in K$,
\[\sum_{n\geq N+1}\frac{1}{n+1}|u_n|^2\leq \frac{\varepsilon^2}{4}\,.\]
Moreover, the unit ball of $\mathbb{C}^{N+1}$ for the standard $\ell^2$-norm is precompact, and we can write
\[\exists (a_0,\ldots, a_{M})\in \mathbb{C}^{N+1}\,,\quad  B_{N+1}(0,1)\subset \bigcup_{j=0}^M B_{N+1,w}\left(a_j,\frac\varepsilon 2\right)\,,\]
where $B_{N+1,w}$ are the balls for the $\ell^2_w$-norm.
We have 
\[K\subset \bigcup_{j=0}^M B_w\left(\underline{a_j},\varepsilon \right)\,,\]
where $\underline{a_j}$ denotes the extension by zero of the finite sequence $a_j$. Indeed, there exists $N\in\mathbb{N}$ such that, for all $u\in K$,
\[\left\|u-\sum_{j=0}^N u_j e_j\right\|_{\ell^2_w}\leq \frac \varepsilon 2\,.\]
Then, $\sum_{j=0}^N u_j e_j\in B_{N+1}(0,1)$, and the conclusion follows from the triangle inequality. Here, $(e_j)_{j\geq 0}$ is the canonical basis of $l^2(\mathbb{N})$. 

\section{Holomorphic tubular coordinates}\label{app.holo}
The aim of this section is to define an appropriate system of coordinates $x=\varphi(s,t)$ near $\partial\Omega$. We want $\varphi$ to be holomorphic, and that $s\mapsto\varphi(s,0)=\gamma(s)$ is a counterclockwise parametrization of the boundary by arc length, \emph{i.e.}, $|\gamma'(s)|=1$.

\subsection{Definition of the coordinates}
\begin{hyp}\label{hyp.analybound}
	The boundary $\partial \Omega$ is an analytic curve, \emph{i.e.} there exist $\rho>1$ and an analytic and injective function 
	\[
	g\colon \{\rho^{-1}<|z|<\rho\}\to\mathbb{R}^2\,,
	\]
	such that $g_{\upharpoonright\{|z|=1\}}$ is a regular parametrization of $\partial \Omega$.
\end{hyp}
\begin{proposition}\label{prop.analyticcoor}
	Under Assumption \ref{hyp.analybound}, there exist $\delta_0>0$ and a function
	\[\begin{array}{lclc}
	\varphi\colon &\mathbb{R}/(|\partial \Omega|\mathbb{Z})\times (-\delta_0,\delta_0)&\longrightarrow &\mathbb{R}^2\,,\\
	&w = (s,t)&\longmapsto&\varphi(s,t) = x\,,
	\end{array}
	\]
	such that
	\begin{enumerate}[\rm (i)]
		\item $\varphi$ is holomorphic \emph{i.e.} $\partial_{\overline{w}}\varphi = 1/2(\partial_s+i\partial_t)\varphi  = 0$,
		\item $\varphi_{\upharpoonright\{t=0\}}=:\gamma$ is a positively oriented parametrization by arc length of $\partial \Omega$,
		\item $\varphi$ is injective and $2\geq|\partial_{w}\varphi(s,t)|\geq 1/2$ for all $(s,t)\in\mathbb{R}/(|\partial \Omega|\mathbb{Z})\times (-\delta_0,\delta_0)$,
		\item $\varphi_{\upharpoonright\{t>0\}}\subset \Omega$ and $\varphi$ induces a parametrization of a neighborhood $\mathcal{V}$ of the boundary.
		\item For all $x\in\mathcal{V}$, 
		\begin{equation}\label{eq.td}
		\frac12 t(x)\leq \mathrm{dist}(x,\partial\Omega)\leq 2 t(x)\,.
		\end{equation}
	\end{enumerate}
\end{proposition}
\subsection{Proof of Proposition \ref{prop.analyticcoor}}
We will define $\varphi$ as the flow of a gradient. The following lemma will be crucial.
\begin{lemma}
There exists an open set $U\subset \mathbb{R}^2$ and a regular function $\zeta\colon U\to \mathbb{R}$ such that $\partial \Omega\subset U$ and
\begin{equation}\label{eq.zetafunctin}
\left\{\begin{array}{ll}
\Delta \zeta &= 0\,\text{ on }U\,,\\
\zeta &=0\,\text { on }\partial \Omega\,,\\
\partial_{\bf n} \zeta &= -1\,\text { on }\partial \Omega\,.
\end{array}\right.
\end{equation}
\end{lemma}
\begin{proof}
By \cite[Proposition 3.1]{P92} and Assumption \ref{hyp.analybound}, there exist $r_0>1$ and a biholomorphism $F\colon \mathbb{D}(0,r_0)\to F(\mathbb{D}(0,r_0)) $ such that $F(\mathbb{D}) = \Omega$. Denoting $\tilde \zeta = \zeta\circ F$, the problem of the existence of $(U,\zeta)$ such that \eqref{eq.zetafunctin} holds is equivalent to finding an open set $\tilde U\supset \partial \mathbb{D}$ and a function $\zeta\colon\tilde U \to \mathbb{R}$ such that
\begin{equation}\label{eq.zetafunctinD}
\left\{\begin{array}{ll}
\Delta \tilde \zeta &= 0\,\text{ on }\tilde U\,,\\
\tilde\zeta &=0\,\text { on }\partial \mathbb{D}\,,\\
\partial_{\bf n} \tilde\zeta &= -|F'|\,\text { on }\partial \mathbb{D}\,.
\end{array}\right.
\end{equation}
Since, the function $F'$ does not vanish on $\mathbb{D}(0,r_0)$, there exists a holomorphic function $G$ on $\mathbb{D}(0,r_0)$ such that $G^2 = F'$ and $G\overline G = |F'|$. The function $z\mapsto G(z)\overline{G(\overline{z}^{-1})}$ is holomorphic on $\{r_0^{-1}<|z|<r_0\}$ and coincides with  $|F'|$ on $\partial \mathbb{D}$. We can write for $r_0^{-1}<|z|<r_0$,
\[
G(z)\overline{G(\overline{z}^{-1})} = \sum_{j\in\mathbb{Z}}a_jz^j\,.
\]
Since $|F'|\in \mathbb{R}$, we have that $a_j = \overline{a_{-j}}$ for all $j$ and the radius of convergence $r_1 = \limsup_{j\to+\infty}|a_j|^{1/j}$ of $\sum_{j\geq 0}a_jz^j$ satisfies $r_1>r_0$.
For $z = re^{is}$ such that $r_0^{-1}<r<r_0$, we define
\[\begin{split}
\tilde\zeta(z) 
&= -a_0\log(r) - \sum_{j\in\mathbb{Z}\setminus\{0\}}
a_j\left(
\frac{r^j}{2j} + \frac{r^{-j}}{-2j}
\right)
e^{ijs} 
\\&= -a_0\log(|z|) -\sum_{j\in\mathbb{Z}\setminus\{0\}}
a_j\left(
\frac{z^j}{2j} + \frac{\overline{z^{-j}}}{-2j}
\right)
\\&=
-a_0\log(|z|) -2\Re\sum_{j\geq 1}
a_j\left(
\frac{z^j}{2j} + \frac{\overline{z^{-j}}}{-2j}
\right)
\,,
\end{split}\]
and the conclusion follows.
\end{proof}

We can now prove Proposition \ref{prop.analyticcoor}.

Let $\gamma\colon \mathbb{R}/(|\partial \Omega|\mathbb{Z})\longrightarrow\partial\Omega$ be a positively oriented parametrization by arc length of $\partial \Omega$.
We define the function $\varphi$ as the solution of the following Cauchy problem
\begin{equation}\label{eq.defcoordana}
\left\{
\begin{array}{ll}
\partial_t\varphi(s,t) &= \frac{\nabla \zeta}{|\nabla \zeta|^2}\circ\varphi(s,t)\,,\\
\varphi(s,0) &= \gamma(s)\,.
\end{array}
\right.
\end{equation}
By Cauchy-Lipschitz Theorem, there exists $\delta_0>0$ such that the function $\varphi\colon\mathbb{R}/(|\partial \Omega|\mathbb{Z})\times (-\delta_0,\delta_0)\longrightarrow U$ is well-defined, regular and injective.

By \eqref{eq.defcoordana}, we have for all $(s,t)$ that
$
\partial_t \zeta\circ\varphi(s,t) = 1
$.
By \eqref{eq.zetafunctin}, we deduce that $\zeta\circ\varphi(s,t) = t$ and $0 = \partial_s(\zeta\circ\varphi)(s,t) = \nabla\zeta(\varphi(s,t))\cdot\partial_s\varphi(s,t)$ so that
\[
\partial_t\varphi(s,t)\cdot\partial_s\varphi(s,t) = 0\,.
\]
Therefore, there exists a regular function $\alpha$ such that
\[
\partial_s\varphi(s,t) = \alpha(s,t)J\partial_t\varphi(s,t)\,,
\]
with $J = \begin{pmatrix}0&-1\\1&0\end{pmatrix}$. By \eqref{eq.zetafunctin}, we have
\[
\partial_s\varphi(s,0) = \gamma'(s)\,,\quad \partial_t\varphi(s,0) = -{\bf n}(s)\,,
\]
so that $\alpha(s,0) = -1$. We also have  $\alpha(s,t) = \partial_s\varphi(s,t)\cdot J\nabla \zeta(\varphi(s,t))$ and
\[
\partial_t\alpha(s,t) 
= \partial^2_{st}\varphi(s,t)\cdot J\nabla \zeta(\varphi(s,t)) 
+ \partial_s\varphi(s,t)\cdot J{\rm Hess}\,\zeta(\varphi(s,t))\partial_t\varphi(s,t)\,.
\]
Notice that
\[\begin{split}
\partial^2_{st}\varphi 
&= 
\left(
\frac{{\rm Hess}\,\zeta}{|\nabla \zeta|^2}
-2\frac{\nabla\zeta(\nabla\zeta)^T{\rm Hess}\,\zeta}{|\nabla \zeta|^4}
\right)\partial_s\varphi\,,
\end{split}\]
and using the fact that $(\nabla\zeta)^TJ\nabla \zeta = 0$,
\[\begin{split}
\partial^2_{st}\varphi\cdot J\nabla \zeta
&=
\partial_s\varphi\cdot\left(\frac{{\rm Hess}\,\zeta}{|\nabla \zeta|^2}
-2\frac{{\rm Hess}\,\zeta\nabla\zeta(\nabla\zeta)^T}{|\nabla \zeta|^4}\right)J\nabla \zeta 
\\&= \partial_s\varphi\cdot\left(\frac{{\rm Hess}\,\zeta}{|\nabla \zeta|^2}
\right)J\nabla \zeta\,.
\end{split}\]
By \eqref{eq.zetafunctin}, we conclude that
\[
\partial_t\alpha
= 
\frac{\partial_s\varphi}{|\nabla \zeta|^2}
\cdot
\left(
{\rm Hess}\,\zeta J
+ 
J{\rm Hess}\,\zeta
\right)\nabla \zeta
=
\frac{\partial_s\varphi}{|\nabla \zeta|^2}
\cdot
\left(
J\Delta\zeta
\right)\nabla \zeta
=0\,,
\]
and $\partial_s\varphi = -J\partial_t\varphi$. This ends the proof of the proposition.

\subsection{Taylor expansions with respect to $t$}
We let $\gamma(s)=\varphi(s,0)$ and we have $|\gamma'(s)|=1$. The curvature $\kappa$ is defined through
\begin{equation}\label{eq.curvature}
\gamma''(s)=-\kappa n(s)\,.
\end{equation}
\begin{lemma}\label{lem.Taylorphi}
	We have
	\[\varphi'(s+it)=(1-\kappa t+\frac{t^2}{2}\kappa^2)\gamma'+\frac{t^2}{2}\kappa' n+o(t^2)\,,\]
	and 
	\[|\varphi'(s+it)|^{2}=1-2\kappa t+2\kappa^2t^2+o(t^2)\,,\quad |\varphi'(s+it)|^{-2}=1+2\kappa t+2\kappa^2t^2+o(t^2)\,.\]
	\end{lemma}

\begin{proof}
	We have
	\[\varphi'(s+it)=\partial_s\varphi(s,t)=\partial_s\varphi(s,0)+t\partial_t\partial_s\varphi(s,0)+\frac{t^2}{2}\partial^2_t\partial_s\varphi(s,0)+o(t^2)\,.\]
	Since
	\[\partial_t\varphi=i\partial_s\varphi\,,\]
	we get
	\[\varphi'(s+it)=\gamma'(s)+it\gamma''(s)-\frac{t^2}{2}\gamma^{(3)}(s)+o(t^2)\,.\]
	By using \eqref{eq.curvature}, we have
	\[\gamma^{(3)}(s)=-\kappa^2\gamma'-\kappa' n\,,\]
	and thus
	\[\varphi'(s+it)=\gamma'-i\kappa tn+\frac{t^2}{2}(\kappa^2\gamma'+\kappa' n)+o(t^2)\,,\]
	and the conclusion follows.
	\end{proof}

\subsection*{Acknowledgment}
The authors thank T. Ourmières-Bonafos for useful discussions. They are also very grateful to the CIRM (and its staff) where this work was initiated. L.~L.T. is supported partially by ANR DYRAQ ANR-17-CE40-0016-01. 
E.S has been partially funded by Fondecyt (Chile) project  \# 118--0355.

\bibliographystyle{abbrv}      
\bibliography{biblioMagneticGraphene}

\end{document}